\documentclass{amsart}

\usepackage{xcolor}

\usepackage{tikz-cd}
\usepackage{adjustbox}

\usepackage{hyperref}
\hypersetup{
  pdftitle = {On the differentials of the Hochschild-Kostant-Rosenberg spectral sequence},
  pdfauthor = {Joshua Mundinger},
  bookmarksdepth=4
}

\usepackage[T1]{fontenc}
\usepackage{enumerate}
\usepackage{enumitem}
\setlist[enumerate,1]{label={\roman*.}}
\usepackage[
  style=alphabetic,
  citestyle=alphabetic,
  maxnames=100, %list up to 100 authors
  doi= false,
  isbn = false,
  giveninits = true, %intialize all given names
  datamodel = mrnumber %import data model for mathscinet reviews in bibliography
]{biblatex}

\DeclareNameAlias{default}{family-given} %put family name first
\DeclareFieldFormat{mrnumber}{%
  MR\addcolon\space
    {
      \href{http://www.ams.org/mathscinet-getitem?mr=#1}
    {\nolinkurl{#1}}
    }
}
\usepackage{xpatch}
\xapptobibmacro{finentry}{\setunit{\par} \printfield{mrnumber}}{}{}
\usepackage{amsthm}

\theoremstyle{plain}
\newtheorem{theorem}{Theorem}[subsection]
\newtheorem{lemma}[theorem]{Lemma}
\newtheorem{proposition}[theorem]{Proposition}
\newtheorem{corollary}[theorem]{Corollary}

\newtheorem{theoremalpha}{Theorem}

\theoremstyle{definition}
\newtheorem{conjecture}[theorem]{Conjecture}

\newtheorem{definition}[theorem]{Definition}
\newtheorem{example}[theorem]{Example}
\newtheorem{question}[theorem]{Question}
\newtheorem{remark}[theorem]{Remark}

\usepackage{amsmath,amssymb}

\newcommand{\Einfty}{{\mathbb{E}_\infty}}
\newcommand{\Fp}{\mathbb{F}_p}
\newcommand{\Fpbar}{\bar{\mathbb{F}}_p}
\newcommand{\g}{\mathfrak{g}}

\newcommand{\OO}{\mathcal{O}}

\newcommand{\Sp}{\mathrm{Sp}}
\newcommand{\V}{\mathbf{V}}
\newcommand{\Z}{\mathbb{Z}}

\newcommand{\cX}{\mathcal{X}}
\newcommand{\cY}{\mathcal{Y}}

\newcommand{\filS}{\sone_{\mathrm{fil}}}

\newcommand{\Gacheck}{\Ghat_a^\vee}
\newcommand{\Gahat}{\Ghat_a}
\newcommand{\Ghat}{\hat{\mathbb{G}}}
\newcommand{\Glhat}{\Ghat_\lambda}
\newcommand{\sone}{\mathbf{S}^1}
\newcommand{\Lfil}{\mathcal{L}_{\mathrm{fil}}}

\newcommand{\Ga}{\mathbb{G}_a}
\newcommand{\Gm}{\mathbb{G}_m}

\DeclareMathOperator{\ad}{ad}
\DeclareMathOperator{\Ass}{Ass}
\DeclareMathOperator{\Aut}{Aut}

\DeclareMathOperator{\Bock}{Bock}
\DeclareMathOperator{\ch}{char}
\DeclareMathOperator{\cofib}{cofib}

\DeclareMathOperator{\Der}{Der}
\DeclareMathOperator{\End}{End}
\DeclareMathOperator{\Ext}{Ext}
\DeclareMathOperator{\fib}{fiber}
\DeclareMathOperator{\gr}{gr}
\DeclareMathOperator{\Hom}{Hom}
\DeclareMathOperator{\Homs}{\underline{Hom}}
\DeclareMathOperator{\im}{im}

\DeclareMathOperator{\Isom}{Isom}
\DeclareMathOperator{\Lie}{Lie}

\DeclareMathOperator{\Maps}{Maps}
\DeclareMathOperator{\Mapsu}{\underline{\Maps}}
\DeclareMathOperator{\Mod}{Mod}%category of modules

\DeclareMathOperator{\QCoh}{QCoh}
\DeclareMathOperator{\Spec}{Spec}
\DeclareMathOperator{\Specu}{\underline{Spec}}
\DeclareMathOperator{\Spect}{\widetilde{Spec}}
\DeclareMathOperator{\Spf}{Spf}
\DeclareMathOperator{\Sym}{Sym}
\DeclareMathOperator{\Vect}{Vect}

\DeclareMathOperator{\colimnolimits}{colim}
\newcommand{\colim}[1]{\underset{#1}{\colimnolimits}}

\newcommand{\id}{\mathrm{id}}

\newcommand{\red}[1]{{}^{\mathrm{red}}#1}

\newcommand{\HH}{\mathrm{HH}}
\newcommand{\HHs}{\underline{\HH}}
\newcommand{\HKR}{\mathrm{HKR}}
\newcommand{\fhkr}{F_{\HKR}}

\newcommand{\HHfil}{\HH_{\mathrm{fil}}}
\newcommand{\dR}{\mathrm{dR}}

\newcommand{\pop}{ {[p]} }

\DeclareMathOperator{\aff}{{aff}}

\newcommand{\lotimes}[1]{\underset{#1}{\overset{L}{\otimes}}}

\newcommand{\Xtilde}{\tilde{X}}

\newcommand{\Spc}{\mathrm{Spc}}   %spaces
\newcommand{\CAlg}{\mathrm{CAlg}}   %animated rings
\newcommand{\DAlg}{\mathrm{DAlg}}  %derived rings
\newcommand{\modc}{\text{-}\mathrm{mod}} %modules
\DeclareMathOperator{\PreStk}{PreStk}   %prestacks

\newcommand{\Excast}{\mathrm{Exc}_\ast}  %excisive reduced functors in the sense of Lurie
\newcommand{\can}{\mathsf{can}} %canonical action of LX 

\DeclareMathOperator{\FuntcA}{Fun^{\otimes, \leq 0}_A}

\DeclareMathOperator{\FunsthA}{\underline{Fun}^{\Theta}_A}
\DeclareMathOperator{\FunsthcA}{\underline{Fun}^{\Theta,\leq 0}_A}

\DeclareMathOperator{\FunthcA}{Fun^{\Theta, \leq 0}_A}

\newcommand{\Cat}{\mathsf{Cat}}
\newcommand{\Fun}{\mathsf{Fun}}
\newcommand{\thetaCat}{\Cat^\Theta}
\newcommand{\thetaAlg}{\Theta\mathrm{Alg}}

\newcommand{\DD}{\mathfrak{D}}
\newcommand{\dd}{\mathfrak{d}}

\newcommand{\dact}{\mathrm{d}act}

\newcommand{\series}[1]{[\![#1]\!]}
\title
[
  On the differentials of the HKR spectral sequence
]
{
  On the differentials of the Hochschild-Kostant-Rosenberg spectral sequence
}
\author{Joshua Mundinger}
\address{University of Wisconsin-Madison\\ Madison\\ WI}
\email{jmundinger@wisc.edu}

\date{March 11, 2026}

\begin{document}

\begin{abstract}
  The Hochschild-Kostant-Rosenberg theorem implies the existence of a spectral sequence computing the Hochschild homology of a variety in terms of the cohomology of differential forms.
  When the base field \(k\) has characteristic \(p>0\),
  we show that the differentials in this spectral sequence are zero before page \(p\); when the variety admits a lift to \(W_2(k)\), we give a formula for the differential on page \(p\). The formula involves the Bockstein associated to the lift and a \(p\)th power operation for the Atiyah class.
  Along the way, we also discuss rudiments of Tannakian reconstruction for derived stacks using the $\Theta$-categories of Nuiten and Toën.
\end{abstract}

\maketitle
\vspace{-2em}
\setcounter{tocdepth}{1} %part,chapters,sections
\tableofcontents
%\vspace{-2em}

\section{Introduction}

The Hochschild-Kostant-Rosenberg theorem describes the Hochschild homology of a ring through differential forms. For an algebraic variety, a spectral sequence controls whether this description extends to a global decomposition of Hochschild homology; this spectral sequence is nondegenerate only in positive characteristic. The goal of this work is to give formulas for differentials in this spectral sequence.

To fix ideas, let \(k\) be a field and \(R\) be a commutative \(k\)-algebra. The Hochschild homology of \(R\) is by definition
\[ \HH(R/k) = R \lotimes{R\underset{k}{\otimes} R} R.\]
The Hochschild homology of an algebraic variety over \(k\) is defined similarly.
The Hochschild-Kostant-Rosenberg theorem, going back to \cite{HKR62}, states that if \(R\) is a smooth \(k\)-algebra, then there are natural isomorphisms \(\HH_i(R/k) \cong \Omega^i_{R/k}\) for all \(i\).
This implies that for a smooth algebraic variety \(X/k\), there is a spectral sequence
\begin{align} \label{equation: hkr spectral sequence}
  E_2^{s,t} = H^s(X,\Omega^{-t}_{X/k}) \implies \HH_{-s-t}(X/k),
\end{align}
called the \emph{Hochschild-Kostant-Rosenberg spectral sequence} (HKR spectral sequence). If \(\dim X < \ch k\), then the HKR spectral sequence canonically degenerates and yields a direct sum decomposition \(\HH_i(X/k) \cong \oplus_s H^s(X, \Omega^{i+s}_{X/k})\), known as the Hochschild-Kostant-Rosenberg decomposition \cite{Swa96,Yek02}. One may wonder whether or not such a decomposition holds in all characteristics; this was unknown until 2019, when Antieau, Bhatt, and Mathew produced for each prime \(p\) a \(2p\)-dimensional smooth projective variety \(X/\Fpbar\) such that the HKR spectral sequence has a nonzero differential on page \(p\) \cite{ABM21}. For this \(X\), it follows that \(\sum_i \dim_{\Fpbar} \HH_i(X/\Fpbar) < \sum_{s,t} \dim_{\Fpbar} H^s(X,\Omega^{-t}_{X/\Fpbar})\), and thus a functorial Hochschild-Kostant-Rosenberg decomposition is impossible.

The goal of this paper is to study the failure of the Hochschild-Kostant-Rosenberg decomposition. To do so, we analyze the differentials of the HKR spectral sequence. First, one can ask which differentials \(d_r\) are identically zero. If \(r\) is least such that \(d_r\) is not identically zero, then \(E_2 \cong E_r\) as groups, and one can ask for a formula for \(d_r\) in terms of the groups on the \(E_2\)-page.
%This paper answers both of these questions:

\begin{theoremalpha}\label{maintheorem: differential}
  Let \(k\) be a perfect field of characteristic \(p > 0\) and \(X/k\) be a smooth variety.
  Consider the Hochschild-Kostant-Rosenberg spectral sequence 
  \[ E_2^{s,t} = H^s(X, \Omega^{-t}_{X/k}) \implies \HH_{-s-t}(X/k).\]
  \begin{enumerate}
    \item The differential \(d_r\) is zero if \(r < p\).
    \item There exists a natural map \(V: \Omega^1_{X/k}[1] \to \Omega^p_{X/k}[p]\) in the derived category of quasi-coherent sheaves on \(X\) such that if \(X\) admits a lift \(\Xtilde / W_2(k)\), then the differential \(d_p\) is induced by
      \[
        [V,\Bock_{\Xtilde}]: \Omega^1_{X/k} \to \Omega^p_{X/k}[p],
      \]
      the commutator of \(V\) with the Bockstein \(\Bock_{\Xtilde}\) associated to \(\Xtilde\).
  \end{enumerate}
\end{theoremalpha}
Note that a map \(\Omega^1_{X/k} \to \Omega^p_{X/k}[p]\) induces a derivation \(\Omega^{\bullet}_{X/k} \to \Omega^{\bullet+p-1}_{X/k}[p]\) of \(\OO_X\)-algebras, and in this way \([V,\Bock_{\Xtilde}]\) induces the differential on page \(p\).

The Bockstein in Theorem \ref{maintheorem: differential} shows that the differentials in the HKR spectral sequence are closely related to torsion in the Hodge cohomology of a lift. Antieau and Vezzosi previously observed that if \(X/k\) lifts all the way to \(W(k)\) with \(p\)-torsion-free Hodge cohomology, then the HKR spectral sequence degenerates \cite[Remark 1.6]{AV20}. Theorem \ref{maintheorem: differential} sharpens their observation: the differential \(d_p\) can only be nonzero when \(k\) appears as a \(W(k)\)-summand of the Hodge cohomology of the lift.
Theorem \ref{maintheorem: differential} also recovers Antieau-Bhatt-Mathew's example of a nondegenerate HKR spectral sequence (Example \ref{example: mup}).

While the appearance of a Bockstein in Theorem \ref{maintheorem: differential} is already interesting, applying Theorem \ref{maintheorem: differential} requires computing \(V\). The map \(V\) is closely tied to the Atiyah class \(at: \Omega^1_{X/k}[1] \to \Omega^1_{X/k}[1] \otimes \Omega^1_{X/k}[1]\). Kapranov observed that the Atiyah class makes the shifted tangent bundle \(T_{X/k}[-1]\) of \(X\) into a Lie algebra in the derived category of \(X\), and that every \(E \in \QCoh(X)\) is a module over \(T_{X/k}[-1]\) via the Atiyah class \(at_E: E \to E \otimes \Omega^1_{X/k}[1]\) \cite{Kap99}. Markarian later explained that in characteristic 0, the Hochschild cohomology of \(X\) is the enveloping algebra of \(T_{X/k}[-1]\), and Markarian used this isomorphism to explain the relationship between the Duflo isomorphism and the Grothendieck-Riemann-Roch theorem \cite{Mar09}. For our purposes, it is better to work dually with the Lie coalgebra \(\Omega^1_{X/k}[1]\) with the Atiyah cobracket.

\begin{theoremalpha}\label{maintheorem: restricted}
  Let \(k\) be an algebraically closed field of characteristic \(p > 0\) and \(X/k\) be a smooth variety.
  The map \(V: \Omega^1_{X/k}[1] \to \Omega^p_{X/k}[p]\) of Theorem \ref{maintheorem: differential}
  %makes \(\Omega^1_{X/k}[1]\) with the Atiyah cobracket into a restricted Lie coalgebra in the derived category of quasi-coherent sheaves on \(X\).
  is a \(p\)th power map for the Atiyah class on \(\QCoh(X)\), in the sense that for \(E \in \QCoh(X)\),
  \[ (1\otimes V) \circ at_E \sim at_E^p: E \to E\otimes \Omega^p_{X/k}[p].\]
\end{theoremalpha}

In the main body of the text, the existence and properties of \(V\) are established not just for smooth varieties, but rather for a large class of derived stacks. Thus, calculations for smooth varieties can be made by approximation of stacks, following the idea of Antieau, Bhatt, and Mathew. For example, if \(X = BG\), the classifying stack of a group scheme \(G\), then Theorem \ref{maintheorem: restricted} recovers the classical statement that the Lie algebra of \(G\) is a restricted Lie algebra, where the map \(V\) recovers the \(p\)-operation. The example of Antieau-Bhatt-Mathew is when $G = \mu_p$.
%, and in this way Theorems \ref{maintheorem: differential} and \ref{maintheorem: restricted} recover their example without ``bare-hands'' calculations.

The main input to Theorems \ref{maintheorem: differential} and \ref{maintheorem: restricted} is the filtered circle of Moulinos, Robalo, and Toën \cite{MRT22} and Raksit \cite{Rak20}. Recall that the Hochschild homology of a variety \(X\) can be described as the ring of global functions on the derived loop space
\begin{align*}
  \mathcal{L} X = X {\underset{{X \times X}}{\times}}X = \Mapsu(\sone,X).
\end{align*}
The idea of Moulinos-Robalo-Toën is to realize the filtration inducing the Hochschild-Kostant-Rosenberg spectral sequence universally as a filtration on an approximation of the circle \(\sone\). By definition, the \emph{filtered circle} \(\filS\) is the classifying stack of the Cartier dual of a formal group \(\Glhat\) over \(k[\lambda]\) with group law
\begin{equation*}
  v,w \mapsto v + w + \lambda v w .
\end{equation*}
If \(X\) is an affine \(k\)-scheme, then the ring of global functions on \(\Mapsu(\filS, X)\) is the Rees construction on the Hochschild homology of \(X\) with respect to the Hochschild-Kostant-Rosenberg filtration. Thus, calculations for the HKR spectral sequence follow from calculations on the filtered circle. We view \(\filS\) as a deformation of the special fiber \(B\Ghat_a^\vee\). This deformation to first order determines a derived vector field on \(B\Ghat_a^\vee\), which in turn induces a derived vector field on every mapping stack \(\Mapsu(B\Ghat_a^\vee, X)\), which finally induces the differential \(d_p\). The proof of Theorem \ref{maintheorem: differential} involves analyzing this derived vector field. Theorem \ref{maintheorem: restricted} involves relating \(\Mapsu(B\Ghat_a,X)\) and the Atiyah class of \(X\).
The name of the map \(V: \Omega^1[1] \to \Omega^p[p]\) is to suggest that \(V\) is the Verschiebung on the Lie algebra of the groupoid \(\mathcal{L} X \to X\).

To work with the mapping stack $\Mapsu(\filS, X)$, in §\ref{subsection: tannakian} we introduce rudiments of Tannakian reconstruction for derived stacks. In the context of spectral algebraic geometry, one can often recover a map $f: \cX \to \cY$ of stacks from the data of the symmetric monoidal functor $f^*: \QCoh(\cY)^{\otimes} \to \QCoh(\cX)^\otimes$. In the derived context, one needs more data to account for the difference between $\Einfty$-rings and animated rings. To this end, Nuiten and Toën recently introduced an appropriate upgrade of the structure of a symmetric monoidal category called a $\Theta$-category \cite{NT25}, which keeps track of the derived symmetric algebra monad. In this paper, we introduce a notion of a derived stack being Tannakian in the sense of \cite{BH17}, and use it to identify 
\[ T[-1]\cX \simeq \Mapsu(B\Ga^\vee, \cX)\]
when $\cX$ is a Tannakian derived stack (Theorem \ref{theorem: general tangent stack as mapping stack}).

The statement of Theorem \ref{maintheorem: differential} is similar to Petrov's recent result on the Hodge-de Rham spectral sequence \cite{Pet23}. Deligne and Illusie proved that if \(X/k\) admits a lift to \(W_2(k)\) and \(\dim X < p\), then the Hodge-de Rham spectral sequence degenerates \cite{DI87}. Later, Achinger and Suh proved that if \(X/k\) admits a lift to \(W_2(k)\), then the differentials of the conjugate Hodge-de Rham spectral sequence are zero before page \(p\) \cite{AS23}. In this situation, Petrov provided a formula for the differential on page \(p\) of the conjugate spectral sequence in terms of the Bockstein associated to the lift.
Petrov's methods differ from those of this paper, and the relationship between the two results is unclear.
It is known that the Hodge-to-de Rham and HKR spectral sequences fit into a ``commutative square'' of spectral sequences
% https://q.uiver.app/#q=WzAsNCxbMCwxLCJIXlxcYnVsbGV0KFgsTFxcT21lZ2FeXFxidWxsZXRfe1gva30pIl0sWzEsMCwiSF5cXGJ1bGxldF97ZFJ9KFgvaykiXSxbMSwyLCJcXG1hdGhybXtISH1fXFxidWxsZXQoWC9rKSJdLFsyLDEsIlxcbWF0aHJte0hQfV9cXGJ1bGxldChYL2spIl0sWzAsMiwiXFx0ZXh0e0hLUn0iLDIseyJzdHlsZSI6eyJib2R5Ijp7Im5hbWUiOiJzcXVpZ2dseSJ9fX1dLFsxLDMsIlxcdGV4dHtkZSBSaGFtLUhQfSIsMCx7InN0eWxlIjp7ImJvZHkiOnsibmFtZSI6InNxdWlnZ2x5In19fV0sWzIsMywiXFx0ZXh0e1RhdGV9IiwyLHsic3R5bGUiOnsiYm9keSI6eyJuYW1lIjoic3F1aWdnbHkifX19XSxbMCwxLCJcXHRleHR7SG9kZ2UtZGUgUmhhbX0iLDAseyJzdHlsZSI6eyJib2R5Ijp7Im5hbWUiOiJzcXVpZ2dseSJ9fX1dXQ==
\[
  \begin{tikzcd}[ampersand replacement=\&]
    \& {H^\bullet_{dR}(X/k)} \\
    {H^\bullet(X,L\Omega^\bullet_{X/k})} \&\& {\mathrm{HP}_\bullet(X/k)} \\
    \& {\mathrm{HH}_\bullet(X/k)}
    \arrow["{\text{de Rham-HP}}", dashed, from=1-2, to=2-3]
    \arrow["{\text{Hodge-de Rham}}", dashed, from=2-1, to=1-2]
    \arrow["{\text{HKR}}"', dashed, from=2-1, to=3-2]
    \arrow["{\text{Tate}}"', dashed, from=3-2, to=2-3]
  \end{tikzcd}
\]
(see \cite[Remark 3.6]{ABM21} for definitions and discussion).
However, Petrov's result is about the conjugate spectral sequence and does not give formulae for differentials in the Hodge-de Rham spectral sequence.
It would be interesting to know whether the results of this paper can be extended to the other three spectral sequences in this square.

The paper closes in §\ref{section: closing} with more open problems.

\subsection*{Acknowledgments}
The author thanks Victor Ginzburg and Andrei Căldăraru for asking the questions that led to this work. Special thanks to Dima Arinkin for many discussions on derived algebraic geometry, which were essential to the writing of this paper.
The author thanks
Ben Antieau,
Bhargav Bhatt,
Akhil Mathew,
Tasos Moulinos,
Sasha Petrov,
Catherine Ray,
Shubhankar Sahai,
Vadim Vologodsky,
and the anonymous referee
for useful comments and conversations.
Thanks to Sanath Devalapurkar for pointing out an error in an earlier version of this paper.

In the course of this work, the author was supported by the National Science Foundation under Award No. 2503534. Any opinions, findings, and conclusions or recommendations expressed in this material are those of the authors and do not necessarily reflect the views of the National Science Foundation.

\section{Preliminaries}

\subsection{Conventions}

Fix a prime integer \(p > 0\). Throughout, \(A\) is a commutative \(\Z_{(p)}\)-algebra. For the main theorems, it will suffice to work over \(k\) or \(W_2(k)\) where \(k\) is a perfect field of characteristic \(p\).

By category we mean \((\infty,1)\)-category in the sense of Lurie \cite{Lur09}. By 2-category we mean \((\infty,2)\)-category (this notion will only appear in §\ref{subsection: tannakian}). All functors are derived. The category of spaces (animated sets) is denoted \(\Spc\).
We write \([1]\) for the suspension/shift functor.
If \(C\) is a chain complex or module spectrum, we take the conventions \(\pi_i(C) = H_i(C) = H^{-i}(C)\). The shift satisfies \(H_i(C[1]) = H_{i-1}(C)\).

\subsection{Derived algebraic geometry}

We follow the standard setup of derived algebraic geometry, see e.g.\ \cite{Toe14, GR17I}. Let \(\CAlg\)be the category of animated rings. A \emph{derived prestack} is a functor \(\CAlg \to \Spc\). 
A derived prestack is \emph{accessible} if it commutes with sufficiently filtered colimits.
A \emph{derived stack} is an accessible derived prestack which is a sheaf with respect to the fpqc topology. Inside of derived stacks are \emph{derived affine schemes}, those stacks representable by an animated ring. Let \(\Spec R\) denote the derived affine scheme represented by \(R \in \CAlg\). A \emph{derived scheme} is a derived stack which admits a Zariski cover by derived affine schemes. From now on, we drop the word ``derived,'' and additionally assume all prestacks are over \(\Spec A\) unless otherwise noted. We write \(\PreStk_A\) for the category of prestacks over \(\Spec A\), and \(\PreStk^{acc}_A\) for the full subcategory of accessible prestacks.

Given two stacks \(\cX\) and \(\cY\) over a base stack \(\mathcal Z\), the mapping stack \(\Mapsu_{\mathcal Z}(\cX, \cY)\) is the \(\mathcal Z\)-stack with functor of points
\[ 
  \Mapsu_{\mathcal Z}(\cX,\cY)(S) = \Maps_{\mathcal Z}(X \times_{\mathcal Z} \Spec S, \cY),
\]
where we remind that fiber products are taken in \(\PreStk_A\), that is, are derived.

If \(\cX\) is a prestack, then we define the category of \emph{quasi-coherent sheaves} on \(\cX\) by
\[
  \QCoh(\cX) = \lim_{\Spec R \to \cX} \Mod_R
\]
\cite[Chapter 3, §1.1]{GR17I}.
An \(R\)-module \(M\) is \emph{connective} if \(H^{>0}(M) = 0\),  \emph{\(n\)-connective} if \(M[-n]\) is connective, and \emph{almost connective} if \(M\) is \(n\)-connective for some \(n\in \Z\).
By taking the limit, the same concepts extend to quasi-coherent sheaves.
We write \(\QCoh(\cX)^{\leq 0} \subseteq \QCoh(\cX)\) and \(\QCoh(\cX)^- \subseteq \QCoh(\cX)\) for the full subcategories of \( \QCoh(\cX)\) spanned by connective and almost connective objects, respectively.

\subsection{Affine stacks and derived rings}\label{subsection: affine stacks}

The notion of an affine stack goes back to the work of Toën \cite{Toe06} and Lurie \cite{luriedagviii}. In Toën's work, an affine stack is defined on discrete commutative algebras. By definition, if \(C\) is a cosimplicial algebra, then the associated affine stack \(\Spec^\Delta(C)\) is defined by \(\Spec^\Delta(C)(S) = \Maps(C, S)\), where the discrete commutative algebra \(S\) is viewed as a cosimplicial commutative algebra.

The notion of an affine stack may be extended to the derived setting using \emph{derived rings} \(\DAlg_A\) in the sense of A.\ Mathew.
The key observation, going back to Illusie \cite{Ill71}, is that the symmetric algebra monad \(\Sym_A\), defined as a left derived functor on connective \(A\)-modules, extends to a monad on all \(A\)-modules. The category \(\DAlg_A\) of derived \(A\)-algebras is then the category of algebras over the monad \(\Sym_A : \Mod_A \to \Mod_A\). The restriction of \(\Sym_A\) to coconnective \(A\)-modules is exactly the right Kan extension of the symmetric algebra functor, and thus the category of cosimplicial \(A\)-algebras is equivalent to the full subcategory of coconnective derived rings \(\DAlg_A^{\geq 0}\). See \cite[§3]{BM25} or \cite[§4]{Rak20} for details about derived rings.

\begin{definition}
  The functor $\Spect: \DAlg_A \to \PreStk_A^{op}$ is defined by sending $C \in \DAlg_A$ to 
  \[ \Spect(C)(S) = \Maps_{\DAlg_A}(C,S).\]
  A prestack is a \emph{derived affine stack} if it is equivalent to $\Spect(C)$ for some $C \in \DAlg_A$.\footnote{This definition differs slightly from the definition of affine stack in \cite{AM25}, but \cite[Theorem 1.3]{AM25} states that the definitions agree when $C$ is eventually connective.}
  A morphism $\mathcal X \to \mathcal Y$ of prestacks is a \emph{relative derived affine stack} if for all affine schemes $\Spec R$ and $y \in \mathcal Y(\Spec R)$, the pullback $\mathcal X \times_{\mathcal Y,y} \Spec R$ is a derived affine stack.
\end{definition}

By construction, \(\Spect\) sends colimits to limits, so derived affine stacks are closed under limits.
The functor $\Spect$ is the left adjoint of $\mathcal X \mapsto R\Gamma(\cX, \OO_{\cX})$ \cite[2]{MM25}. 

\begin{remark}
  The work of Moulinos-Robalo-Toën on the filtered circle \cite{MRT22} and the original work of Toën \cite{Toe06} consider affine stacks as functors on discrete commutative rings.
  However, if \(C \in \DAlg_A^{\geq 0}\), then \(\Spect(C): \CAlg_A \to \Spc\) is left Kan extended from discrete rings \cite[Corollary 3.7]{MM25}\cite[Proposition 4.4.6]{luriedagviii}. Thus, results from \cite{MRT22} about affine stacks may be directly applied in the derived setting. To keep this paper self-contained, we include some direct proofs of derived analogues of results already known for non-derived affine stacks.
\end{remark}

\begin{example}\label{example: K G_a is affine}
  For \(i \geq 0\), let \(K(\mathbb G_a, i)\) be the prestack over \(\Spec A\) with functor of points
  \[ K(\mathbb G_a, i)(S) = S[i]. \]
  Since \(S[i] \simeq \Hom_{\Mod_A}(A[-i],S)\),
  \[ K(\mathbb G_a,i) \simeq \Spect(\Sym_A(A[-i])),\]
  so \(K(\mathbb G_a,i)\) is a derived affine stack.
\end{example}

%Now we can prove a derived version of Theorem \ref{theorem: filtered circle is affinization}.

Another example of an affine stack is the classifying stack of the Cartier dual of a formal Lie group.
For any formal Lie group \(\Ghat \to \mathcal Z\), let \(\Ghat^\vee \to \mathcal Z\) be the relative Cartier dual of \(\Ghat \to \mathcal Z\); it is an affine group scheme over \(\mathcal Z\).

\begin{theorem}[\cite{MM25}, Lemma 6.2]\label{theorem: B H dual affine stack}
  Let $R$ be a classical commutative ring. If $H \to \Spec R$ is a formal Lie group, then the classifying stack $BH^\vee$ of the Cartier dual of $H$ over $\Spec R$ is a (derived) affine stack.
\end{theorem}

We give another proof of this theorem in the case that $R$ is a $\mathbb{Z}_{(p)}$-algebra. In this case, the computation can be reduced to the calculation that $BW$ is an affine stack where $W$ is the group scheme of $p$-typical Witt vectors.

\begin{proof}[Proof when $R$ is a $\mathbb{Z}_{(p)}$-algebra.]
  Let $W$ be the group scheme of $p$-typical Witt vectors and $\hat{W}$ its Cartier dual.
  Let $\mathbb{E}_p = \End(\hat{W})^{op} = \End(W)$.
  By $p$-typical Cartier theory (see e.g. \cite[Theorem 4.23]{Zin84}), there is a finitely generated left $\mathbb{E}_p$-module $M$ such that $H = \hat{W} \otimes_{\mathbb{E}_p} M$. 
  Thus $H^\vee = \Hom_{\mathbb{E}_p}(M, W)$.
  The claim is Zariski local on $R$, so we can assume that $M/VM$ is a free $R$-module of finite rank $r$.
  By the structure equations for a Cartier module \cite[Theorem 4.39]{Zin84}, there is a short exact sequence 
  \[0 \to \mathbb{E}_p^r \overset{\varphi}{\to} \mathbb{E}_p^r \to M \to 0\]
  where $L$ is a finite free $\mathbb{E}_p$-module and $\varphi = F - A$ where $A$ is a matrix with entries in $W(R)\{V\} \subseteq \mathbb{E}_p$.
  The map $\varphi: W^r \to W^r$ is surjective in fppf topology since $F - [x]: W \to W$ is for all $x \in R$.
  Thus $H^\vee$ fits into a short exact sequence 
  \[ 0 \to H^\vee \to W^r \overset{\varphi}{\to} W^r \to 0\]
  and $BH^\vee$ is the homotopy fiber of $BW^r \overset{\varphi}{\to} BW^r$.
  Since affine stacks are closed under limits, it suffices to show that $BW$ is an affine stack.

  First, we show that \(BW = \varprojlim_m BW_m\). Observe that for an animated ring \(S\),
  \[ \pi_i \Maps(\Spec S, BW_m) = \pi_{i-1} W_m(S)\]
  for \(m \leq \infty\), as \(W_m\) is an iterated extension of \(\mathbb G_a\) and \(\Spec S\) is affine.
  As a simplicial set, \(W_m(S)\) is the \(m\)-fold product of \(S\), and thus the maps \(\pi_{i-1}W_{m+1}(S) \to \pi_{i-1}W_m(S)\) are surjective.
  It follows that \(R^1\varprojlim_m \pi_i \Maps(\Spec S, BW_m) = 0\) for all \(i\), and thus \(\Maps(\Spec S, BW) \simeq \varprojlim_m \Maps(\Spec S, BW_m)\).

  Now \(K(\mathbb G_a,i)\) is an affine stack for all $i \geq 0$ by Example \eqref{example: K G_a is affine}.
  Since \(W_m\) is an iterated extension of \(\mathbb G_a\), we have pullback squares
  % https://q.uiver.app/#q=WzAsNCxbMCwwLCJCV197bSsxfSJdLFswLDEsIkJXX20iXSxbMSwxLCJLKFxcbWF0aGJiIEdfYSwgMikiXSxbMSwwLCJcXGFzdCJdLFswLDNdLFswLDFdLFsxLDJdLFszLDJdLFswLDIsIiIsMSx7InN0eWxlIjp7Im5hbWUiOiJjb3JuZXIifX1dXQ==
  \[
    \begin{tikzcd}
      {BW_{m+1}} & \ast \\
      {BW_m} & {K(\mathbb G_a, 2)}
      \arrow[from=1-1, to=1-2]
      \arrow[from=1-1, to=2-1]
      \arrow["\lrcorner"{anchor=center, pos=0.125}, draw=none, from=1-1, to=2-2]
      \arrow[from=1-2, to=2-2]
      \arrow[from=2-1, to=2-2]
    \end{tikzcd}
  \]
  for all \(m \geq 0\).
  By induction, \(BW_m\) is a limit of affine stacks and hence is an affine stack for all \(m \geq 1\).
\end{proof}

\subsection{Filtrations and spectral sequences}\label{subsection: filtrations}

The Hochschild-Kostant-Rosenberg spectral sequence arises from a filtered complex. Instead of working with only the differentials, our results on the spectral sequence will be formulated in terms of the underlying filtered complex. In this section, the notion of a filtered complex being \emph{split to order \(n\)} is introduced. This notion refines the statement that the first \(n\) differentials in the associated spectral sequence are zero.

The category of filtered complexes can be formulated as sheaves on a stack. Let \(\mathbb A^1 = \Spec A[\lambda]\) where \(\mathbb G_m\) acts by dilating the coordinate \(\lambda\) with weight\footnote{Choosing weight \(-1\) corresponds to viewing our filtrations as descending. In the derived context, the morphisms in a filtration need not be monomorphisms; the theories of increasing and decreasing filtrations are equivalent.} \(-1\). Then the category of quasi-coherent sheaves on the quotient stack \(\mathbb A^1/\mathbb G_m\) is equivalent to the \(\infty\)-category of filtered \(A\)-modules via the Rees construction \cite{Mou21}.
We call a morphism \(\cY \to \mathbb A^1/\mathbb G_m\) a \emph{filtered stack}. Given a filtered stack \(\cY\), write \(\cY^u\), the \emph{underlying stack}, for the pullback along \(\mathbb G_m/\mathbb G_m \to \mathbb A^1/\mathbb G_m\); write \(\cY^{gr}\), the \emph{associated graded}, for the pullback along \(0/\mathbb G_m \to \mathbb A^1/\mathbb G_m\).

\begin{definition} Let
  $Z_n = \Spec(A[\lambda]/(\lambda^{n+1}))/\Gm$
  be the
  \(n\)th order neighborhood of \(0 \in \mathbb A^1/\mathbb G_m\),
  and let
  \(\rho_n: Z_n \to \mathbb A^1/\mathbb G_m\)
  be the inclusion.
  If \(C \in \QCoh(\mathbb A^1/\mathbb G_m)\), then a \emph{splitting to order \(n\)} of \(C\) is an equivalence
  \[
    \sigma: A[\lambda]/\lambda^{n+1}\otimes_A (\rho_0)^\ast C\ \simeq (\rho_n)^\ast C
  \]
  in \(\QCoh(Z_n)\).
\end{definition}

Now we explain the relationship between splittings to order \(n\) and spectral sequences.
Let \(F^\bullet C\) be a complex of \(A\)-modules equipped with a descending filtration \(\cdots \to F^{t+1} C \to F^t C \to \cdots\).
The associated graded is defined by \(\gr^t C = F^t C / F^{t+1}C = \cofib(F^{t+1}C \to F^t C)\).
The complex \(F^\bullet C\) is split to order \(n\) if we are given equivalences
\[ F^t C / F^{t+n+1} C \simeq \bigoplus_{i=0}^n \gr^{t+i}C \]
for all \(t\), together with certain compatibilities between these equivalences.
Being split to order \(0\) is a null statement.
Associated to \(F^\bullet C\) is a spectral sequence with
\[ E_2^{s,t} = H^{s+t} \gr^{-t} C\]
and differentials \(d_r: E_r^{s,t} \to E_r^{s+r,t-r+1}\) of bidegree \((r, 1-r)\)\footnote{See \cite{Ant24} for a discussion of the modern perspective on spectral sequences, as well as a discussion of degree conventions in §7.}.
If \(F^\bullet C\) is split to order \(n\), then in the associated spectral sequence,
the differentials \(d_r\) are zero for \(r \leq n+1\). Further, from the canonical triangle
\[
  \gr^{t+n+1} C \to F^t C / F^{t+n+2} C \to F^{t}C / F^{t+n+1}C \overset{+1}{\to}
\]
the boundary morphism furnishes a morphism
\begin{equation}\label{eq: canonical extension for n split}
  e: \gr^{t} C \to \bigoplus_{i=0}^n \gr^{t+i}C \simeq F^tC / F^{t+n+1}C \to \gr^{t+n+1}C[1],
\end{equation}
whose null-homotopy for all \(t\) is equivalent to an extension of the splitting to order \(n+1\). As \(d_r = 0\) for \(r \leq n+1\), we can identify the \(E_2\) and \(E_{n+2}\)-pages; the differential \(d_{n+2}: H^{s+t}\gr^{-t}C \to H^{s+t+1}\gr^{-t+n+1}C\) in our spectral sequence is induced by the morphism \(e\) from \eqref{eq: canonical extension for n split}.

\begin{remark}\label{remark: e by deformation theory}
  The extension $e$ has an interpretation in terms of deformation theory. Consider the pullback square of animated rings 
  \[
  % https://q.uiver.app/#q=WzAsNCxbMCwwLCJBW1xcbGFtYmRhXS8oXFxsYW1iZGFee24rMn0pIl0sWzAsMSwiQSJdLFsxLDAsIkFbXFxsYW1iZGFdLyhcXGxhbWJkYV57bisxfSkiXSxbMSwxLCJBIFxcb3BsdXMgKFxcbGFtYmRhXntuKzF9KS8oXFxsYW1iZGFee24rMn0pWzFdIl0sWzAsMV0sWzEsM10sWzAsMl0sWzIsM11d
\begin{tikzcd}
	{A[\lambda]/(\lambda^{n+2})} & {A[\lambda]/(\lambda^{n+1})} \\
	A & {A \oplus (\lambda^{n+1})/(\lambda^{n+2})[1]}
	\arrow[from=1-1, to=1-2]
	\arrow[from=1-1, to=2-1]
	\arrow[from=1-2, to=2-2]
	\arrow[from=2-1, to=2-2]
\end{tikzcd}.
\]
  This square induces the square of affine schemes
  \[
  % https://q.uiver.app/#q=WzAsNCxbMSwxLCJaX3tuKzF9Il0sWzAsMSwiWl9uIl0sWzEsMCwiWl8wIl0sWzAsMCwiXFx0aWxkZSBaX3tuKzF9Il0sWzMsMV0sWzMsMl0sWzEsMF0sWzIsMF1d
\begin{tikzcd}
	{\tilde Z_{n+1}} & {Z_0} \\
	{Z_n} & {Z_{n+1}}
	\arrow[from=1-1, to=1-2]
	\arrow[from=1-1, to=2-1]
	\arrow[from=1-2, to=2-2]
	\arrow[from=2-1, to=2-2]
\end{tikzcd}
\]
where $\tilde{Z}_{n+1} = \Spec(A \oplus (\lambda^{n+1})/(\lambda^{n+2})[1])/\Gm$. Taking $\QCoh$, we find a functor
\begin{equation}\label{eq: e by deformation theory}
  \QCoh(Z_{n+1}) \to \QCoh(Z_0) \underset{\QCoh(\tilde{Z}_{n+1})}{\times} \QCoh(Z_n)
\end{equation}
which is fully faithful by Lemma \ref{lemma: modules over pullback}.
Suppose that $C \in \QCoh(Z_{n+1})$ and we are given a splitting of $C|_{Z_n}$. Then if $\gr C = C|_{Z_0}$, then applying \eqref{eq: e by deformation theory} to $C$ 
yields the object $(\gr C, A[\lambda]/(\lambda^{n+1})\otimes_A \gr C, \gamma)$ where $\gamma$ is an isomorphism 
\[ \gamma: (A \oplus (\lambda^{n+1})/(\lambda^{n+2})[1]) \otimes_A \gr C \to (A \oplus (\lambda^{n+1})/(\lambda^{n+2})[1]) \otimes_A \gr C.\]
Hence $\gamma$ is determined by a $\Gm$-equivariant morphism 
\[ \gr C \to \gr C \otimes_A (\lambda^{n+1})/(\lambda^{n+2})[1],\]
which is exactly the morphism $e$ from \eqref{eq: canonical extension for n split}. 

As a consequence of the full faithfulness of \eqref{eq: e by deformation theory}, we find that the space of nullhomotopies of \eqref{eq: canonical extension for n split} is equivalent to the space of extensions of the splitting from order $n$ to order $n+1$.
\end{remark}

\section{Hochschild homology and the filtered circle}\label{section: hh and filtered circle}

\subsection{The Hochschild-Kostant-Rosenberg filtration and the filtered circle}

\begin{definition}
  If \(X\) is a prestack, the \emph{free loop space} on \(X\) is the stack
  \[
    \mathcal{L}X = \Mapsu(\sone,X) = X \underset{X \times X}{\times} X.
  \]
\end{definition}
Here we consider \(\sone\) as the constant prestack assigning to each input the space $\sone$. The identification \(\Mapsu(\sone,X) = X \times_{X \times X} X\) follows from realizing \(\sone\) as the homotopy pushout \(\ast \sqcup \ast \rightrightarrows \ast\) of two points mapping to a point.

The loop space comes equipped with a projection \(\pi: \mathcal{L}X \to X\), which evaluates a loop at the base point.

\begin{definition}
  If \(X\) is a stack with affine diagonal, then the \emph{Hochschild homology sheaf} of \(X\) is
  \[
    \HHs(X/A) = \pi_* \OO_{\mathcal{L} X} \in \QCoh(X).
  \]
\end{definition}
If \(\Delta: X \to X \times X\) is the diagonal and \(\pi_1:X \times X \to X\) is the first projection, then the Hochschild homology sheaf of \(X\) has equivalent descriptions as
\[
  \HHs(X/A) = (\pi_1)_*(\Delta_*\OO_X \otimes \Delta_* \OO_X) = \Delta^*\Delta_*\OO_X.
\]
If \(X\) is a quasi-compact scheme with affine diagonal, then we define the Hochschild homology of \(X\) by
\[
  \HH(X/A) = R\Gamma(X, \HHs(X/A)) = R\Gamma(\mathcal{L} X, \OO_{\mathcal{L}X}).
\]

\begin{definition}[\cite{ABM21}, Definition 3.1(a)]
  \label{definition: hkr filtration}
  If \(X\) is a quasi-compact \(A\)-scheme with affine diagonal, the \emph{Hochschild-Kostant-Rosenberg filtration} (HKR filtration) \(\fhkr \HHs(X/A)\) on \(\HHs(X/A)\) is the descending filtration obtained by left Kan extension of the Postnikov filtration on the Hochschild homology of affine spaces $\mathbb{A}^n$ for all $n$.
\end{definition}

The HKR spectral sequence from the introduction is the spectral sequence associated to the filtered complex \(R\Gamma(X, \fhkr \HHs(X/A))\) \cite[Definition 3.1(a)]{ABM21}.
In this context, the Hochschild-Kostant-Rosenberg theorem identifies the associated graded of the filtered complex and thus the initial page of the spectral sequence:

\begin{lemma}[Hochschild-Kostant-Rosenberg theorem]
  For quasi-compact \(A\)-schemes \(X\) with affine diagonal, there are natural isomorphisms
  \[ \gr^t \HHs(X/A) \cong L\Omega^t_{X/A}[t]\]
 in \(\QCoh(X)\) for all \(t\).
\end{lemma}
\begin{proof}
  Since both sides are Zariski sheaves \cite[Lemma 4.2]{BN12}, it suffices to consider when \(X\) is affine. Both sides are left Kan extensions from smooth morphisms \(X/A\), in which case it follows from the usual Hochschild-Kostant-Rosenberg theorem \cite[3.4.4 Theorem]{Lod92}.
\end{proof}

We now discuss the filtered circle, discovered independently by Moulinos-Robalo-Toën \cite{MRT22} and Raksit \cite{Rak20}; our treatment follows \cite{MRT22}.
Consider the following formal group law over \(A[\lambda]\):
\[ \Glhat: v,w \mapsto  v + w + \lambda vw.\]
As \(\lambda\) has weight \(-1\), this formal group law is homogeneous when \(v\) and \(w\) are given weight \(1\), and so descends to a filtered formal group \(\Glhat \to \mathbb A^1/\mathbb G_m\). The underlying formal group is \(\Ghat_m\) and the associated graded formal group is \(\Ghat_a\).

\begin{definition}\cite[Proposition 6.3.3]{MRT22}
  The \emph{filtered circle} is
  \[ \filS = B\Glhat^\vee,\]
  the classifying stack of the Cartier dual of \(\Glhat \to \mathbb A^1/\mathbb G_m\).
\end{definition}
The filtered circle is a filtered stack, that is, it is equipped with \(\filS \to \mathbb A^1/\mathbb G_m\).
This stack ought to be called a circle because of one of the main calculations of Moulinos, Robalo, and Toën \cite{MRT22}.
If \(\cY\) is a stack, then the \emph{affinization} of \(\cY\) is $\Spect(R\Gamma(\cY,\OO_{\mathcal Y}))$. By definition, there is a map $\cY \to \Spect(R\Gamma(\cY,\OO_{\cY}))$ initial among maps from $\cY$ to affine stacks.
By Theorem \ref{theorem: B H dual affine stack}, $\filS \to \mathbb{A}^1/\mathbb{G}_m$ is a relative affine stack.

\begin{theorem}\cite[Theorem 1.2.4]{MRT22}\label{theorem: filtered circle is affinization}
  \((\filS)^u = B\Ghat_m^\vee\) is the affinization of \(\sone\).
\end{theorem}
The affinization map $\sone = B\mathbb{Z} \to B\Ghat_m^\vee$ is induced by the homomorphism $\mathbb{Z} \to \Ghat_m^\vee$ Cartier dual to the inclusion $\Ghat_m \to \mathbb{G}_m$.

\begin{definition}
  The \emph{filtered loop space} of $X/A$ is the mapping stack 
  \[ \Lfil X = \Mapsu_{{\mathbb{A}^1/\Gm}}(\filS, X).\]
\end{definition}

For any affine $A$-scheme \(\Spec(R)\), Theorem \ref{theorem: filtered circle is affinization} implies that the pullback map 
\[
  (\Lfil \Spec(R))^u = \Mapsu((\filS)^u, \Spec (R)) \to \Mapsu(\sone, \Spec(R)) = \mathcal{L}(\Spec R)
\]
is an equivalence, that is, the underlying object of the filtered loop space is the loop space.
\begin{theorem}\cite[Theorem 1.2.4]{MRT22}\label{theorem: filtered circle and HKR}
  For affine \(A\)-schemes \(\Spec(R)\), there is a natural equivalence
  \[
    %\Mapsu(\filS, \Spec(R)) 
    \Lfil \Spec(R)
    \simeq \Specu_{\mathbb{A}^1/\Gm} \fhkr\HH(\Spec(R)/A)
  \]
  of stacks over \(\mathbb{A}^1/\Gm\).
\end{theorem}
Since \(\gr \HH(R/A) = \bigoplus_t L\Omega^t_{R/A}[t] = \Sym (L\Omega^1_{R/A}[1])\), which is the ring of functions on the shifted tangent bundle \(T[-1]\Spec(R)\),
\begin{corollary}\label{corollary: associated graded is shifted tangent bundle}
  For affine \(A\)-schemes \(\Spec(R)\), there is a natural equivalence of stacks
  \[
    \Mapsu(B\Gahat^\vee, \Spec(R)) \simeq T[-1]\Spec(R).
  \]
\end{corollary}

Theorem \ref{theorem: filtered circle and HKR} implies that we can study the HKR filtration, and thus the HKR spectral sequence, of an affine scheme by studying the filtered stack \(\filS\). As the Hochschild homology sheaf functor \(\HHs(-/A)\) satisfies Zariski descent \cite[Lemma 4.2]{BN12}, this analysis extends to schemes by Zariski descent.

\begin{remark}
  The mapping stack \(\Mapsu(\filS, \cX)\) does not compute the Hochschild homology of a general stack \(\cX\); instead, it computes a completion of the Hochschild homology of \(\cX\), as originally observed in characteristic zero by Ben-Zvi and Nadler \cite{BN12}.
  See Appendix \ref{appendix: loops on stacks} for more details.
  We study the discrepancy between the shifted tangent bundle and ${\Mapsu((\filS)^{gr}, \mathcal X)}$ in §\ref{subsection: shifted tangent}.
\end{remark}

\subsection{The filtered circle as a deformation of its special fiber}
\label{subsection: filtered circle as deformation}

In this section, we study the deformation theory of the filtered circle as a family over \(\mathbb A^1/\mathbb G_m\). The special fiber of \(\filS\) is \(B\Gahat^\vee\), so we can view \(\filS\) as a graded deformation of \(B\Gahat^\vee\). Splittings of the HKR filtration will arise from isomorphisms of $\filS$ with $B\Gahat^\vee$ over substacks of $\mathbb{A}^1/\mathbb{G}_m$.

A full splitting of the HKR filtration, natural in affine derived \(A\)-schemes, is equivalent to an isomorphism of \(\Ghat_a\) with \(\Ghat_m\) \cite{Rob23}. Such an isomorphism is necessarily an exponential map, which does not exist in characteristic \(p\) because of the denominators in the Taylor expansion of the exponential. However, the truncated exponential \(1 + u + u^2/2 + \cdots + u^{p-1}/(p-1)!\) makes sense since \(A\) is a \(\Z_{(p)}\)-algebra.

\begin{theorem}\label{theorem: initial-split}
  There is a splitting to order \(p-2\) of the HKR filtration on \(\HHs(\Spec(R)/A)\), natural in \(\Spec(R)\).
\end{theorem}
\begin{proof}
  Since \((p-1)!\) is invertible in \(A\), we can write the truncated rescaled exponential
  \[
    E_\lambda(u) = \sum_{n=1}^{p-1} \frac{\lambda^{n-1}u^n}{n!} \in A[\lambda]/(\lambda^{p-1})\series{u}.
  \]
  The truncated exponential \(E_\lambda\) defines an isomorphism of \(\rho_{p-2}^*\Glhat\) with \(\Ghat_a\) over \(A[\lambda]/\lambda^{p-1}\).
  This also provides a splitting of \(\filS\) to order \(p-2\), which provides a natural splitting of \(\fhkr\) by Theorem \ref{theorem: filtered circle and HKR}.
\end{proof}

The remainder of this section is the calculation of the deformation class of $\filS$ at order $p-1$, that is, along the square-zero extension 
\[
  0 \to (\lambda^{p-1})/(\lambda^p) \to A[\lambda]/(\lambda^p) \to A[\lambda]/(\lambda^{p-1})\to 0.
\]
In §\ref{subsection: proof of A}, that calculation will be used to prove Theorem \ref{maintheorem: differential}. 
For by Theorem \ref{theorem: initial-split}, there is a natural splitting of the HKR filtration to order $p-2$, which by §\ref{subsection: filtrations} gives an extension \(L\Omega^\bullet[\bullet] \to L\Omega^{\bullet + p-1}[\bullet + p]\) of functors. 
That extension will be induced by the deformation class associated to $\filS$ at order $p-1$.
%Since the filtration is multiplicative, this extension is induced by an extension \(L\Omega^1 \to L\Omega^p[p]\) of functors. 
%As the filtration and its splitting are induced by the filtered circle, so is this extension, which corresponds to the deformation class of \(\filS\) at order $p-1$, that is, along the square-zero extension 
%The remainder of this section is devoted to this class.

To describe the deformation theory of \(\filS\), we consider deforming the classifying stack of a group scheme.
Let \(G/\Spec S\) be a group scheme with identity section \(e\).
The cotangent complex \(L\Omega^1_{BG/S}\) of the classifying stack \(BG\) is equivalent to \(\ell_{G/S}[-1]\) where \(\ell_{G/S} := e^*L\Omega^1_{G/S}\) is the co-Lie complex of \(G\).
\footnote{This observation essentially goes back to Illusie \cite[Chapter VII]{Ill72}. For a modern perspective, see e.g.\ \cite[Example A.3.9]{KP24}.}
If $0 \to I \to \tilde S \to S \to 0$ is a square-zero extension, then deformations of the group scheme $G$ to $\Spec \tilde S$, if they exist, are a torsor over 
\[ \Ext^1_{BG}(L\Omega^1_{BG/S},I) = \Ext^2_{BG}(\ell_{G/S},I)\]
\cite[VII, Theorem 3.2.1]{Ill72}.
In our situation, $G = \Gahat^\vee$ and $S = A[\lambda]/(\lambda^{p-1})$. The group $G$ admits a lift to $\tilde S = A[\lambda]/(\lambda^p)$, namely $\Gahat^\vee$. This provides a basepoint for the torsor of lifts, and so we may associate to $\Glhat^\vee$ a class 
\[ [\Glhat] \in \Ext^2_{B\Gahat^\vee}(\ell_{\Gahat^\vee}, (\lambda^{p-1})/(\lambda^p)).
\] 
The main result of this section is Theorem \ref{theorem: first formula for deformation class}, which gives a formula for this class modulo $p$.

We begin by calculating the filtered co-Lie complex of \(\Glhat^\vee\) when \(A = \Z_{(p)}\).
It is possible to calculate the co-Lie complex of \(\Gacheck\) directly from the definition as the Cartier dual of \(\Ghat_a\)\footnote{The deformation class $[\Glhat^\vee]$ may also be computed directly using the formula \cite[Proposition D.2.2]{Dri24group}
\[ \Glhat^\vee = \Spec A[\lambda]\left[\lambda^n\binom{x/\lambda}{n} \mid n \geq 0 \right] ,\]
from which it follows that $\Glhat^\vee|_{\Fp[\lambda]/(\lambda^p)} = \Spec \Fp[\lambda][x_0,x_1,x_2\ldots]/(\lambda^p, x_0^p - \lambda^{p-1}x_0, x_1^p,x_2^p,\ldots)$. }.
However, in the direct approach, it is difficult to calculate the deformation class of \(\Glhat^\vee\). Instead, there is another approach to the formal group \(\Glhat\) using the group scheme of \(p\)-typical Witt vectors \(W\), discovered by Sekiguchi and Suwa \cite{SS01}. Moulinos, Robalo, and To\"en took this approach when they introduced the filtered circle \cite{MRT22}.

We recall standard properties of the Witt vectors, following the notation of \cite{SS01}. Let \(W\) be the \(p\)-typical Witt vector ring scheme over \(\Z\). As a scheme, \(W = \Spec \Z[T_0,T_1,T_2,\ldots]\). The ring structure of \(W\) is uniquely defined by requiring that the ghost maps
\[
  \Phi_i: W \to \mathbb A^1, \qquad \Phi_i(T_0,T_1,\ldots) = T_0^{p^i} + pT_1^{p^{i-1}} + \cdots + p^i T_i
\]
are ring homomorphisms.
A point of \(W\) is denoted \((x_0,x_1,\ldots)\).
The Witt vector scheme \(W\) has two endomorphisms \(F\) and \(V\), the Witt vector Frobenius and Verschiebung.
There is a multiplicative map \([-]: \mathbb{A}^1 \to W\) sending \(a\) to \((a,0,0,\ldots)\), called the \emph{multiplicative lift}.
Let \(\hat{W}\) be the formal Witt vector group, the completion of \(W\) at the identity.
\begin{proposition}[Sekiguchi-Suwa \cite{SS01}]
  \label{proposition: sekiguchi-suwa isomorphism}
  There is a short exact sequence
  % https://q.uiver.app/#q=WzAsNSxbMCwwLCIwIl0sWzEsMCwiXFxHbGhhdF5cXHZlZSJdLFsyLDAsIlciXSxbMywwLCJXIl0sWzQsMCwiMCJdLFswLDFdLFsxLDJdLFsyLDMsIkYtW1xcbGFtYmRhXntwLTF9XSJdLFszLDRdXQ==
  \[
    \begin{tikzcd}[column sep = large]
      0 & {\Glhat^\vee} & W & W & 0
      \arrow[from=1-1, to=1-2]
      \arrow[from=1-2, to=1-3]
      \arrow["{F-[\lambda^{p-1}]}", from=1-3, to=1-4]
      \arrow[from=1-4, to=1-5]
    \end{tikzcd}
  \]
  of filtered group schemes over \(\Z_{(p)}\). Further, the resulting isomorphism 
  \(\rho_{p-2}^* \Glhat^\vee \cong \Gahat^\vee\)
  agrees with the Cartier dual of the truncated exponential
  \[ E_\lambda(u) = u + \lambda \frac{u^2}{2!} + \cdots + \lambda^{p-2}\frac{u^{p-1}}{(p-1)!}.\]
\end{proposition}

\begin{proof}
The isomorphism \(\Ghat_\lambda^\vee \cong \ker(F - [\lambda^{p-1}]:W \to W)\) is also proved by Drinfeld in \cite[Proposition D.4.10]{Dri24group}. In Drinfeld's notation, \(G^! = \Glhat^\vee\) and \(G^{!?} = \ker(F - [\lambda^{p-1}]:W \to W)\).
We do not give a new proof of this isomorphism here, but we do give a formula for the map \(\Glhat^\vee \to W\).

It suffices to give a formula for the Cartier dual map \(W^\vee \to \Glhat\). Now \(W^\vee \cong \hat W\), the formal Witt vector group, via the Artin-Hasse-Witt pairing. Under this pairing, \(F\) and \(V\) are dual, while multiplication by an element of \(W\) is self-dual. Thus, it suffices to define a map
\[ \hat W / (V - [\lambda^{p-1}])\hat W \to \Glhat.\]
Recall that the ghost maps \(\Phi_i: W \to \mathbb A^1\) give an isomorphism \(W_{\mathbb Q} \cong \prod_{i \geq 0} \mathbb A^1_{\mathbb Q}\) and thus an isomorphism \(\hat W_{\mathbb Q} \cong \bigoplus_{i \geq 0} (\Gahat)_{\mathbb Q}\). In terms of the ghost coordinates,
\(V^*\Phi_i = p \Phi_{i-1},\)
while
\([a]^*\Phi_i = \Phi_i([a])\Phi_i = a^{p^i} \Phi_i.\)
Hence
\[ \hat W_{\mathbb Q}/(V -[\lambda^{p-1}])\hat W_{\mathbb Q} \cong (\Gahat)_{\mathbb Q}\]
via the map
\[ (T_0,T_1,\ldots) \mapsto \Phi_0 + \frac{\lambda^{p-1} \Phi_1}{p} + \frac{\lambda^{p^2-1}\Phi_2}{p^2} + \cdots.\]
Since every homomorphism \(\Ghat_a \to \Glhat\) is of the form \(u \mapsto (\exp(c \lambda u) - 1)/\lambda\) for some constant \(c\), the map \(\hat W_{\mathbb Q} \to (\Glhat)_{\mathbb Q}\) must be of the form
\[
  \frac{ \exp\left( c \lambda \left( \Phi_0 +
  \frac{\lambda^{p-1}\Phi_1}{p}  + \frac{\lambda^{p^2-1}\Phi_2}{p^2} + \cdots \right)\right) -1}{\lambda}.
\]
Observe that
\[ \sum_{i=0}^\infty \frac{\lambda^{p^i}\Phi_i}{p^i} = \sum_{j=0}^\infty \sum_{r=0}^\infty (\lambda^{p^j}T_j)^{p^r}{p^r},\]
and so
\begin{align*}
  \Lambda&: \left(\hat W / (V -[\lambda^{p-1}])\hat W\right)_{\mathbb Q} \to (\Glhat)_{\mathbb Q} \\
  \Lambda&(T_0,T_1,\ldots) = \frac{\prod_{j=0}^\infty  \exp(\sum_{r \geq 0} \frac{(\lambda^{p^j}T_j)^{p^r}}{p^r}) - 1}{\lambda}
\end{align*}
is a homomorphism.
The power series \(\exp(\sum_{r \geq 0} T^{p^r}/p^r)\) is the famous Artin-Hasse exponential and has coefficients in \(\Z_{(p)}\). Thus \(\Lambda \in \Z_{(p)}[\lambda]\series{T_0,T_1,\ldots}\) also, and thus \(\Lambda\) defines our homomorphism \(\Lambda: \hat W/ (V - [\lambda^{p-1}])\hat W \to \Glhat\) over \(\Z_{(p)}[\lambda]\).
Furthermore, \(\Lambda\) agrees with the canonical projection \(\hat W /V \hat W \cong \Gahat\) at \(\lambda = 0\).
The morphism \(\Lambda\) is the only homomorphism \(\hat W /(V - [\lambda^{p-1}])\hat W \to \Glhat\) which agrees with the canonical projection \(\hat W / V \hat W \to \Gahat\), as both sides are flat over \(\Z_{(p)}\) and \(\Lambda\) is the unique such homomorphism over \(\mathbb Q\). Finally,
\[ \Lambda(T_0,T_1,\ldots) \equiv T_0 + \lambda\frac{ T_0^2}{2!} + \cdots + \lambda^{p-2}\frac{T_0^{p-1}}{(p-1)!} \mod \lambda^{p-1},\]
so that the map \(\hat W \to \Glhat\) modulo \(\lambda^{p-1}\) is the truncation \(\hat W \to \hat W /V\hat W = \Gahat\) followed by the truncated exponential.
\end{proof}

%\begin{proof}[Sketch of proof]
%  Let \(e_p\) be the Artin-Schreier exponential
%  \[ e_p(T) = \exp\left(\sum_{r=1}^\infty \frac{T^{p^r}}{p^r}\right) \in \Z_{(p)}\series{T},\]
%  and let \(E\) be the Artin-Schreier-Witt exponential
%  \[ E((x_0,x_1,x_2,\ldots)) = \prod_{i=0}^\infty e_p(x_i^{p^i}).\]
%  Then \(W\) and \(\hat W\) are Cartier dual to each other over \(\Z_{(p)}\) via the Artin-Schreier-Witt pairing \(x,y \mapsto E(xy)\).
%  The endomorphisms \(V\) and \(F\) of \(W\) and \(\hat W\) are interchanged under Cartier duality. At \(\lambda = 0\), we find
%  \[
%    \ker(F: W \to W) \cong (\hat W / V\hat W)^\vee \cong \hat{\mathbb G}_a
%  \]
%  by the zeroth coordinate map \(T_0: \hat W \to \hat{\mathbb G}_a\). Sekiguchi and Suwa found a filtered deformation \(E_\lambda\) of this pairing between \(\Glhat\) and the kernel of \(F - [\lambda^{p-1}]\).
%  Their deformed pairing modulo \(\lambda^{p-1}\) maps \(\hat{W}/V\hat{W} \to \Glhat\) by \(\lambda^{-1}(\exp(\lambda T_0) - 1)\).
%\end{proof}

Using Proposition \ref{proposition: sekiguchi-suwa isomorphism}, we can compute the co-Lie complex of \(\Glhat^\vee \to \mathbb A^1/\mathbb G_m\).
\begin{corollary}\label{corollary: co-Lie of Ga dual}
  The filtered co-Lie complex of \(\Glhat^\vee\) over \(\Z_{(p)}\) is given by
  \begin{align*}
    \ell_{\Glhat^\vee} &\simeq \left(
      % https://q.uiver.app/#q=WzAsMixbMCwwLCJcXGVsbF9XIl0sWzIsMCwiXFxlbGxfVyJdLFswLDEsImQoRiAtIFtcXGxhbWJkYV57cC0xfV0pIl1d
      \begin{tikzcd}[ampersand replacement=\&]
        {\ell_W} \&\& {\ell_W}
        \arrow["{d(F - [\lambda^{p-1}])}", from=1-1, to=1-3]
      \end{tikzcd}
    \right).
  \end{align*}
\end{corollary}
\begin{proof}
  Given a short exact sequence of groups \(0 \to K \to H \overset{u}{\to} G \to 0\), the associated fiber sequence of cotangent complexes gives a fiber sequence \(u^*\ell_G \to \ell_H \to \ell_K \to^{+1}\) of \(H\)-representations \cite[VII, Proposition 3.1.1.5]{Ill72}. Apply this to Proposition \ref{proposition: sekiguchi-suwa isomorphism}.
\end{proof}

Now that we have calculated the co-Lie complex of $\Glhat$, to find the deformation class 
\[[\Glhat^\vee] \in \Ext^2_{B\Gahat^\vee/\Fp}(\ell_{\Gahat^\vee/\Fp}, \Fp \otimes (\lambda^{p-1})/(\lambda^p)),\]
we must compute extensions in the category of \(\Gacheck\)-representations.
The following Lemma describes representations of \(\Gacheck\) by Cartier duality.
Here it is more convenient to work with \(\Gacheck\) by hand.
If \(\OO(\Ghat_a) = A\series{u}\), then \(\OO(\Ghat_a)\) has a basis of divided powers \(\{s^{(i)} \mid i \geq 0\}\) dual to the basis of powers of \(u\).
\footnote{The description \(\OO(\Gahat^\vee) = A\{s^{(i)} \mid i \geq 0\}\) shows \(\Gahat^\vee\) is the group scheme \(\mathbb G_a^\sharp\), the divided power neighborhood of the identity in \(\mathbb G_a\). This group scheme appears in the theory of prismatic cohomology \cite{Dri24prism}, and the computation of its cohomology after Lemma \ref{lemma: category of Ga check reps} is similar to computations for \(\mathbb G_m^\sharp\) in \cite[§3.5]{BL22}.}
There is a derivation \(\Theta\) of \(\OO(\Gacheck)\) which sends \(s^{(i)}\) to \(s^{(i-1)}\). 
Under linear duality, \(u:\OO(\Gacheck) \to A\) is exactly \(\Theta\) followed by the counit of \(\Gacheck\).
In terms of Proposition \ref{proposition: sekiguchi-suwa isomorphism}, \(\Theta\) is the restriction of \(\partial/\partial T_0\) from \(W\) to \(W[F]\).

\begin{lemma}\label{lemma: category of Ga check reps}
  Let \(S\) be an animated ring.
  The functor
  \[ \QCoh(B\Gacheck/S) \to \Mod_{S[\theta]}\]
  assigning an \(S\)-linear representation \(M\) of \(\Gacheck\) to the infinitesimal action of \(\Theta \in \Lie(\Gacheck)\)
  is fully faithful. The essential image of the functor is \(\Mod_{S[\theta]}^{loc.nilp.}\), the full subcategory of those \(M\) where \(\theta\) acts locally nilpotently on all cohomology groups.
  The functor is symmetric monoidal with respect to the convolution monoidal structure
  \[ \theta_{M \otimes M'} \simeq \theta_M \otimes 1 + 1 \otimes \theta_{M'}\]
  on \(\Mod_{S[\theta]}\).
\end{lemma}

\begin{proof}
  By base change, it suffices to prove the claim when \(S = \Z\) \cite[Chapter 3, Proposition 3.4.2]{GR17I}.
  In this case \cite[Corollary 10.4.6.8]{Lurie-SAG} shows that \(\QCoh(B\Gacheck/\Z)\) is the left completion of the derived \(\infty\)-category of its heart \(\QCoh(B\Gacheck/\Z)^\heartsuit\).
  By fpqc descent along \(\Spec \Z \to B\Gacheck\), we find that \(\QCoh(B\Gacheck/\Z)^\heartsuit\) is the abelian category of \(\OO(\Gacheck)\)-comodules.
  By linear duality, the abelian category of \(\OO(\Gacheck)\)-comodules is equivalent to the abelian category of continuous \(\OO(\Gahat) = \Z\series{u}\)-modules.
  Now the left completion of the derived \(\infty\)-category of continuous \(\Z\series{u}\)-modules is exactly the full subcategory of \(\Mod_{\Z[u]}\) where \(u\) acts locally nilpotently on all cohomology groups.

  To complete the proof, we must relate the \(\Z[u]\)-module structure to the action of the Lie algebra of \(\Gacheck\). If \(M \in \Mod_{\Z[u]}^{loc.nilp.}\), then the corresponding dual coaction must be of the form \(\mathrm{coact} = \sum_i s^{(i)}u^i: M \to M \otimes \OO(\Gacheck)\). Since \(\Theta: \OO(\Gacheck) \to \OO(\Gacheck)\) is the derivation \(s^{(i)} \mapsto s^{(i-1)}\), we find
  \[ u = \Theta(\mathrm{coact})|_{s=0}: M \to M.\]
  Thus \(u\) agrees with the infinitesimal action of \(\Theta\).
\end{proof}
\begin{remark}
  Lemma \ref{lemma: category of Ga check reps} is essentially a consequence of Cartier duality, and is known to experts. 
  The proof here is reproduced from \cite[§4.5]{MM25} for the reader's convenience.
  In characteristic zero, Lemma \ref{lemma: category of Ga check reps} appears in \cite[Proposition 5.17]{NS23}. %It is an example of Laumon's generalized Fourier transform \cite{Lau96}.
  See \cite{AM25} and in particular \cite[Theorem 5.4.9]{AM25} for more details on the relationship between Lemma \ref{lemma: category of Ga check reps} and Cartier duality.
\end{remark}

By Lemma \ref{lemma: category of Ga check reps}, the trivial representation \(\Z\) of \(\Gahat^\vee\) over \(\Z\) has a resolution
\[
  % https://q.uiver.app/#q=WzAsNSxbMCwwLCIwIl0sWzEsMCwiXFxaIl0sWzIsMCwiXFxPTyhcXEdhaGF0XlxcdmVlKSJdLFszLDAsIlxcT08oXFxHYWhhdF5cXHZlZSkiXSxbNCwwLCIwIl0sWzIsMywiXFxUaGV0YSJdLFswLDFdLFsxLDJdLFszLDRdXQ==
\begin{tikzcd}[ampersand replacement=\&]
	0 \& \Z \& {\OO(\Gahat^\vee)} \& {\OO(\Gahat^\vee)} \& 0
	\arrow[from=1-1, to=1-2]
	\arrow[from=1-2, to=1-3]
	\arrow["\Theta", from=1-3, to=1-4]
	\arrow[from=1-4, to=1-5]
\end{tikzcd}
\] 
by copies of the regular representation \(\OO(\Gahat^\vee)\) of \(\Gahat^\vee\). Being the regular representation, \(R\Gamma(B\Gahat^\vee, \OO(\Gahat^\vee)) = \Z\), and thus
\begin{equation} \label{eq: cohomology of classifying stack}
  R^i\Gamma(B\Gacheck,\Z) =
  \begin{cases} \Z & i=0,1 \\ 0 & i \neq 0,1
  \end{cases}.
\end{equation}

\begin{definition}\label{definition: jordan block extension}
  Let \(\tau: \Z \to \Z[1]\) be the boundary of the extension of \(\Gacheck\) representations \(0 \to \Z \to \Z^2 \to \Z \to 0\) where \(\Theta\) acts on \(\Z^2\) by the matrix
  \[ \theta_\tau =
    \begin{pmatrix} 0 & 0 \\ 1 & 0
  \end{pmatrix}.\]
\end{definition}
The class $[\tau]$ is a basis for \(R^1\Gamma(B\Gacheck/A, A)\) for any discrete ring \(A\).

Now we return to computing \(\Ext^2_{B\Gacheck/\Fp}(\ell_{\Gacheck/\Fp}, \Fp \otimes (\lambda^{p-1})/(\lambda^p))\).
Let \(\{T_0,T_1,\ldots\}\) and \(\{S_0,S_1,\ldots\}\) be coordinates for the domain and codomain of \(F - [\lambda^{p-1}]: W \to W\), respectively.
Then \(\{dT_i\}_{i=0}^\infty\) is the basis for the degree \(0\) cochains of the complex \((\ell_W \to \ell_W)\) and \(\{dS_i\}_{i=0}^\infty\) is a basis for the degree \(-1\) cochains.
Note that \(dF \equiv 0 \mod p\), so the differential in the complex \(\ell_W \to \ell_W\) is zero modulo \((p,\lambda^{p-1})\). Since \(W\) is a commutative group scheme, the \(\Gacheck\)-action on this complex is canonically trivial.
Let \(\{\partial_{S_i}\}_{i=0}^\infty\) be dual to \(\{dS_i\}\) and let \(\{\partial_{T_i}\}\) be dual to \(\{dT_i\}\).
As \(\Ext^*_{B\Gacheck/\Fp}(\Fp,\Fp) = \Fp \oplus \Fp \tau\),

\[ \Ext^*_{B\Gacheck/\Fp}(\ell_{\Gacheck/\Fp},\Fp) =
  \begin{cases} \bigoplus_i \Fp \partial_{T_i} & \ast = 0\\
    \bigoplus_i \Fp \tau \partial_{T_i} \oplus \bigoplus_i \Fp \partial_{S_i} & \ast = 1 \\
    \bigoplus_i \Fp \tau \partial_{S_i} & \ast = 2 \\
    0 & \ast > 2
\end{cases}.\]
%is concentrated in degrees \([0,2]\), and in degree \(2\) has basis \(\{\tau \partial_{S_i}\}_{i=0}^\infty\).

\begin{theorem}\label{theorem: first formula for deformation class}
  The deformation class \([\Glhat^\vee] \in \Ext^2_{B\Gacheck}(\ell_{\Gacheck/\Fp},\Fp\otimes (\lambda^{p-1})/(\lambda^p))\) is
  \[ [\Glhat^\vee] = -\tau \partial_{S_0} \cdot \lambda^{p-1} \in \Ext^2_{B\Gacheck}(\ell_{\Gacheck/\Fp},\Fp \otimes (\lambda^{p-1})/(\lambda^p)).\]
\end{theorem}

\begin{proof}
  Let \(I = (\lambda^{p-1})/(\lambda^p) \subseteq \Fp[\lambda]/(\lambda^p)\).
  The Sekiguchi-Suwa exact sequence 
  % https://q.uiver.app/#q=WzAsNSxbMCwwLCIwIl0sWzEsMCwiXFxHbGhhdF5cXHZlZSJdLFsyLDAsIlciXSxbMywwLCJXIl0sWzQsMCwiMCJdLFswLDFdLFsxLDJdLFsyLDMsIkYtW1xcbGFtYmRhXntwLTF9XSJdLFszLDRdXQ==
  \[
    \begin{tikzcd}[column sep = large]
      0 & {\Glhat^\vee} & W & W & 0
      \arrow[from=1-1, to=1-2]
      \arrow[from=1-2, to=1-3]
      \arrow["{F-[\lambda^{p-1}]}", from=1-3, to=1-4]
      \arrow[from=1-4, to=1-5]
    \end{tikzcd},
  \]
  of Proposition \ref{proposition: sekiguchi-suwa isomorphism}, exhibits \(\Glhat^\vee\) as the kernel of a deformation of the endomorphism \(F: W \to W\) of the group scheme \(W\).
  We will express $[\Glhat^\vee]$ in terms of this deformation. 

  We review various compatibilities of deformations of group schemes and morphisms.
  %Deformations of the morphism \(F\) are a torsor over \(\Ext^1_{BW}(F^*\ell_W, I)\) 
  Let \(f: G \to H\) be a morphism of group schemes, and suppose that $f$ admits a deformation to $f_0: G_0 \to H_0$ along a square-zero extension with kernel $I$.
  Trivialize all torsors in sight with the basepoint corresponding to $f_0$.
  Deformations of $f$ with fixed domain and codomain are a torsor over \( \Ext^1_{BG}(f^*\ell_H, I)\)
  \cite[VII, §3.3]{Ill72}.
  Instead of just deforming a morphism, consider also deforming the domain.
  Deformations of \((G,f)\) with fixed codomain are a torsor over \(\Ext^2_{BG}(\ell_{\ker(f)}, I)\) \cite[VII, Theorem 3.3.2]{Ill72}. If $G \to H$ is surjective, the map 
  \[ \Ext^1_{BG}(f^*\ell_H,I) \to \Ext^2_{BG}(\ell_{\ker(f)},I)\]
  induced by the triangle $f^*\ell_{H} \to \ell_G \to \ell_{\ker(f)} \to^{+1}$ 
  corresponds to sending a deformation of $f$ to a deformation of $(G,f)$.
  Since deforming group schemes and morphisms is natural under pullback, if $(\tilde G, \tilde f)$ is a deformation of $(G,f)$, then the class $[\ker \tilde f] \in \Ext^2_{B\ker(f)}(\ell_{\ker(f)}, I)$ is the image of $[(\tilde G, \tilde f)]$ under 
  \[ \Ext^2_{BG}(\ell_{\ker(f)},I) \to \Ext^2_{B\ker(f)}(\ell_{\ker(f)}, I).\]

  %A deformation of \(F\) gives a deformation of \((W,F)\) where the domain is constant; this assignment is induced by the map of groups
  %\begin{equation}
  %  \label{eq: deform morphism to deform pair}\Ext^1_{BW}(F^*\ell_W, I) \to \Ext^2_{BW}(\ell_{\Gacheck},I)
  %\end{equation}
  %induced by the connecting morphism of the triangle \(F^*\ell_W \to \ell_W \to \ell_{\Gacheck} \to^{+1}\).
%
  %Now we compare deforming \((W,F)\) to deforming \(\Gacheck\), the kernel of \(F\).
  %Pulling back along \(B\Gacheck \to BW\) induces a morphism on cohomology
  %\begin{equation}
  %  \label{eq: pullback of deformation class}
  %  \Ext^2_{BW}(\ell_{\Gahat^\vee}, I) \to \Ext^2_{B\Gacheck}(\ell_{\Gahat^\vee}, I).
  %\end{equation}
  %Since the question of deforming group schemes and morphisms is natural under pullback, this morphism corresponds to pulling back along the horizontal arrows of the commutative square
  %% https://q.uiver.app/#q=WzAsNCxbMSwxLCJCVyJdLFsxLDAsIkJXIl0sWzAsMCwiQlxcR2FoYXReXFx2ZWUiXSxbMCwxLCJcXGFzdCJdLFszLDBdLFsyLDNdLFsyLDFdLFsxLDAsIkYiXSxbMiwwLCIiLDEseyJzdHlsZSI6eyJuYW1lIjoiY29ybmVyIn19XV0=
  %\[
  %  \begin{tikzcd}
  %    {B\Gacheck} & BW \\
  %    \ast & BW
  %    \arrow[from=1-1, to=1-2]
  %    \arrow[from=1-1, to=2-1]
  %    \arrow["\lrcorner"{anchor=center, pos=0.125}, draw=none, from=1-1, to=2-2]
  %    \arrow["F", from=1-2, to=2-2]
  %    \arrow[from=2-1, to=2-2]
  %\end{tikzcd}.\]
  %Thus, we send deformations of \((W,F)\) to deformations of \((\Gacheck, \ast)\), which is just the same as deformations of \(\Gacheck\) as a group scheme since there is a unique morphism to the trivial group.

  Thus, \([\Glhat^\vee]\) is the image of the deformation class of \(F - [\lambda^{p-1}]\) under the composition 
  \[ 
    \Ext^1_{BW}(F^*\ell_W,I) \to \Ext^2_{BW}(\ell_{\Gahat^\vee},I) \to \Ext^2_{B\Gahat^\vee}(\ell_{\Gahat^\vee},I).
  \]
  Since \(W\) is affine, we can identify \(\Ext^1_{BW}(F^*\ell_W, I)\) with the Hochschild homology \(H^1_{Hoch}(W, \Homs(F^*\ell_W,I))\). This is the space of crossed homomorphisms \(u: W \to \Homs(F^*\ell_W,I)\) up to equivalence.
  A crossed homomorphism \(u\) corresponds to the extension \(I \rtimes F^*\ell_W  \to F^*\ell_W \to I[1]\) of $W$-representations, where $W$ acts on $I \rtimes F^*\ell_W$ by  
  \[ act_{u} =
    \begin{pmatrix} 1 & 0 \\ u & act_{\ell_W}
  \end{pmatrix}.\]
  Note that \(F - [\lambda^{p-1}]\) differs from the basepoint \(F\) of our deformation problem by \((-[\lambda^{p-1}]): W\to W\). Multiplication by \(-[\lambda^{p-1}]\) is given by
  \begin{align*}
    -[\lambda^{p-1}](x_0,x_1,x_2,\ldots) & = -(\lambda^{p-1}x_0, \lambda^{p(p-1)}x_1,\lambda^{p^2(p-1)}x_2,\ldots) \\
    & \equiv (-\lambda^{p-1}x_0,0,0,\ldots) \mod \lambda^{p}.
  \end{align*}
  Thus \(F - [\lambda^{p-1}]\) corresponds to the crossed homomorphism \(u: W \to \Homs(F^*\ell_W,I)\) defined by
  \begin{equation*}
    u: x \mapsto -\lambda^{p-1} x_0 \partial_{S_0}.
  \end{equation*}
  Since \(T_0 \in \OO_W\) is the function \(x \mapsto x_0\), we see \(u = -\lambda^{p-1}T_0\partial_{S_0}\) as a crossed homomorphism.
  We have found the class of \(F - [\lambda^{p-1}]\) as a deformation of \(F\).

  There is a commutative square
  % https://q.uiver.app/#q=WzAsNCxbMSwxLCJcXEV4dF4yX3tCXFxHYWhhdF5cXHZlZX0oXFxlbGxfe1xcR2FoYXReXFx2ZWV9LEkpIl0sWzEsMCwiXFxFeHReMl97Qld9KFxcZWxsX3tcXEdhaGF0XlxcdmVlfSxJKSJdLFswLDAsIlxcRXh0XjFfe0JXfShGXipcXGVsbF9XLEkpIl0sWzAsMSwiXFxFeHReMV97QlxcR2FoYXReXFx2ZWV9KEZeKlxcZWxsX1csSSkiXSxbMywwXSxbMiwzXSxbMiwxXSxbMSwwXV0=
  \[
    \begin{tikzcd}
      {\Ext^1_{BW}(F^*\ell_W,I)} & {\Ext^2_{BW}(\ell_{\Gacheck},I)} \\
      {\Ext^1_{B\Gacheck}(F^*\ell_W,I)} & {\Ext^2_{B\Gacheck}(\ell_{\Gacheck},I)}
      \arrow[from=1-1, to=1-2]
      \arrow[from=1-1, to=2-1]
      \arrow[from=1-2, to=2-2]
      \arrow[from=2-1, to=2-2]
    \end{tikzcd}
  \]
  and the class we seek is the image of the crossed homomorphism \(-\lambda^{p-1} T_0\partial_{S_0}\) along the top and right. Now we go the other way around the commutative square.
  Note that \(-\lambda^{p-1}T_0\partial_{S_0}\) is the composition of \(-\partial_{S_0} \in \Hom(F^*\ell_W,A)\) and the crossed homomorphism \(\lambda^{p-1}T_0: W \to \Homs(A,I) = I\).
  The crossed homomorphism \(T_0: W \to A\) corresponds to the extension $0 \to A \to E \to A \to 0$ of \(W\)-representations whose underlying $A$-module is $A^2$ and with $W$-action given by the matrix 
  \[
  \begin{pmatrix} 1 & 0 \\ T_0 & 1
  \end{pmatrix}. 
  \]
  By Lemma \ref{lemma: category of Ga check reps}, we can compute the restriction of $E$ to $\Gacheck$ by computing the infinitesimal action of $\partial/\partial T_0$. As
  \begin{equation*}
    \frac{\partial}{\partial T_0}
    \begin{pmatrix} 1 & 0 \\ T_0 & 1
    \end{pmatrix} =
    \begin{pmatrix} 0 & 0 \\ 1 & 0
    \end{pmatrix},
  \end{equation*}
  the crossed homomorphism \(T_0\) restricted to \(\Gacheck\) corresponds to the extension \(\tau\) of Definition \ref{definition: jordan block extension}. This completes the proof.
\end{proof}
\begin{remark}
  Note that \(\mathbb G_m\) acts on everything in sight, where \(T_i\) has weight \(p^i\) and \(S_i\) has weight \(p^{i+1}\). This forces the class \([\Glhat]\) to be a scalar multiple of \(\tau \partial_{S_0}\cdot \lambda^{p-1}\). 
\end{remark}

\subsection{Bocksteins}\label{subsection: Bocksteins}
It turns out that the deformation class \([\Glhat^\vee]\) of Theorem \ref{theorem: first formula for deformation class} is in the image of a Bockstein homomorphism.
Recall that if \(M \in \Mod_{\Z/p^2\Z}\) and \(M_0 = M \otimes \Z/p\Z\), then the \emph{Bockstein} associated to \(M\) is the coboundary \(\Bock_M: M_0 \to M_0[1]\) associated to the distinguished triangle \(M \otimes (\Z/p \to \Z/p^2 \to \Z/p \overset{+1}{\to})\).
Since \(\Gacheck\) is defined over \(\Z/p^2\Z\) and the cotangent complex satisfies base change, we obtain a Bockstein homomorphism on \(\ell_{\Gacheck/\Fp}\).

\begin{lemma}\label{lemma: bockstein on co-Lie}
  If \(\Bock\) is the Bockstein on \(\ell_{\Gacheck/\Fp}\) induced by  \(\Gacheck\) over \(\Z/p^2\Z\), then under the equivalence \(\ell_{\Gacheck} \simeq (\ell_W \overset{dF}{\to} \ell_W)\) of Corollary \ref{corollary: co-Lie of Ga dual},
  \[ \Bock(dS_i) = dT_{i+1}.\]
\end{lemma}
\begin{proof}
  Recall that \(T_0,T_1,\ldots\) and \(S_0,S_1,\ldots\) are the standard coordinates on the domain and codomain of \(F: W \to W\). We claim that for \(F: W\to W\) the Witt vector Frobenius,
  \[ F^* dS_i = p dT_{i+1}\]
  in the co-Lie complex \(\ell_{W/\Z}\).
  If \(\Phi_i(T) = T_0^{p^i} + \cdots + p^i T_i\) is the \(i\)th ghost polynomial, then in the co-Lie complex \(d\Phi_i(T)|_{T=0} = p^i dT_i|_{T=0}\).
  The ghost polynomials \(\Phi_i\) satisfy \(F^*\Phi_i(S) = \Phi_{i+1}(T)\), and thus 
  \[ p^{i+1}dS_{i+1} = d\Phi_{i+1}(S)|_{S=0} = F^*d\Phi_i(T)|_{T=0} = p^i (F^*dT_{i})|_{T=0}.\]
  Since \(\ell_{W/\Z}\) is a free abelian group, we conclude that \(F^*dS_i = pdT_{i+1}\) in \(\ell_{W/\Z}\).
\end{proof}

We need to know how the Bockstein interacts with the tensor structure on quasi-coherent sheaves.

\begin{lemma}\label{lemma: leibniz for Bockstein}
  Let \(k\) be a perfect field of characteristic \(p\).
  Suppose \(X/k\) is a prestack with a lift to \(\Xtilde/W_2(k)\).
  Let \(M_0,M'_0 \in \QCoh(X)\) equipped with lifts \(M,M' \in \QCoh(\Xtilde)\) and associated Bockstein morphisms \(\Bock_M,\Bock_{M'}\).
  Then there are homotopies natural in \(\Xtilde\), \(M\), and \(M'\):
  \begin{enumerate}
    \item \(\Bock_{M \otimes M'} \sim \Bock_M \otimes 1_{M'_0} + 1_{M_0} \otimes \Bock_{M'}\).
    \item \(\Bock_{\Homs(M,M')} \sim (\Bock_{M'})_* - (\Bock_M)^*\).
  \end{enumerate}
\end{lemma}
\begin{proof}
  By descent, it suffices to treat \(\Xtilde = \Spec \tilde R\).
  If \(M\) is a dg \(\tilde R\)-module which is flat over \(\Z/p^2\Z\), then the Bockstein is defined at the chain level by \(\Bock_M(m_0) = dm/p\),
  where \(m_0\) is a cocycle in \(M_0\) and \(m\) is a chain lift of \(m_0\) to \(M\).
  \begin{enumerate}
    \item If \(M\) and \(M'\) are dg \(\tilde R\)-modules, then the differential on \(M \otimes M'\) is \(d_{M \otimes M'} = d_M \otimes 1_{M'} + 1_M \otimes d_{M'}\).
    \item The differential on \(\Homs(M,M')\) is
      \(d_{\Homs(M,M')} = (d_{M'})_* - (d_M)^*\).\qedhere
  \end{enumerate}
\end{proof}

\begin{remark}
  In modern terms, $\Bock_M$ is the tensor product of $M$ with $\alpha: \Z/p \to \Z/p[1]$, which is a $\Z/p^2$-linear derivation in the sense that the map $\Z/p \to \Z/p \oplus \Z/p[1]$ is a $\Z/p^2$-algebra homomorphism.
  The identity for the Bockstein for the tensor product follows from the Leibniz rule for $\alpha$.
\end{remark}

\begin{corollary}\label{corollary: deformation class is Bockstein}
  The class \([\Glhat^\vee] \in \Ext^2_{B\Gacheck/\Fp}(\ell_{B\Gacheck/\Fp},\Fp \otimes (\lambda^{p-1})/(\lambda^p))\) satisfies
  \[
    [\Glhat^\vee] = -\lambda^{p-1}\Bock(\tau \partial_{T_1}).
  \]
\end{corollary}
\begin{proof}
  By Lemma \ref{lemma: bockstein on co-Lie}, the Bockstein on \(\ell_{\Gacheck/\Fp}\) satisfies \(\Bock(dS_i) = dT_{i+1}\).
  By Lemma \ref{lemma: leibniz for Bockstein}, it follows that dually, \(\Bock(\partial_{T_{i+1}}) = - \partial_{S_i}\).
  Further \(\Bock(\tau) = 0\) since \(\tau\) is defined over \(\Z\).
  Again by the Leibniz rule, \(\Bock(\tau \partial_{T_1}) = -\tau \Bock(\partial_{T_1}) = \tau \partial_{S_0}\).
  Thus by Theorem \ref{theorem: first formula for deformation class},
  \[
    [\Glhat^\vee] = - \lambda^{p-1}\tau \partial_{S_0} = -\lambda^{p-1}\Bock(\tau \partial_{T_1}).\qedhere
  \]
\end{proof}

\subsection{Proof of Theorem \ref{maintheorem: differential}}\label{subsection: proof of A}

All the calculations about the filtered circle we require are now in hand. We must now apply the mapping stack construction to make conclusions about Hochschild homology.

First, we recall how deformations and derivations on \(B\Gacheck\) and \(\Mapsu(B\Gacheck,-)\) are related.
Let \(J\) be an \(A\)-module and \(\delta: \ell_{\Gacheck}[-1] = L\Omega^1_{B\Gacheck} \to J\) be a morphism in \(\QCoh(B\Gacheck/A)\).
Since the cotangent complex represents the functor of derivations, \(\delta\) induces a morphism
\[ \delta: B\Gacheck \times \Spec(A \oplus J) \to B\Gacheck.\]
Taking \(\Maps(\delta, X)\) gives a morphism
\[ \delta^*: \Mapsu(B\Gacheck, X) \to \Mapsu(B\Gacheck \times \Spec(A \oplus J),X).\]
By the adjunction between mapping stack and product, we have equivalences
\begin{align*}
  \Maps(\cX, \Mapsu(\cY \times \mathcal U, \mathcal Z)) \simeq \Maps(\cX \times \cY \times \mathcal U, Z)
  \simeq \Maps(\cX \times \mathcal U, \Mapsu(\cY,\mathcal Z))
\end{align*}
for all prestacks \(\cX,\cY, \mathcal Z, \mathcal U\).
Thus \(\delta^*\) above is equivalent to a morphism
\[ \delta^*: \Mapsu(B\Gacheck,X)\times \Spec(A \oplus J) \to \Mapsu(B\Gacheck,X),\]
natural in \(X/A\).
We have constructed a morphism from \(J\)-valued derivations of \(B\Gacheck\) to \(J\)-valued derivations of the mapping stack \(\Mapsu(B\Gacheck,-)\).
These derivations act on functions on the mapping stack, and this action is given by a morphism
\begin{equation}
  \label{eq: action morphism}
  act: R\Hom_{B\Gacheck}(L\Omega^1_{B\Gacheck},J) \otimes \OO_{\Mapsu(B\Gacheck,X)} \to \OO_{\Mapsu(B\Gacheck, X)} \otimes J.
\end{equation}
Recall that in Corollary \ref{corollary: deformation class is Bockstein} we showed that for \(J = (\lambda^{p-1})/(\lambda^p)\), the deformation class \([\Glhat^\vee] \in R^1\Hom(L\Omega^1_{B\Gacheck/\Fp}, J)\) is equal to \(-\lambda^{p-1}\Bock(\tau \partial_{T_1})\),
where \(\tau\) is as in Definition \ref{definition: jordan block extension} and \(\partial_{T_1}\) is the restriction of \(\partial/\partial T_1\in \Lie(W)\) to \(\Gahat^\vee \subseteq W\).
\begin{definition}\label{definition: V}
  If \(A\) is an \(\Fp\)-algebra,
  define the natural transformation \(V: L\Omega^1_{R/A}[1] \to L\Omega^p_{R/A}[p]\) on animated \(A\)-algebras \(R\) by
  \[
    V = act(\tau \partial_{T_1}).
  \]
\end{definition}

\begin{remark}
  There is another description of \(V\) by Cartier duality: it is the action on \(\Maps(B\Gacheck,-)\) of an element of the Lie algebra \(\End(\Gacheck)\) of \(\mathrm{Aut}(B\Gacheck)\) of cohomological degree zero defined over \(\mathbb{F}_p\).
  Under Cartier duality, 
  \[
    \pi_0\Hom_{\mathbb{F}_p}(\Gacheck,\Gacheck) \cong \pi_0\Hom_{\mathbb{F}_p}(\Gahat,\Gahat) = \mathbb{F}_p[F],
  \] 
  and \(V\) corresponds exactly to the endomorphism \(F\). That is, \(V\) is the action of the infinitesimal automorphism of \(B\Gacheck\) which is Cartier dual to \(1 + \epsilon F\).
\end{remark}

\emph{A priori}, \(V\) is a derivation of the ring of functions on the shifted tangent bundle \(T[-1]\Spec R\). However, this derivation is homogeneous of weight \(p-1\) and cohomological degree zero and thus defines a map \(L\Omega^\bullet[\bullet] \to L\Omega^{\bullet + p-1}[\bullet + p-1]\).
By Zariski descent, \(V\) extends to a natural transformation \(L\Omega^1_{X/A}[1] \to L\Omega^p_{X/A}[p]\) for characteristic \(p\) schemes \(X/A\).

\begin{theorem}\label{theorem: formula for differential}
  Let \(k\) be a field of characteristic \(p\).
  Let \(e: L\Omega^1 \to L\Omega^p[p]\) be the extension of the HKR filtration induced by the splitting to order \(p-2\) induced by the truncated exponential map.
  If \(X/k\) is a scheme with a lift to \(\Xtilde / W_2(k)\), then there is a homotopy
  \[ e \sim [V,\Bock_{\Xtilde}]: L\Omega^1_{X/k} \to L\Omega^p_{X/k}[p],\]
  natural in \(\Xtilde\).
\end{theorem}
\begin{proof}
  By Zariski descent, it suffices to assume \(X\) is affine. By Theorem \ref{theorem: filtered circle and HKR}, the HKR filtration on Hochschild homology is the filtration induced by the filtered circle \(\filS\) on \(\Mapsu(\filS, X)\). Thus the extension \(e\) is exactly the action of \([\Glhat^\vee]\) under the action morphism \eqref{eq: action morphism} on the ring of functions on \(\Mapsu(B\Gacheck, X) \simeq T[-1]X\).
  By Corollary \ref{corollary: deformation class is Bockstein},
  \[ [\Glhat^\vee]=  -\lambda^{p-1}\Bock(\tau \partial_{T_1})\]
  in \(\Ext^2_{B\Gacheck}(\ell_{B\Gacheck},(\lambda^{p-1})/(\lambda^p))\).
  Thus we find a natural homotopy between the extension \(e\) and the action of \(-\Bock(\tau \partial_{T_1})\).

  In the presence of a lift \(\Xtilde \to \Spec W_2(k)\), the action morphism \eqref{eq: action morphism} lifts over \(W_2(k)\), so Lemma \ref{lemma: leibniz for Bockstein} applies and provides a natural homotopy
  \begin{align*}
    act(-\Bock(\tau \partial_{T_1})) &\sim [\Bock_{\Xtilde}, -act(\tau \partial_{T_1})] \\
    &= [V,\Bock_{\Xtilde}]. \qedhere
  \end{align*}
\end{proof}

\section{Atiyah class and Verschiebung}

In Definition \ref{definition: V}, the natural transformation
\[V: L\Omega^1[1] \to L\Omega^p[p]\]
for \(A\)-schemes when \(A\) has characteristic \(p\) was defined by a rather indirect method: first, the shifted differential forms are viewed as functions on the mapping stack \(\Mapsu(B\Ghat_a^\vee,-)\); then, a vector field on \(B\Ghat_a^\vee\) is used to define a natural vector field on those mapping stacks. From this definition, it is not immediately clear how to calculate the morphism \(V\). In this section, the morphism \(V\) is related to the Atiyah class, which makes \(L\Omega^1[1]\) into a homotopy Lie coalgebra. The main result, Theorem \ref{theorem: V is a pth power} is that the endomorphism \(V\) is a \(p\)th power operation for the Atiyah cobracket. This allows \(V\) to be computed in more examples.

In §\ref{subsection: tannakian}, we introduce the rudiments of Tannakian formalism for derived algebraic geometry, following the work of Nuiten and Toën \cite{NT25}.
In §\ref{subsection: shifted tangent}, we describe the shifted tangent bundle of prestacks using Tannakian reconstruction and define \(V\) for such prestacks.
In §\ref{subsection: atiyah and shifted tangent bundle}, we find that \(V\) is a \(p\)th power operation for the Atiyah cobracket. In §\ref{subsection: restricted Lie algebras}, we explain that the identity from §\ref{subsection: atiyah and shifted tangent bundle} is an analogue of one of the identities of a restricted Lie algebra. Finally, in §\ref{subsection: V on classifying stack}, we show that the morphism \(V\) on the classifying stack \(BG\) of a group recovers the classical \(p\)th power operation on the Lie algebra of \(G\).

\subsection{Tannakian reconstruction}\label{subsection: tannakian}

In \cite{luriedagviii}, Lurie developed a version of Tannakian reconstruction in spectral algebraic geometry. In this version, one compares maps of spectral stacks with symmetric monoidal functors on quasi-coherent sheaves.
This is based on the fundamental identification of $\Einfty$-homomorphisms $A\to B$ with symmetric monoidal functors $\Mod_A \to \Mod_B$.
Since we work in derived algebraic geometry (that is, with animated rings instead of \(\Einfty\)-rings), we need an enhancement of the theory of symmetric monoidal categories, which accounts for the extra structure that an animated ring carries compared to the underlying \(\Einfty\)-ring.
An appropriate enhancement, the notion of \emph{\(\Theta\)-category}, was introduced in recent work of Nuiten and Toën \cite{NT25}. These are stable presentable symmetric monoidal categories equipped with a monad suitably compatible with the \(\Einfty\)-monad.

We now recall the precise definition of \(\Theta\)-category. To do so, we first need to introduce the 2-category of sifted-colimit-preserving monads. Let \(\Cat\) be the 2-category of categories, and let \(\Cat_{pr}^R\) be the 2-category of presentable stable categories and right adjoints. 

\begin{definition}[\cite{NT25}, Definition 1.1-1.2]\label{definition: mu-category}
  \begin{enumerate}
    \item A \(\mu\)-category is an object \(f_*: T' \to T\) of \(\Fun(\Delta^1,\Cat_{pr}^R)\) such that \(f_*\) is conservative and preserves sifted colimits.
    \item If \(f_*: T'_1 \to T_1\) and \(g_*: T'_2 \to T_2\) are \(\mu\)-categories, a \(\mu\)-morphism is a commutative square in \(\Cat\) 
    \begin{equation}\label{eq: mu-morphism}
    % https://q.uiver.app/#q=WzAsNCxbMCwwLCJUXzEnIl0sWzAsMSwiVF8xIl0sWzEsMCwiVF8yJyJdLFsxLDEsIlRfMiJdLFswLDIsInUiXSxbMSwzLCJ2Il0sWzAsMSwiZl8qIl0sWzIsMywiZ18qIl1d
\begin{tikzcd}[ampersand replacement=\&]
	{T_1'} \& {T_2'} \\
	{T_1} \& {T_2}
	\arrow["u", from=1-1, to=1-2]
	\arrow["{f_*}", from=1-1, to=2-1]
	\arrow["{g_*}", from=1-2, to=2-2]
	\arrow["v", from=2-1, to=2-2]
\end{tikzcd}
\end{equation}
%which is \emph{left adjointable}\footnote{\cite[Definition 4.7.4.13]{lurieha}}, that is, %NOT QUITE SEE NT25 1.4
where \(u\) and \(v\) preserve colimits
such that the induced natural transformation \(g^* \circ v \implies u \circ f^*\) is an equivalence.
  \item The 2-category of \(\mu\)-categories \(\Cat^\mu\) is the 1-full subcategory of \(\Fun(\Delta^1,\Cat)\) consisting of \(\mu\)-categories and \(\mu\)-morphisms.
  \end{enumerate}
\end{definition}

\begin{remark}[\cite{NT25}, Remark 1.4]
  By passing to the right adjoints \(u_*\) and \(v_*\) of \(u\) and \(v\), we obtain a 1-full functor 
  \(r: \Cat^\mu \to \Fun(\Delta^1,\Cat_{pr}^R)^{op}\)
  which sends a commutative square \eqref{eq: mu-morphism} to the left-adjointable square 
  \[ 
    % https://q.uiver.app/#q=WzAsNCxbMSwwLCJUXzEnIl0sWzEsMSwiVF8xIl0sWzAsMCwiVF8yJyJdLFswLDEsIlRfMiJdLFsyLDAsInVfKiIsMl0sWzMsMSwidl8qIiwyXSxbMCwxLCJmXyoiXSxbMiwzLCJnXyoiXV0=
\begin{tikzcd}[ampersand replacement=\&]
	{T_2'} \& {T_1'} \\
	{T_2} \& {T_1}
	\arrow["{u_*}"', from=1-1, to=1-2]
	\arrow["{g_*}", from=1-1, to=2-1]
	\arrow["{f_*}", from=1-2, to=2-2]
	\arrow["{v_*}"', from=2-1, to=2-2]
\end{tikzcd}
  \]
\end{remark}

\begin{example}
  If \(\Cat^\otimes_{pr}\) is the 2-category of presentable symmetric monoidal categories, there is a functor 
  \(\Einfty: \Cat^{\otimes}_{pr} \to \Cat^\mu\)
  sending a symmetric monoidal category \(T\) to \(T^{\Einfty} \to T\), the category \(T\) equipped with the \(\Einfty\)-algebra monad \cite[6]{NT25}.
\end{example}

With the definition of \(\mu\)-category in hand, we can now define a \(\Theta\)-category as a category equipped with both a \(\mu\)-structure and a symmetric monoidal structure satisfying certain compatibilities.
Let \(C\) be the 2-category defined by the lax pullback 
  \[
% https://q.uiver.app/#q=WzAsNCxbMCwwLCJDIl0sWzEsMCwiXFxDYXReXFxtdSJdLFswLDEsIlxcQ2F0Xlxcb3RpbWVzX3twcn0iXSxbMSwxLCJcXEZ1bihcXERlbHRhXjEsXFxDYXRfe3ByfV5SKV57b3B9Il0sWzIsMywiciBcXGNpcmMgXFxFaW5mdHkiLDJdLFsxLDMsInIiXSxbMCwxXSxbMCwyXSxbMiwxLCIiLDAseyJsZXZlbCI6Mn1dXQ==
\begin{tikzcd}[ampersand replacement=\&]
	C \& {\Cat^\mu} \\
	{\Cat^\otimes_{pr}} \& {\Fun(\Delta^1,\Cat_{pr}^R)^{op}}
	\arrow[from=1-1, to=1-2]
	\arrow[from=1-1, to=2-1]
	\arrow["r", from=1-2, to=2-2]
	\arrow[Rightarrow, from=2-1, to=1-2]
	\arrow["{r \circ \Einfty}"', from=2-1, to=2-2]
\end{tikzcd}
  \]
An object of \(C\) is a triple \((T_1, (T_2'\to T_2),u)\) where \(T_1 \in \Cat^\otimes_{pr}\) and \(u\) is a commutative square in \(\Cat^R_{pr}\) of the form
\[
% https://q.uiver.app/#q=WzAsNCxbMSwwLCJUXzFee1xcRWluZnR5fSJdLFswLDAsIlRfMiciXSxbMCwxLCJUXzIiXSxbMSwxLCJUXzEiXSxbMSwyXSxbMSwwXSxbMiwzXSxbMCwzXV0=
\begin{tikzcd}[ampersand replacement=\&]
	{T_2'} \& {T_1^{\Einfty}} \\
	{T_2} \& {T_1}
	\arrow[from=1-1, to=1-2]
	\arrow[from=1-1, to=2-1]
	\arrow[from=1-2, to=2-2]
	\arrow[from=2-1, to=2-2]
\end{tikzcd}\]

\begin{definition}[\cite{NT25}, Definition 1.5]\label{definition: theta-category}
  A \emph{\(\Theta\)-category} is an object \((T_1, (T_2'\to T_2),u)\) of \(C\) satisfying 
  \begin{enumerate}
    \item the induced (right adjoint) functor \(T_2 \to T_1\) is an equivalence; 
    \item the induced (right adjoint) functor \(T_2' \to T_1^{\Einfty}\) commutes with arbitrary colimits.
  \end{enumerate}
  The 2-category \(\thetaCat\) of \(\Theta\)-categories is the 0-full subcategory of \(C\) whose objects are \(\Theta\)-categories.
  Given a \(\Theta\)-category \(T = (T_1,(T_2'\to T_2),u)\), the category \(T_2'\) will be referred to by the notation \(\thetaAlg(T)\).
\end{definition}
Condition (2) of Definition \ref{definition: theta-category} is equivalent to the following: if \(\mu: T \to T\) is the endofunctor underlying the associated monad, then we have a map \(\Einfty \to \mu\) such that \(1 \to \mu(0)\) is an equivalence and \(\mu(x) \otimes \mu(y) \to \mu(x \oplus y)\) is an equivalence for all \(x,y \in T\).

\begin{lemma}[\cite{NT25}, Remark 1.6]
  \label{lemma: forget theta}
  The functor \(\thetaCat \to \Cat_{pr}\) sending a \(\Theta\)-category to its underlying category is conservative and preserves limits.
\end{lemma}

The main example of a \(\Theta\)-category we will need is the category of modules over an animated ring. Recall from §\ref{subsection: affine stacks} that the category of derived commutative $A$-algebras $\DAlg_A$ is the category of algebras for the monad $\Sym_A$ on $\Mod_A$.
There is a natural map \(\mathbb{E}_{\infty,A} \to \Sym_A\) and by \cite[Proposition 4.2.27]{Rak20}, the induced functor \(\DAlg_A \to (\Mod_A)^\Einfty\) preserves colimits. Thus, \(\Mod_A\) is equipped with a \(\Theta\)-structure.

\begin{definition}
  An \(A\)-linear \(\Theta\)-category is a morphism of \(\Theta\)-categories \(\Mod_A \to T\).
\end{definition}

Only using the \(\Theta\)-structure on \(\Mod_A\), we can construct an \(A\)-linear \(\Theta\)-structure on every accessible prestack over \(A\).
Following \cite[Definition 1.11]{NT25}, if \(T\) is a \(\Theta\)-category and \(B \in \thetaAlg(T)\), the category \(B\modc(T)\) carries a \(\Theta\)-structure where the functor \(\thetaAlg(B\modc(T)) \to B\modc(T)\) is given by the composite 
\[ \thetaAlg(B\modc(T)) := B/\thetaAlg(T) \to B/T^{\Einfty} \simeq (B\modc(T))^{\Einfty} \to B\modc(T).\]
This defines a functor 
\[ \thetaAlg(T) \to \Cat^\Theta_{T/}\]
which is a fully faithful left adjoint \cite[Proposition 1.12]{NT25}.
In particular, if \(T = \Mod_A\), then \(B\modc(T)\) is naturally identified with \(\Mod_B\) equipped with the monad \(\Sym_B\).

We can now define the \(\Theta\)-structure on the category of quasi-coherent sheaves on an accessible prestack.

\begin{definition}
  \(\QCoh^{\Sym}: \left(\PreStk^{acc}_A\right)^{op} \to \thetaCat\) is defined to be the right Kan extension of the functor 
  \[ \Mod: \CAlg_A \to \thetaCat\]
  sending an animated \(A\)-algebra \(B\) to the \(\Theta\)-category 
  \[B\modc(\Mod_A) \simeq \Mod_B.\]
\end{definition}

\begin{remark}
  Explicitly, for a prestack \(\cX\), the \(\Theta\)-category \(\QCoh^{\Sym}(\cX)\) is given by the limit 
  \[ \QCoh^{\Sym}(\cX) = \lim_{\Spec R \to \cX} \Mod_R\]
  with the limit \(\Theta\)-structure.
  Since the forgetful functor \(\thetaCat \to \Cat_{pr}\) is conservative and preserves limits, the underlying category of \(\QCoh^{\Sym}(\cX)\) is naturally identified with \(\QCoh(\cX)\).
\end{remark}

Given \(\Theta\)-categories \(C\) and \(D\) equipped with t-structures, 
let \(\FunthcA(C,D)\) be the space of \(A\)-linear morphisms of \(\Theta\)-categories \(C \to D\) which send \(C^{\leq 0}\) into \(D^{\leq 0}\).

\begin{definition}\label{definition: derived Tannakian prestack}
  A prestack \(\cX/\Spec A\) is \emph{Tannakian} if for all affine schemes \(Y/\Spec A\), the natural morphism 
  \[ \Maps_A(Y,\cX) \to \FunthcA(\QCoh^{\Sym}(\cX),\QCoh^{\Sym}(Y))\]  
  is an equivalence.
  %\TODO{what about almost-perfect complexes??}
\end{definition}
By \cite[Proposition 1.12]{NT25}, if \(\cX\) is an affine scheme, then \(\cX\) is Tannakian.
\begin{remark}
  The analogous definition in spectral algebraic geometry was introduced by Bhatt and Halpern-Leistner in \cite{BH17}, following work of Lurie \cite{luriedagviii}. By \cite[Theorem 4.1]{BH17}, if \(\cX\) is an fpqc spectral stack with quasi-affine diagonal such that \(\QCoh(\cX)\) is compactly generated, then for all affine spectral schemes \(Y/\Spec A\), the natural morphism 
  \[ \Maps_A(Y,\cX) \to \FuntcA(\QCoh(\cX),\QCoh(Y))\]
  is an equivalence.
  When \(\QCoh(\cX)\) is not necessarily compactly generated, \cite{luriedagviii} and \cite{BH17} prove similar statements about tensor functors required to preserve flat or almost-perfect complexes, respectively.
  The notion of derived Tannakian prestack in Definition \ref{definition: derived Tannakian prestack} will be studied in future work with Shubhankar Sahai \cite{MS}.
\end{remark}
Given \(A\)-linear \(\Theta\)-categories \(C\) and \(D\) equipped with \(t\)-structures, 
define the prestack \(\FunsthcA(C,D)\) by the functor of points
\[ \FunsthcA(C,D)(S) = \FunthcA(C, D \otimes \Mod_S).\]
\begin{lemma}
  The following are equivalent:
  \begin{enumerate}
    \item \(\cX/\Spec A\) is Tannakian;
    \item for all affine schemes \(Y/\Spec A\), the natural morphism 
    \[ \Mapsu_{A}(Y,\cX) \to \FunsthcA(\QCoh(\cX),\QCoh(Y))\]
    is an equivalence;
    \item for all prestacks \(\cY/\Spec A\), the natural morphism 
    \[\Mapsu_A(\cY,\cX) \to \FunsthcA(\QCoh(\cX),\QCoh(\cY))\]
    is an equivalence.
  \end{enumerate}
\end{lemma}
\begin{proof}
  The lemma follows from the definitions and that a prestack is the colimit of all affine schemes over it.
\end{proof}

\subsection{Shifted tangent bundles}\label{subsection: shifted tangent}

\begin{definition}
  Suppose \(\cX\) is a prestack and \(E \in \QCoh(\cX)\). Then the \emph{vector prestack} \(\V(E^\vee)\) is the prestack over \(\cX\) with functor of points the space of pairs
  \[ \V(E^\vee)(S) = \{( x \in \cX(S), f: x^*E \to S) \}.
  \]
\end{definition}

If \(E\) is connective, then \(\V(E^\vee) \to \cX\) is represented by the relative spectrum \(\underline{\Spec}_{\cX}(\Sym(E))\). More generally, if \(E\) has amplitude in \([-n,\infty)\), then \(\V(E^\vee) \to \cX\) is a relative \(n\)-stack.
%\footnote{Our notation for vector stacks is taken from \cite{GM25}.}

The shifted tangent bundle of a morphism of prestacks will be defined as a vector prestack. However, the shifted tangent bundle is only defined for those morphisms admitting a cotangent complex; this notion is reviewed in Appendix \ref{subsection: cotangent complex for stacks}. Implicitly, if a morphism \(\cX \to \cY\) admits a cotangent complex \(L\Omega^1_{\cX/\cY}\), then \(L\Omega^1_{\cX/\cY}\) is almost connective.

\begin{definition}\label{definition: shifted tangent bundle}
  Let \(\cX \to \cY\) be a morphism of prestacks admitting a cotangent complex.
  Then the \emph{\(n\)-shifted tangent bundle}
  \(T_{\cY}[n]\cX\) of \(\cX \to \cY\) is by definition the vector prestack
  \[ T_{\cY}[n]\cX = \V((L\Omega^1_{\cX/\cY}[-n])^\vee).\]
\end{definition}

If the base prestack is \(\cY = \Spec A\), then we will simply write \(T[n]\cX\) for the \(n\)-shifted tangent bundle of \(\cX \to \Spec A\).
It follows from the definition of the cotangent complex that if \(n \geq 0\), then \(T[n]\cX = \Mapsu(\Spec(A\oplus A[n]),\cX)\).

Moulinos-Robalo-Toën proved that \(\Mapsu(B\Gahat^\vee, \Spec R) \simeq T[-1]\Spec R\) for all animated \(A\)-algebras \(R\) (Corollary \ref{corollary: associated graded is shifted tangent bundle}). 
It turns out that for more general prestacks, the shifted tangent bundle is not the stack of maps out of \(B\Gahat^\vee\), but instead the stack of maps out of the closely related stack \(B\Ga^\vee\), where \(\Ga^\vee\) is the Cartier dual of \(\Ga\).
Concretely, \(\Ga^\vee = \Spf A[\partial^{(n)} | n \geq 0]\) where \(A[\partial^{(n)} | n \geq 0]\) is the free divided power polynomial ring on \(\delta\) with the divided power filtration.
\(\Ga^\vee\) is a group ind-scheme.

To apply Tannakian reconstruction to \(\Maps(B\Ga^\vee,-)\), we must have a description of the \(\Theta\)-category \(\QCoh^{\Sym}(B\Ga^\vee)\).
As an intermediate step, we first describe the category \(\QCoh(B\Ga^\vee)\) via the Fourier-Mukai transform.
%\footnote{We do not attempt to explain the compatibility between the Fourier-Mukai transform and the \(\Theta\)-structure.}

\begin{lemma}[\cite{AM25}, Theorem C]
  \label{lemma: reps of Ga dual}
  Let \(R\) be an animated ring and let \(\pi: \Ga \to \Spec R\) be the additive group over \(\Spec R\). Let \(z: \Spec R \to B\Ga^\vee\) be the zero section of \(B\Ga^\vee\). Then there is a symmetric monoidal equivalence
  \[ \Phi_{\Ga}: (\QCoh(\Ga),\star) \overset{\sim}{\rightarrow} \QCoh(B\Ga^\vee)\]
  fitting into a commutative square 
  \[
  % https://q.uiver.app/#q=WzAsNCxbMSwwLCJcXFFDb2goQlxcR2FeXFx2ZWUpIl0sWzEsMSwiXFxNb2RfUiJdLFswLDAsIlxcUUNvaChcXEdhKSJdLFswLDEsIlxcTW9kX1IiXSxbMiwwLCJcXFBoaV97XFxHYX0iXSxbMCwxLCJ6XioiXSxbMywxLCI9Il0sWzIsMywiXFxwaV8qIl1d
\begin{tikzcd}
	{\QCoh(\Ga)} & {\QCoh(B\Ga^\vee)} \\
	{\Mod_R} & {\Mod_R}
	\arrow["{\Phi_{\Ga}}", from=1-1, to=1-2]
	\arrow["{\pi_*}", from=1-1, to=2-1]
	\arrow["{z^*}", from=1-2, to=2-2]
	\arrow["{=}", from=2-1, to=2-2]
\end{tikzcd}
.\]
\end{lemma}
In other words, Lemma \ref{lemma: reps of Ga dual} identifies \(\QCoh(B\Ga^\vee)\) with modules over a polynomial ring in one variable \(R[\dd]\).

\begin{remark}
  Note that \(B\Ga^\vee\) is the 
  \emph{fpqc classifying stack}
  of \(\Ga^\vee\). 
  This may be understood as the geometric realization \(\colim{[n] \in \Delta^{op}} (\Ga^\vee)^{\times n}\) of the nerve of \(\Spec A \to B\Ga^\vee\) in fpqc stacks.
  This does not agree with the corresponding colimit in prestacks, unlike the case of the classifying stack of a flat affine group scheme. However, \(\QCoh\) is insensitive to this sheafification \cite[Chapter 3, Corollary 1.3.8]{GR17I}.
\end{remark}

\begin{remark}
  The Lemma holds for any base animated ring \(R\), but since \(\Ga\) is defined over \(\Z\) and \(\QCoh(\cX \times \Spec R) = \QCoh(\cX) \otimes \Mod_R\) for any prestack \(\cX\) and any animated ring \(R\) \cite[Chapter 3, 3.5.1]{GR17I}, it suffices to prove the Lemma when \(R = \Z\). Although the statement only involves quasi-coherent sheaves, the proof passes through ind-coherent sheaves.
\end{remark}

Our description of \(\QCoh^{\Sym}(B\Ga^\vee)\)
is constructed using an element \(\DD: \Spec A[\epsilon]/\epsilon^2 \to \mathbb G_a^\vee\) of the Lie algebra of \(\Ga^\vee\). 
Viewing \(\Ga^\vee = \Hom(\Ga,\mathbb{G}_m)\), \(\DD\) is the homomorphism \(x \mapsto 1 + \epsilon x: \mathbb G_a \to \mathbb G_m\).
Then \(\DD\) induces a map 
\begin{equation}\label{eq: lie algebra element of Gadual}% https://q.uiver.app/#q=WzAsNSxbMCwwLCJcXFNwZWMgQVtcXGVwc2lsb25dL1xcZXBzaWxvbl4yIl0sWzEsMSwiXFxtYXRoYmIgR19hXlxcdmVlIl0sWzEsMiwiXFxhc3QiXSxbMiwxLCJcXGFzdCJdLFsyLDIsIkJcXG1hdGhiYiBHX2FeXFx2ZWUiXSxbMiw0XSxbMSwyXSxbMSwzXSxbMyw0XSxbMCwxXSxbMSw0LCIiLDIseyJzdHlsZSI6eyJuYW1lIjoiY29ybmVyIn19XV0=
  \begin{tikzcd}[ampersand replacement=\&]
    {\Spec A[\epsilon]/\epsilon^2} \\
    \& {\mathbb G_a^\vee} \& \ast \\
    \& \ast \& {B\mathbb G_a^\vee}
    \arrow[from=1-1, to=2-2]
    \arrow[from=2-2, to=2-3]
    \arrow[from=2-2, to=3-2]
    \arrow["\lrcorner"{anchor=center, pos=0.125}, draw=none, from=2-2, to=3-3]
    \arrow[from=2-3, to=3-3]
    \arrow[from=3-2, to=3-3]
  \end{tikzcd}
.
\end{equation}

\begin{remark}\label{remark: FM transform and derivative for Ga}
  The equivalence in Lemma \ref{lemma: reps of Ga dual} sends \(E \in \QCoh(B\Ga^\vee)\) to the \(R[\dd]\)-module whose underlying \(R\)-module is \(z^*E\), where the action of \(\dd\) is given by the infinitesimal action of \(\DD\), as in Lemma \ref{lemma: category of Ga check reps}.
\end{remark}

\begin{lemma}\label{lemma: reps of Ga dual as pullback}
  For all animated \(A\)-algebras \(R\), 
  the diagram \eqref{eq: lie algebra element of Gadual} induces a pullback square of \(\Theta\)-categories
  \begin{equation}\label{eq: endomorphisms as pullback}
    % https://q.uiver.app/#q=WzAsNCxbMCwwLCJcXFFDb2hee1xcU3ltfShCXFxtYXRoYmJ7R31fYV5cXHZlZS9SKSJdLFswLDEsIlxcTW9kX1IiXSxbMSwwLCJcXE1vZF9SIl0sWzEsMSwiXFxNb2Rfe1JbXFxlcHNpbG9uXS9cXGVwc2lsb25eMn0iXSxbMCwxXSxbMCwyXSxbMSwzXSxbMiwzXV0=
\begin{tikzcd}[ampersand replacement=\&]
	{\QCoh^{\Sym}(B\Ga^\vee/R)} \& {\Mod_R} \\
	{\Mod_R} \& {\Mod_{R[\epsilon]/\epsilon^2}}
	\arrow[from=1-1, to=1-2]
	\arrow[from=1-1, to=2-1]
	\arrow[from=1-2, to=2-2]
	\arrow[from=2-1, to=2-2]
\end{tikzcd}
  .
  \end{equation}

\end{lemma}
\begin{proof}
  By Lemma \ref{lemma: forget theta}, the functor of forgetting the \(\Theta\)-structure is conservative and preserves limits. 
  Hence, it suffices to show that \eqref{eq: endomorphisms as pullback} is a pullback square of categories.

  Under the Fourier-Mukai equivalence \(\Phi_{\Ga}:\Mod_{R[\dd]} \simeq \QCoh(B\Ga^\vee)\) of Lemma \ref{lemma: reps of Ga dual}, the map 
  \[ T: \Mod_{R[\dd]} \to \Mod_R \underset{\Mod_{R[\epsilon]/\epsilon^2}}{\times} \Mod_R \]
  sends an \(R[\dd]\)-module \(M\) to the restriction of scalars of \(M\) along \(R \to R[\dd]\),
  together with the automorphism \(1 + \epsilon \dd\) 
  of \(M \otimes_R R[\epsilon]/\epsilon^2\).
  
  First, we check that \(T\) is fully faithful on the generator \(R[\dd]\) of \(\Mod_{R[\dd]}\).
  Given \(R[\dd]\)-modules \(M\) and \(N\),
  \(T\) induces a sequence
  \[ R\Hom_{R[\dd]}(M,N) \to R\Hom_{R}(M,N) \oplus R\Hom_R(M,N) \overset{a}{\to} R\Hom_{R[\epsilon]/\epsilon^2}(M[\epsilon]/\epsilon^2, N[\epsilon]/\epsilon^2), \]
  where the map \(a\) sends \((x,y)\) to \(x- (1+\epsilon \dd)y(1 - \epsilon \dd) = (x-y) + \epsilon[\dd,y]\).
  This is easily seen to be a fiber sequence when \(M = N = R[\dd]\), and thus \(T\) is fully faithful on the generator.

  Secondly, we check that \(T\) is essentially surjective.
  Given a pair of \(R\)-modules \(M\) and \(M'\) and an isomorphism \(\phi: M[\epsilon]/\epsilon^2 \to M'[\epsilon]/\epsilon^2\),
  let \(a = \phi \mod \epsilon : M \to M'\).
  Then \(a\) is an isomorphism, and so 
  \((M,M',\phi) \cong (M,M, a^{-1}\phi)\).
  As \(a^{-1}\phi\) is equivalent to 1 mod \(\epsilon\),
  it is of the form \(1 + \epsilon \dd\) for \(\dd:M \to M\), so \((M,M,a^{-1}\phi)\) is in the essential image of \(T\).
\end{proof}

\begin{theorem}\label{theorem: general tangent stack as mapping stack}
  If \(\cX/\Spec A\) is a Tannakian prestack admitting a cotangent complex, then there is a natural equivalence
  \[ \Psi: \Mapsu(B\Ga^\vee, \cX) \simeq T[-1]\cX.\]
\end{theorem}

\begin{proof}
  Since \(\cX/\Spec A\) admits a cotangent complex, applying \(\Mapsu_A(-,\cX)\) to \eqref{eq: lie algebra element of Gadual} gives a commutative square
  \begin{equation}\label{eq: from mapping stack to shifted tangent bundle}
  % https://q.uiver.app/#q=WzAsNCxbMCwwLCJcXE1hcHN1KEJcXG1hdGhiYntHfV9hXlxcdmVlLFgpIl0sWzAsMSwiWCJdLFsxLDAsIlgiXSxbMSwxLCJUWCJdLFswLDFdLFsxLDNdLFsyLDNdLFswLDJdXQ==
  \begin{tikzcd}[ampersand replacement=\&]
  	{\Mapsu(B\Ga^\vee,\cX)} \& \cX \\
  	\cX \& T\cX
  	\arrow[from=1-1, to=1-2]
  	\arrow[from=1-1, to=2-1]
  	\arrow[from=1-2, to=2-2]
  	\arrow[from=2-1, to=2-2]
  \end{tikzcd}
  \end{equation}
  which induces a map \(\Psi: \Maps(B\Ga^\vee,\cX) \to T[-1]\cX\).
  We will show that \(\Psi\) is an equivalence.

  Apply \(\FunsthcA(\QCoh^{\Sym}(\cX),\QCoh^{\Sym})\) to \eqref{eq: endomorphisms as pullback}.
  If we did not require our functors to preserve connective objects, we would immediately obtain a pullback square.
  But since \(B\Ga^\vee\) is the fpqc classifying stack of \(\Ga^\vee\), the zero section \(z: \Spec A \to B\Ga^\vee\) is an fpqc cover, 
  so for \(E \in \QCoh(B\Ga^\vee)\), \(E\) is connective if and only if \(z^*E\) is connective.
  Thus \(\alpha \in \FunsthA(\QCoh(\cX),\QCoh(B\Ga^\vee))(S)\) preserves connective objects if and only if \(z^*\alpha\) does.
  
  Thus we have a pullback square
\begin{equation}\label{eq: tannakian square for shifted tangent bundle}
% https://q.uiver.app/#q=WzAsNCxbMCwwLCJcXEZ1bnN0aGNBKFxcUUNvaChcXG1hdGhjYWwgWCksIFxcUUNvaChCXFxtYXRoYmJ7R31fYV5cXHZlZSkpIl0sWzEsMCwiXFxGdW5zdGhjQShcXFFDb2goXFxtYXRoY2FsIFgpLFxcTW9kX1IpIl0sWzAsMSwiXFxGdW5zdGhjQShcXFFDb2goXFxtYXRoY2FsIFgpLFxcTW9kX1IpIl0sWzEsMSwiXFxGdW5zdGhjQShcXFFDb2goXFxtYXRoY2FsIFgpLFxcTW9kX3tSW1xcZXBzaWxvbl0vXFxlcHNpbG9uXjJ9KSJdLFswLDJdLFswLDFdLFsyLDNdLFsxLDNdXQ==
\begin{tikzcd}[ampersand replacement=\&]
	{\FunsthcA(\QCoh(\cX), \QCoh(B\Ga^\vee))} \& {\FunsthcA(\QCoh(\cX),\Mod_R)} \\
	{\FunsthcA(\QCoh(\cX),\Mod_R)} \& {\FunsthcA(\QCoh(\cX),\Mod_{R[\epsilon]/\epsilon^2})}
	\arrow[from=1-1, to=1-2]
	\arrow[from=1-1, to=2-1]
	\arrow[from=1-2, to=2-2]
	\arrow[from=2-1, to=2-2]
\end{tikzcd}
\end{equation}
  If \(\cX\) is Tannakian, then \eqref{eq: tannakian square for shifted tangent bundle} is equivalent to \eqref{eq: from mapping stack to shifted tangent bundle}, and thus \(\Psi\) is an equivalence.
\end{proof}

\begin{remark}
  By construction, the composition of $\Psi$ with the structure map $\pi: T[-1]\cX \to \cX$ is induced by pullback along the quotient map $z: \Spec A \to B\Ga^\vee$.
\end{remark}

\begin{remark}
  If \(\mathbb{Q} \subseteq A\), then \(\Ga^\vee = \Gahat\); in this case, Naef and Safranov \cite{NS23} proved Theorem \ref{theorem: general tangent stack as mapping stack} without any Tannakian hypothesis using Lie theory.
  In order to obtain the shifted tangent stack away from characteristic zero, one must work with \(\Ga^\vee\) instead of \(\Gahat\).
\end{remark}

Consider the morphism \(B\Ga^\vee \to B\Gacheck\) induced by the Cartier dual of the inclusion \(\Gahat \to \Ga\).
By Theorem \ref{theorem: B H dual affine stack}, \(B\Gacheck\) is an affine stack. 
From Lemma \ref{lemma: reps of Ga dual}, it follows that \(B\Ga^\vee\) and \(B\Gacheck\) have the same global functions, that is, the natural map \(R\Gamma(B\Gacheck, \OO_{B\Gacheck}) \to R\Gamma(B\Ga^\vee,\OO_{B\Ga^\vee})\) is an equivalence.
Thus \(B\Gacheck\) is the affinization of \(B\Ga^\vee\).
Together with Theorem \ref{theorem: general tangent stack as mapping stack}, this gives another proof of Corollary \ref{corollary: associated graded is shifted tangent bundle}: for an animated \(A\)-algebra \(R\),
\[ \Mapsu(B\Gacheck, \Spec R) \simeq \Mapsu(B\Ga^\vee, \Spec R) \simeq T[-1]\Spec R.\]

However, when \(\cX\) is not an affine scheme, the induced map \(\Mapsu(B\Gahat^\vee, \cX) \to \Maps(B\Ga^\vee, \cX) \simeq T[-1]\cX\) is not necessarily an equivalence.

\begin{example}
  Consider \(\cX = BG\) for a group scheme \(G/\Spec A\).
  For any group prestack \(H\), \(\Maps(BH, BG)\) is equivalent to the space of homomorphisms \(H \to G\) modulo conjugation by \(G\). 
  
  For the rest of this example, let us ignore derived structure and calculate the classical homomorphism sets \(\Hom(\Gacheck,G) \to \Hom(\Ga^\vee,G)\) over a classical ring \(R\).
  Let \(\Delta\) and \(\epsilon\) be the coproduct and counit of the Hopf algebra \(\OO_G\).
  Viewing \(\Ga^\vee = \Spf R[\partial^{(n)} | n \geq 0]\), we find that a homomorphism \(\Ga^\vee \to G\) is exactly 
  %a compatible collection of morphisms from \(\OO_G\) to truncations of \(R[\partial^{(n)} | n \geq 0]\), i.e. 
  an infinite sum of the form \(\alpha = \sum_{n=0}^\infty \beta_n \partial^{(n)}\) where \(\beta_n: \OO_G \to R\) are linear maps satisfying the following four conditions:
  \begin{itemize}
    \item (unit) \[\beta_n(1) = \begin{cases}
      1 & n = 0\\
      0 & n \neq 0 
    \end{cases}\]
    \item (counit) \(\beta_0 = \epsilon\)
    \item (product) \(\beta_n(fg) = \sum_{i=0}^n \binom{n}{i}\beta_i(f)\beta_{n-i}(g)\) for all \(n\)
    \item (coproduct) \(\beta_n(f) = \beta_i\beta_j(\Delta f)\) whenever \(n = i+ j\).
  \end{itemize}
  The product rule implies \(\beta_1(fg) = \epsilon(f)\beta_1(g) + \beta_1(f)\epsilon(g)\), that is. that \(\beta_1 \in \Lie G\).
  Given such an element, \(\beta_n\) is uniquely defined by \(\beta_n = \beta_1^n(\Delta^{(n)}f)\),
  which satisfies the product rule thanks to the coassociativity of \(\Delta\).
  Thus the natural map 
  \[ \Hom(\Ga^\vee,G) \to \Hom(\Lie \Ga^\vee,\Lie G) \overset{\DD}{\to} \Lie G\]
  is an isomorphism.
  
  Now a homomorphism \(\Gacheck \to G\) is exactly such an infinite sum \(\sum_{n=0}^\infty \beta_n\partial^{(n)}\) satisfying the condition that for all \(f \in \OO_G\), \(\beta_n(f) = 0\) for \(n \gg 0\).
  This is equivalent to requiring that \(\beta_1 \in \Lie G\) defines a locally nilpotent operator on \(\OO_G\). 
  Hence \(\Hom(\Gahat^\vee, G) \to \Hom(\Ga^\vee, G)\) is exactly the inclusion of the ind-scheme of nilpotent elements of \(\Lie(G)\).  
  This agrees with Chen's calculation in characteristic zero \cite[Proposition 2.1.25]{Che20}.
\end{example}

Let \(\cX/\Spec A\) be a prestack admitting a \((-1)\)-connective cotangent complex. Then \(T[-1]\cX \simeq \underline{\Spec}_{\cX}(\Sym(L\Omega^1_{\cX/\Spec A}[1]))\) is relatively affine over \(\cX\). It follows that morphisms \(T[-1]\cX \to T[-1]\cX\) over \(\cX\) are equivalent to maps \(L\Omega^1_{\cX/\Spec A}[1] \to \Sym(L\Omega^1_{\cX/\Spec A}[1])\). 
If \(\cX\) is Tannakian, then we can apply Theorem \ref{theorem: general tangent stack as mapping stack} to generalize Definition \ref{definition: V}:

\begin{definition}\label{definition: general V}
  If \(A\) is an \(\Fp\)-algebra and \(\cX/\Spec A\) is a Tannakian prestack admitting a \((-1)\)-connective cotangent complex,
  then define
  \[V: L\Omega^1_{\cX/A}[1] \to \Sym^p(L\Omega^1_{\cX/A}[1]) \simeq L\Omega^p_{\cX/A}[p]\]
  to be the map on functions on \(T[-1]\cX \simeq \Mapsu(B\Ga^\vee,\cX)\)
  induced by the action of the infinitesimal automorphism of \(B\Ga^\vee\) Cartier dual to \(1 + \epsilon F: \Ga \to \Ga\).
\end{definition}

As the Frobenius is a natural transformation, the action of the Cartier dual of \(1+ \epsilon F\) is compatible with the morphism \(B\Ga^\vee \to B\Gacheck\).
Thus Definition \ref{definition: general V} is compatible with Definition \ref{definition: V}.

\begin{remark}
  The map \(T[-1]\cX \to \cX\) is the relative spectrum of the derived ring \(\Sym( L\Omega^1_{\cX/A}[1]) \in \DAlg_{\cX}\). \cite[Theorem 1.3]{MM25} states that taking the relative spectrum of a derived ring is fully faithful on bounded below derived rings.
  Thus, for given \(\cX/A\),
  if \(\Sym(L\Omega^1_{\cX/A}[1])\) is bounded below, then Definition \ref{definition: general V} makes sense for \(\cX\).
\end{remark}

\subsection{The Atiyah class}\label{subsection: atiyah and shifted tangent bundle}

Let $\cX/\Spec A$ be a prestack admitting a cotangent complex.
The canonical Lie algebra element \(\DD \in \Lie \Ga^\vee\) induces a map 
\[\Mapsu(B\Ga^\vee,\cX) \to T[-1]\cX,\]
which by Theorem \ref{theorem: general tangent stack as mapping stack} is an equivalence when \(\cX\) is Tannakian.
In this section, we will reinterpret this equivalence in terms of the Atiyah class, following an idea of Markarian. 
In his work on the Atiyah class and Hochschild homology, Markarian showed that in the setting of the triangulated category \(\pi_0\QCoh(\cX)\), the Atiyah class is the universal assignment to each \(E \in \QCoh(\cX)\) an endomorphism of \(E\) satisfying the Leibniz rule \cite[Introduction]{Mar09}. As a direct consequence of this interpretation, the morphism \(V\) is seen to be a \(p\)th-power operation for the Atiyah cobracket on \(L\Omega^1[1]\).

We now recall the Atiyah class and prove in Proposition \ref{proposition: mapping stack by Atiyah class} a universal property of the Atiyah class in our setting of derived stacks.
If \(R' \to R\) is a morphism of animated rings and \(M\) is an \(R\)-module, then the \emph{Atiyah class} of \(M\) is the morphism
\[
  at_M: M \to M \otimes L\Omega^1_{R/R'}[1]
\] induced by the first principal parts sequence
\[ M \otimes L\Omega^1_{R/R'} \to LP^1(M) \to M \to^{+1},\]
see e.g.\ \cite[\href{https://stacks.math.columbia.edu/tag/09DF}{Tag 09DF}]{stacks-project}.
In the more general context of a morphism \(\cX \to \cY\) of prestacks admitting a cotangent complex, every
\(E \in \QCoh(\cX)\) admits an Atiyah class
\begin{equation}\label{eq: atiyah class}
  at_E: E \to E \otimes L\Omega^1_{\cX/\cY}[1],
\end{equation}
whose definition we now give. The relative loop space 
\[
% https://q.uiver.app/#q=WzAsNCxbMCwwLCJcXG1hdGhjYWx7TH1fe1xcbWF0aGNhbHtZfX1cXG1hdGhjYWx7WH0iXSxbMSwwLCJcXG1hdGhjYWx7WH0iXSxbMCwxLCJcXG1hdGhjYWx7WH0iXSxbMSwxLCJcXG1hdGhjYWx7WH0gXFx1bmRlcnNldHtcXG1hdGhjYWx7WX19XFx0aW1lcyBcXG1hdGhjYWx7WH0iXSxbMCwyXSxbMiwzLCJcXERlbHRhIiwyXSxbMSwzLCJcXERlbHRhIl0sWzAsMV0sWzAsMywiIiwxLHsic3R5bGUiOnsibmFtZSI6ImNvcm5lciJ9fV1d
\begin{tikzcd}[ampersand replacement=\&,cramped]
	{\mathcal{L}_{\cY}\cX} \& {\cX} \\
	{\cX} \& {\cX \underset{\cY}\times \cX}
	\arrow[from=1-1, to=1-2]
	\arrow[from=1-1, to=2-1]
	\arrow["\lrcorner"{anchor=center, pos=0.125}, draw=none, from=1-1, to=2-2]
	\arrow["\Delta", from=1-2, to=2-2]
	\arrow["\Delta"', from=2-1, to=2-2]
\end{tikzcd}
\]
is naturally a groupoid over $\cX$, equipped with a morphism from its quotient stack $c: \cX/\mathcal{L}_{\cY}\cX \to \cX \times_{\cY}\cX$ to $\mathcal{X} \times_{\cY} \cX$. The first projection $q_1: \cX \times_{\cY} \cX \to \cX$ gives $\mathcal{L}_{\cY}\cX$ the structure of a group over $\cX$. Now any $E \in \QCoh(\cX)$ becomes a representation of this group by taking $(q_2 \circ c)^*E \in \QCoh(X/\mathcal{L}_{\cY}\cX)$.
The Atiyah class is then the infinitesimal action of the group $\mathcal{L}_{\cY}\cX$ on $(q_2 \circ c)^*E$. The co-Lie complex of this group is naturally identified with $L\Omega^1_{\cX/\cY}[1]$, and thus the infinitesimal action of $\mathcal{L}_{\cY}\cX$ on $E$ is of the form \eqref{eq: atiyah class}.
See Appendix \ref{subsection: d of inertia} for more details.

Our interpretation is that the Atiyah class arises naturally when a point \(x \in T[-1]\cX\) is viewed as a symmetric monoidal functor \(\QCoh(\cX) \to \QCoh(B\Ga^\vee) \simeq (\QCoh(\Ga),\star)\).
By definition, such functors are assignments of a quasi-coherent sheaf \(E\) to an endomorphism of \(E\) satisfying the Leibniz rule. 
Tautologically, there is a universal such functor over the mapping stack \(\Mapsu(B\Ga^\vee,\cX)\).
Proposition \ref{proposition: mapping stack by Atiyah class} below states that this functor sends an object \(E\) to its Atiyah class, viewed as an endomorphism with values in \(L\Omega^1_{\cX/A}[1]\).

Note that $\QCoh(B\Ga^\vee) \simeq \QCoh(\Ga)$ is dualizable, and hence for any prestack $\cY$,
\[ \QCoh(B\Ga^\vee) \otimes \QCoh(\cY) \to \QCoh(B\Ga^\vee \times \cY)\]
is an equivalence \cite[Chapter 3, Proposition 3.1.7]{GR17I}.
Hence, the Cartier duality equivalence of Lemma \ref{lemma: reps of Ga dual} gives an equivalence 
\[ \QCoh(B\Ga^\vee \times \cY) \simeq \QCoh(\Ga \times \cY).\]

\begin{proposition}\label{proposition: mapping stack by Atiyah class}
  Let \(\cX/\Spec A\) be a Tannakian prestack admitting a cotangent complex,  
  \[\Psi: \Mapsu(B\Ga^\vee, \cX) \overset{\sim}{\to} T[-1]\cX\] 
  be the equivalence of Theorem \ref{theorem: general tangent stack as mapping stack}, and let 
  \[ \widetilde{ev} = ev (1 \times \Psi^{-1}): B\Ga^\vee \times T[-1]\cX \to B\Ga^\vee \times \Mapsu(B\Ga^\vee,\cX) \to \cX.\]
  Then for $E \in \QCoh(\cX)$,
  \[\widetilde{ev}^*E \in \QCoh(B\Ga^\vee \times T[-1]\cX) \simeq \QCoh(\Ga \times T[-1]\cX)\]
  is identified under Cartier duality with $\pi^*E \in \QCoh(T[-1]\cX)$
  equipped with the endomorphism 
  \[
  \dd_E:
  % https://q.uiver.app/#q=WzAsMyxbMCwwLCJcXHBpXipFIl0sWzEsMCwiXFxwaV4qRSBcXG90aW1lcyBMXFxPbWVnYV4xX3tcXG1hdGhjYWx7WH0vQX1bMV0iXSxbMiwwLCJcXHBpXipFIl0sWzAsMSwiXFxwaV4qYXRfRSJdLFsxLDJdXQ==
\begin{tikzcd}
	{\pi^*E} & {\pi^*E \otimes L\Omega^1_{\cX/A}[1]} & {\pi^*E}
	\arrow["{\pi^*at_E}", from=1-1, to=1-2]
	\arrow[from=1-2, to=1-3]
\end{tikzcd}
\]

\end{proposition}
\begin{proof}

  By definition, $ev^*E$ is the sheaf whose pullback over $\eta: \cY \times B\Ga^\vee \to \cX$ is $\eta^*E$.
  If $z: \Spec A \to B\Ga^\vee$ is the quotient map, then the Cartier duality equivalence sends $M \in \QCoh(B\Ga^\vee \times \cY)$ to $z^*M \in \QCoh(\cY)$ equipped with the endomorphism given by the infinitesimal action of $\DD$.
  Let us view $\eta$ as a family of maps $B\Ga^\vee \to \mathcal{X}$ parameterized by $\mathcal{Y}$.
  By the naturality of the Atiyah class, there is a commutative square
  \begin{equation}\label{eq: atiyah class and evaluation map}
  % https://q.uiver.app/#q=WzAsNCxbMCwwLCJcXGV0YV4qRSJdLFsxLDAsIlxcZXRhXipFIFxcb3RpbWVzIFxcZXRhXipMXFxPbWVnYV4xX3tcXGNYL0F9WzFdIl0sWzEsMSwiXFxldGFeKkUgXFxvdGltZXMgTFxcT21lZ2FeMV97XFxjWSBcXHRpbWVzIEJcXEdhXlxcdmVlL0F9WzFdIl0sWzAsMSwiXFxldGFeKkUiXSxbMCwzLCI9IiwxXSxbMCwxLCJcXGV0YV4qYXRfRSIsMl0sWzMsMiwiYXRfe1xcZXRhXipFfSIsMl0sWzEsMiwiZFxcZXRhIl1d
\begin{tikzcd}
	{\eta^*E} & {\eta^*E \otimes \eta^*L\Omega^1_{\cX/A}[1]} \\
	{\eta^*E} & {\eta^*E \otimes L\Omega^1_{\cY \times B\Ga^\vee/A}[1]}
	\arrow["{\eta^*at_E}"', from=1-1, to=1-2]
	\arrow["{=}"{description}, from=1-1, to=2-1]
	\arrow["{d\eta}", from=1-2, to=2-2]
	\arrow["{at_{\eta^*E}}"', from=2-1, to=2-2]
\end{tikzcd}
.
\end{equation}
  Since the Atiyah class is by definition the infinitesimal action of the inertia group, the infinitesimal action of $\DD$ is given by the composite $(1\otimes \DD) \circ at_{\eta^*E} : \eta^*E \to \eta^* E$.

  Now $\Psi$ is defined by applying $\Mapsu(-,\cX)$ to the diagram 
  \[
% https://q.uiver.app/#q=WzAsNSxbMiwxLCJCXFxHYV5cXHZlZSJdLFswLDAsIlxcU3BlYyBBW1xcZXBzaWxvbl0vXFxlcHNpbG9uXjIiXSxbMiwwLCJcXFNwZWMgQSJdLFsxLDEsIlxcU3BlYyBBIl0sWzEsMCwiXFxHYV5cXHZlZSJdLFsxLDQsIlxcREQiXSxbNCwzXSxbNCwyXSxbMiwwXSxbMywwXV0=
\begin{tikzcd}
	{\Spec A[\epsilon]/\epsilon^2} & {\Ga^\vee} & {\Spec A} \\
	& {\Spec A} & {B\Ga^\vee}
	\arrow["\DD", from=1-1, to=1-2]
	\arrow[from=1-2, to=1-3]
	\arrow[from=1-2, to=2-2]
	\arrow[from=1-3, to=2-3]
	\arrow[from=2-2, to=2-3]
\end{tikzcd}.
  \]
  Thus the map $\Psi(\eta) :\cY \to T[-1]\cX$ is the map over $y = z^* \eta: \cY \to \cX$ classified by the composite
  \[ % https://q.uiver.app/#q=WzAsMyxbMCwwLCJ5XipMXFxPbWVnYV4xX3tcXGNYL0F9Il0sWzEsMCwiXFxPT197XFxjWX0gXFxvdGltZXMgTFxcT21lZ2FeMV97QlxcR2FeXFx2ZWUvQX1bMV0iXSxbMiwwLCJcXE9PX3tcXGNZfSJdLFswLDEsInleKmRcXGV0YSJdLFsxLDIsIlxcREQiXV0=
\begin{tikzcd}
	{y^*L\Omega^1_{\cX/A}} & {\OO_{\cY} \otimes L\Omega^1_{B\Ga^\vee/A}[1]} & {\OO_{\cY}}
	\arrow["{y^*d\eta}", from=1-1, to=1-2]
	\arrow["\DD", from=1-2, to=1-3]
\end{tikzcd}\]
  which is $z^*((1 \otimes \DD) d\eta$.
Taking the universal case $\eta = ev$ and applying $z^*$ to \eqref{eq: atiyah class and evaluation map}, we obtain the following commutative diagram:
\[
% https://q.uiver.app/#q=WzAsNixbMCwxLCJcXHBpXipFIl0sWzEsMSwiXFxwaV4qRSBcXG90aW1lcyBMXFxPbWVnYV4xX3tCXFxHYV5cXHZlZX1bMV0iXSxbMiwxLCJcXHBpXipFIl0sWzAsMCwiXFxwaV4qRSJdLFsxLDAsIlxccGleKkUgXFxvdGltZXMgTFxcT21lZ2FeMV97XFxjWC9BfVsxXSJdLFsyLDAsIlxccGleKkUiXSxbMCwxLCJhdF97XFxwaV4qRX0iXSxbMSwyLCIxIFxcb3RpbWVzIFxcREQiXSxbMyw0LCJcXHBpXiphdF9FIl0sWzQsMSwiZFxcZXRhIl0sWzMsMCwiPSIsMV0sWzQsNSwiMSBcXG90aW1lcyBcXFBzaShldikiXSxbNSwyLCI9IiwxXV0=
\begin{tikzcd}
	{\pi^*E} & {\pi^*E \otimes L\Omega^1_{\cX/A}[1]} & {\pi^*E} \\
	{\pi^*E} & {\pi^*E \otimes L\Omega^1_{B\Ga^\vee}[1]} & {\pi^*E}
	\arrow["{\pi^*at_E}", from=1-1, to=1-2]
	\arrow["{=}"{description}, from=1-1, to=2-1]
	\arrow["{1 \otimes \Psi(ev)}", from=1-2, to=1-3]
	\arrow["{d\eta}", from=1-2, to=2-2]
	\arrow["{=}"{description}, from=1-3, to=2-3]
	\arrow["{at_{\pi^*E}}", from=2-1, to=2-2]
	\arrow["{1 \otimes \DD}", from=2-2, to=2-3]
\end{tikzcd}
.
\]
Hence, the infinitesimal action of $\DD$ on $\widetilde{ev}^*E$ agrees with $\dd_E$, as desired.
\end{proof}

We are ready to prove Theorem \ref{maintheorem: restricted}.

\begin{theorem}\label{theorem: V is a pth power}
  Let \(A\) be an \(\Fp\)-algebra.
  If \(\cX/\Spec A\) is a Tannakian prestack admitting a \((-1)\)-connective cotangent complex, then for all \(E \in \QCoh(\cX)\), the following diagram commutes:
  \[% https://q.uiver.app/#q=WzAsNCxbMCwwLCJFIl0sWzAsMSwiRSBcXG90aW1lcyBMXFxPbWVnYV4xX3tcXG1hdGhjYWx7WH0vQX1bMV0iXSxbMSwwLCJFIFxcb3RpbWVzIChMXFxPbWVnYV4xX3tcXG1hdGhjYWx7WH0vQX1bMV0pXntcXG90aW1lcyBwfSJdLFsxLDEsIkUgXFxvdGltZXMgXFxTeW1ecChMXFxPbWVnYV4xX3tcXG1hdGhjYWx7WH0vQX1bMV0pIl0sWzAsMSwiYXRfRSJdLFswLDIsImF0X0Vee1xcb3RpbWVzIHB9IiwyXSxbMSwzLCIxIFxcb3RpbWVzIFYiXSxbMiwzXV0=
  \begin{tikzcd}[ampersand replacement=\&]
    E \& {E \otimes (L\Omega^1_{\cX/A}[1])^{\otimes p}} \\
    {E \otimes L\Omega^1_{\cX/A}[1]} \& {E \otimes \Sym^p(L\Omega^1_{\cX/A}[1])}
    \arrow["{at_E^{\otimes p}}"', from=1-1, to=1-2]
    \arrow["{at_E}", from=1-1, to=2-1]
    \arrow[from=1-2, to=2-2]
    \arrow["{1 \otimes V}", from=2-1, to=2-2]
  \end{tikzcd}
  .\]
\end{theorem}
\begin{proof}
  We begin with a universal calculation on \(B\Ga^\vee\).
  Consider the Cartier duality equivalence 
  \[ \Phi_{\Ga}: \QCoh(\Ga) \to \QCoh(B\Ga^\vee).\]
  Suppose $(E, \dd) \in \QCoh(\Ga)$ is a complex equipped with an endomorphism.
  Since $1 + \epsilon V$ is the Cartier dual of $1 + \epsilon F: \Ga \to \Ga$, 
  \[ 
    (1 + \epsilon V)^* \Phi_{\Ga}(E,\dd) = \Phi_{\Ga}\left((1 + \epsilon F)^* (E,\dd)\right) = \Phi_{\Ga}((E, \dd + \epsilon \dd^p)).
  \]

  Now let \(\cX/\Spec A\) be a Tannakian prestack admitting a (-1)-connective cotangent complex.
  By definition, \(V:L\Omega^1_{\cX/A}[1] \to \Sym^p(L\Omega^1_{\cX/A}[1])\) is induced by the action of \(1 + \epsilon V\) on \(\Mapsu(B\Ga^\vee,\cX)\).
  By Proposition \ref{proposition: mapping stack by Atiyah class}, 
  $(\pi^*E, at_E)$ is endomorphism associated to the Fourier-Mukai transform of $\widetilde{ev}^*E \in \QCoh(B\Ga^\vee \times T[-1]\cX)$.
  Thus 
  \[ (1 + \epsilon V)^*\Phi_{\Ga}(\pi^*E, at_E) = \Phi_{\Ga}(\pi^*E, at_E + \epsilon at_E^p),\]
  as desired.
\end{proof}

\subsection{Restricted Lie algebras}\label{subsection: restricted Lie algebras}

Theorem \ref{theorem: V is a pth power} is a derived analogue of one of the identities of a restricted Lie algebra.
To see how, we give a non-standard formulation of the definition of a restricted Lie algebra.
The results in this section may be read independently of the rest of the paper.

Let \(\Vect_k^\heartsuit\) denote the 1-category of vector spaces over \(k\).
Recall the definition of a restricted Lie algebra in \(\Vect_k^\heartsuit\):

\begin{definition}[see e.g.\ \cite{DG70}, II.7.3.3]
  \label{definition: restricted Lie algebra}
  A \emph{restricted Lie algebra} \(\mathfrak g\) over a field \(k\) of characteristic \(p > 0\) is a Lie algebra in \(\Vect_k^\heartsuit\) and a set map \(-^\pop: \g \to \g\) satisfying
  \begin{enumerate}
    \item \((\lambda x)^\pop = \lambda^p x^\pop\) for \(\lambda \in k, x \in \g\);
    \item \((x+y)^\pop = x^\pop + y^\pop + L(x,y)\) for all \(x,y \in \g\), where \(L\) is Jacobson's universal Lie polynomial;
    \item \(ad(x^\pop) = ad(x)^p\) for all \(x \in \g\).
  \end{enumerate}
\end{definition}
Fresse described restricted Lie algebras in terms of the Lie operad. Let \(\Lie(n)\) be the arity \(n\) operations in the Lie operad; the structure of an (operadic) Lie algebra is the structure of an algebra over the monad
\(\mathfrak g \mapsto \sum_{n\geq 1} (\Lie(n)\otimes \mathfrak g^{\otimes n})_{\Sigma_n}\).
\begin{theorem}\cite[Theorem 1.2.5]{Fre00}\label{theorem: restricted is pd Lie}
  The structure of a restricted Lie algebra on \(\g\) is equivalent to the structure of an algebra over the divided power monad
  \[ \mathfrak g \mapsto \sum_{n \geq 1} (\Lie(n) \otimes \mathfrak g^{\otimes n})^{\Sigma_n}.\]
\end{theorem}
Note the appearance of invariants instead of coinvariants.
The structure of a restricted Lie algebra is recovered from the above as follows: in characteristic \(p\), there is a \(\Sigma_p\)-invariant element
\begin{equation} \label{eq: w}
  w = \sum_{\sigma(1)=1} [\cdots[[X_{\sigma(1)},X_{\sigma(2)}],X_{\sigma(3)}],\cdots ,X_{\sigma(p)}] \in \Lie(p)^{\Sigma_p}.
\end{equation}
If \(act: (\Lie(p)\otimes \g^{\otimes p})^{\Sigma_p}\to \g\) is the action map of the divided power monad in arity \(p\) and \(x \in \g\), then \(w \otimes x^{\otimes p}\) is invariant under \(\Sigma_p\), and
\[ x^\pop  = act(w \otimes x^{\otimes p}) \in \g.\]
The advantage of the divided power monad formulation is that the nonlinear map \(x\mapsto x^\pop\) is replaced by a linear map
% https://q.uiver.app/#q=WzAsMyxbMCwwLCJcXEdhbW1hXnBcXGciXSxbMSwwLCIoXFxMaWUocClcXG90aW1lcyBcXGdee1xcb3RpbWVzIHB9KV57XFxTaWdtYV9wfSJdLFsyLDAsIlxcZyJdLFsxLDIsImFjdCJdLFswLDEsIncgXFxvdGltZXMgMSJdXQ==
\begin{equation}\label{eq: abstract V}
  \begin{tikzcd}[cramped]
    {\Gamma^p\g} & {(\Lie(p)\otimes \g^{\otimes p})^{\Sigma_p}} & \g
    \arrow["act", from=1-2, to=1-3]
    \arrow["{w \otimes 1}", from=1-1, to=1-2]
  \end{tikzcd}
  ,
\end{equation}
and the nonlinear behavior of \(x \mapsto x^\pop\) is entirely explained by the nonlinearity of the assignment \(x \mapsto x^{\otimes p}\).

While the divided power Lie monad gives a linear formulation of the structure of a restricted Lie algebra, much of its data is redundant. The following Corollary gives a more economical linear formulation of this structure:

\begin{corollary}\label{corollary: restricted Lie algebra definition}
  The following are equivalent for a Lie algebra \(\g\):
  \begin{enumerate}
    \item a restricted Lie structure on \(\g\);
    \item a compatible structure of an algebra over the divided power Lie monad;
    \item a linear map \(V: \Gamma^p \g \to \g\) such that the following diagrams commute:
      % https://q.uiver.app/#q=WzAsMyxbMSwwLCJcXEdhbW1hXnBcXGciXSxbMCwwLCJcXFN5bV5wXFxnIl0sWzEsMSwiXFxnIl0sWzEsMCwiTiJdLFswLDIsIlYiXSxbMSwyLCJ3IiwyXV0=
      % https://q.uiver.app/#q=WzAsNCxbMCwwLCJcXEdhbW1hXnAgXFxnIFxcb3RpbWVzIFxcZyJdLFsxLDAsIlxcZyBcXG90aW1lcyBcXGciXSxbMSwxLCJcXGciXSxbMCwxLCJcXGdee1xcb3RpbWVzIHArMX0iXSxbMCwxLCIxIFxcb3RpbWVzIFYiXSxbMSwyLCJbLSwtXSJdLFswLDNdLFszLDIsIltcXGNkb3RzW1stLC1dLC1dXFxjZG90c10iLDJdXQ==
      \[
        \begin{tikzcd}[cramped]
          {\Sym^p\g} & {\Gamma^p\g} \\
          & \g
          \arrow["N", from=1-1, to=1-2]
          \arrow["V", from=1-2, to=2-2]
          \arrow["w"', from=1-1, to=2-2]
        \end{tikzcd}
      \]
      and
      \begin{equation}\label{eq: restricted Lie V and bracket}
        \begin{tikzcd}[cramped]
          {\Gamma^p \g \otimes \g} & {\g \otimes \g} \\
          {\g^{\otimes p+1}} & \g
          \arrow["{V \otimes 1}", from=1-1, to=1-2]
          \arrow["{[-,-]}", from=1-2, to=2-2]
          \arrow[from=1-1, to=2-1]
          \arrow["{[\cdots [-,[-,-]]]}"', from=2-1, to=2-2]
        \end{tikzcd}
      \end{equation}
  \end{enumerate}
\end{corollary}
\begin{proof}
  i.\ \(\implies\) ii.\ is Fresse's theorem (Theorem \ref{theorem: restricted is pd Lie}). Now assume ii. If \(\g\) is given a compatible PD structure, then define \(V: \Gamma^p \g \to \g\) by
  \[ V:% https://q.uiver.app/#q=WzAsMyxbMCwwLCJcXEdhbW1hXnAgXFxnIl0sWzEsMCwiKFxcTGllKHApXFxvdGltZXMgXFxnXntcXG90aW1lcyBwfSlee1xcU2lnbWFfcH0iXSxbMiwwLCJcXGciXSxbMSwyXSxbMCwxLCJ3IFxcb3RpbWVzIDEiXV0=
    \begin{tikzcd}[cramped]
      {\Gamma^p \g} & {(\Lie(p)\otimes \g^{\otimes p})^{\Sigma_p}} & \g
      \arrow[from=1-2, to=1-3]
      \arrow["{w \otimes 1}", from=1-1, to=1-2]
    \end{tikzcd}.
  \]
  The first identity follows from that the Lie structure on \(\g\) is compatible with the divided power Lie structure via the norm map
  \[ N: (\Lie(n)\otimes \g^{\otimes n})_{\Sigma_n} \to (\Lie(n) \otimes \g^{\otimes n})^{\Sigma_n}.\]
  For the second identity, we compare to associative algebras.
  Let \(\Ass\) be the associative operad. Then \(\Ass(n)\) is a free \(\Sigma_n\)-module of rank 1 for all \(n \geq 1\). 
  Thus, for any associative algebra \(\mathcal A\), the norm map 
  \[ N : (\Ass(n) \otimes \mathcal{A}^{\otimes n})_{\Sigma_n} \overset{\sim}{\to} (\Ass(n) \otimes \mathcal{A}^{\otimes n})^{\Sigma_n}\]
  is an isomorphism, and both sides are naturally identified with \(\mathcal{A}^{\otimes n}\). 
  Under the natural injection \(\Lie(p) \to \Ass(p)\), \(w\) is sent to the associative word
  \[\sum_{\sigma \in \Sigma_p}X_{\sigma(1)}\cdots X_{\sigma(p)} = N(X_1X_2\cdots X_p)\in \Ass(p)\] 
  \cite[(1.2.8)]{Fre00}.
  Thus, if \(\mathcal A\) is an associative algebra and \(z \in \Gamma^p\mathcal A\), then \(w \otimes z = N(X_1\cdots X_p \otimes z)\), so the induced map
  % https://q.uiver.app/#q=WzAsNCxbMSwwLCIoXFxBc3MocClcXG90aW1lcyBcXG1hdGhjYWwgQV57XFxvdGltZXMgcH0pXntcXFNpZ21hX3B9Il0sWzIsMCwiKFxcQXNzKHApIFxcb3RpbWVzIFxcbWF0aGNhbCBBXntcXG90aW1lcyBwfSlfe1xcU2lnbWFfcH0iXSxbMCwwLCJcXEdhbW1hXnAgXFxtYXRoY2FsIEEiXSxbMywwLCJcXG1hdGhjYWwgQSJdLFsyLDAsIndcXG90aW1lcyAxIl0sWzAsMSwiTl57LTF9Il0sWzEsM11d
  \[
    \begin{tikzcd}[cramped]
      {\Gamma^p \mathcal A} & {(\Ass(p)\otimes \mathcal A^{\otimes p})^{\Sigma_p}} & {(\Ass(p) \otimes \mathcal A^{\otimes p})_{\Sigma_p}} & {\mathcal A}
      \arrow["{w\otimes 1}", from=1-1, to=1-2]
      \arrow["{N^{-1}}", from=1-2, to=1-3]
      \arrow[from=1-3, to=1-4]
    \end{tikzcd}
  \]
  is exactly the multiplication map \(\Gamma^p\mathcal A \to \mathcal A^{\otimes p} \to \mathcal A\).
  The adjoint action \(\g \to \End(\g)\) is a restricted Lie homomorphism, thus a divided power Lie homomorphism by Fresse's theorem, so for \(x \in \Gamma^p\g\), \(\ad(V(x)) = V(\Gamma^p\ad(x))\) where \(\Gamma^p\ad: \Gamma^p \g \to \Gamma^p \End(\g)\) is the \(p\)th divided power of the adjoint action.
  But \(\End(\g)\) is an associative algebra, so \(V(\Gamma^p\ad(x)) = \ad(x_{(1)}) \circ \ad(x_{(2)}) \circ \cdots \circ \ad(x_{(p)})\) (using Sweedler notation), as desired. Thus ii.\ \(\implies\) iii.

  Now assume iii. For \(x \in \g\), define
  \[ x^\pop = V(x^{\otimes p}).\]
  We check that \(x \mapsto x^\pop\) satisfies Definition \ref{definition: restricted Lie algebra}.
  This operation satisfies \((\lambda x)^\pop = V(\lambda^px^{\otimes p}) = \lambda^px^\pop\).
  If \(x,y \in \g\), then \((x+y)^{\otimes p} - x^{\otimes p} - y^{\otimes p}\) is in the image of \(N: \Sym^p\g \to \Gamma^p \g\). The first diagram implies
  \[ (x+y)^\pop - x^\pop - y^\pop = w(\widetilde{(x+y)^{\otimes p} - x^{\otimes p} - y^{\otimes p}})\]
  where \(\widetilde{-}\) means a lift through the norm map. Taking \(x\) and \(y\) to be generators of a free associative algebra implies that this Lie polynomial must be Jacobson's Lie polynomial \(L(x,y)\).
  % since the image of \(w\) in \(\Ass(p)\) is a basis for \(\Ass(p)^{\Sigma_p}\).
  Finally, the second diagram implies
  \[
    [x^\pop,y] = [V(x^{\otimes p}),y] = \ad(x)^p(y).
  \]
  We have defined a restricted Lie structure on \(\g\). Thus iii.\ \(\implies\) i.
\end{proof}

Theorem \ref{theorem: V is a pth power} is an analogue of \eqref{eq: restricted Lie V and bracket} in the derived setting. More precisely, taking \(E = L\Omega^1[1]\) proves that \(V\) satisfies the dual of \eqref{eq: restricted Lie V and bracket} up to homotopy. 
We expect the following more general statement: $L\Omega^1[1]$ has the structure of a partition Lie algebroid in the sense of Branter and Mathew \cite{BM25} with bracket given by the Atiyah class and $p$th power operation given by $V$; then $L\Omega^1[1]$ acts on $\QCoh$ with bracket given by the Atiyah class. 
Theorem \ref{theorem: V is a pth power} would then follow from that action being compatible not just with the bracket, but with all power operations on $L\Omega^1[1]$.
See §\ref{closing sub: partition Lie}.

\subsection{Classifying stacks}\label{subsection: V on classifying stack}

It is well-known that the classical Lie algebra of a group scheme in characteristic \(p\) is a restricted Lie algebra. Further, if \(G/\Spec A\) is a group scheme, then $L\Omega^1_{BG/A}[1] \simeq \ell_{G/A}$ is the co-Lie complex.
The Atiyah class of $\ell_{G/A}$
\[ at_{\ell_{G/A}}: \ell_{G/A} \to \ell_{G/A} \otimes \ell_{G/A}\]
is the derivative of the coadjoint action, and thus 
induces the dual of the usual Lie bracket on the classical Lie algebra $\g = \Lie G = \Hom(\ell_{G/A},A)$.
The map $V: \ell_{G/A} \to \Sym^p \ell_{G/A}$ also recovers the classical $p$th power operation on $\g$:

\begin{theorem}\label{theorem: V on classifying stack}
  Suppose that \(A\) has characteristic \(p\)
  and \(G/\Spec A\) is a group scheme.
  Then the operation \(V: \ell_G \to \Sym^p(\ell_G)\) acts on \(\g\) by
  \begin{align*}
    V^\vee: &\Gamma^p(\g) \to \g \\
    V^\vee: &\xi^{\otimes p} \mapsto \xi^\pop
  \end{align*}
  where \(\xi^\pop\) is the \(p\)th power of \(\xi\) as a vector field on \(G\).
\end{theorem}
\begin{proof}
  Since the Atiyah class is the infinitesimal action of the inertia group, the Atiyah class of $E \in \QCoh(BG)$ recovers the usual infinitesimal action of $\g$ on a representation of $G$.
  The classical \(p\)th power of vector fields is defined by the formula
  \[ \xi^\pop(f) = \underbrace{\xi(\cdots \xi(\xi}_{p}(f))\cdots)\]
  for \(\xi \in \g\) and \(f \in H^0(\OO_G)\).
  This is exactly taking the \(p\)th power of the Atiyah class of the representation \(\OO_G\), then contracting with \(\xi^{\otimes p}\).
  By Theorem \ref{theorem: V is a pth power},
  \[ (1 \otimes V)at_{\OO_G} \sim at_{\OO_G}^p: \OO_G \to \OO_G \otimes \Sym^p \ell_{G/A},\]
  so dually \(V^\vee (\xi^{\otimes p})\) acts on \(\OO_G\) by \(\xi^\pop\).
  Since the action of $\g$ on \(\OO_G\) is faithful, we see \(V^\vee(\xi^{\otimes p}) = \xi^{\pop}\) for \(\xi \in \pi_0(\ell_G)\), as desired.
\end{proof}

Theorem \ref{theorem: V on classifying stack} allows for the computation of \(V\) on Hodge classes pulled back from classifying stacks.

\begin{example}[Projective space]
  Let \(X = \mathbb P^n/\Fp\). The Hodge cohomology ring of \(\mathbb P^n\) is
  \[ \bigoplus_{i,j} H^i(\mathbb P^n, \Omega^j_{\mathbb P^n}) = \Fp[c]/c^{n+1}\]
  where \(c \in H^1(\mathbb P^n, \Omega^1_{\mathbb P^n})\) is the Chern class of \(\mathcal O(1)\).
  If \(\rho: X \to B\mathbb G_m\) is the map classifying the line bundle \(\mathcal O(1)\),
  then the pullback map
  \[
    \rho^*: H^\bullet(B\mathbb G_m, L\Omega^\bullet_{B\mathbb G_m}) \to H^\bullet(\mathbb P^n, L\Omega^\bullet_{\mathbb P^n})
  \]
  is identified with the quotient map \(k[c] \to k[c]/c^{n+1}\).
  Since \(c\) is dual to the generator \(\partial\) of the Lie algebra of \(\mathbb G_m\), and \(\partial^\pop = \partial\),
  Theorem \ref{theorem: V on classifying stack} implies that \(V(c) = c^p\) on \(B\mathbb G_m\).
  Hence \(V(c) = c^p\) also on \(\mathbb P^n\).
\end{example}

\begin{example}[Roots of unity]
  \label{example: mup}
  Let \(X = B\mu_p\), the classifying stack of the \(p\)th roots of unity. As \(\mu_p\) is the kernel of multiplication by \(p\) on \(\mathbb G_m\), the co-Lie complex of \(\mu_p/\Spec \Z\) is
  % https://q.uiver.app/#q=WzAsMixbMCwwLCJrIl0sWzEsMCwiayJdLFswLDEsInAiXV0=
  \[\ell_{\mu_p/\Z} = \left(\begin{tikzcd}
    \Z & \Z
    \arrow["p", from=1-1, to=1-2]
  \end{tikzcd}\right).\]
  Thus \(H^\ast(B\mu_p/\Fp, L\Omega^1_{B\mu_p/\Fp})\) has basis a class \(d\) in cohomological degree \(0\) and a class \(c\) in cohomological degree 1, and \(\Bock(d)=c\). Moreover, \(H^\ast(B\mu_p/\Fp, L\Omega^1_{B\mu_p/\Fp}) = \Fp[c,d]\) is the graded polynomial algebra on \(c\) and \(d\) \cite[Proposition 4.4]{ABM21}.
  %As \(c\) is pulled back from the inclusion \(\mu_p \to \mathbb G_m\), we find \(V(c) = c^p\).

  Antieau-Bhatt-Mathew calculated that \(d_p(d)\) is a unit multiple of \(c^p\) \cite[Proposition 4.5]{ABM21}, and thus the HKR spectral sequence for \(B\mu_p\) does not degenerate. Here is how to recover their calculation from our work, following a suggestion of Vologodsky.
  Since \(B\mu_p\) is a Tannakian stack, Theorem \ref{theorem: general tangent stack as mapping stack} implies \(T[-1]B\mu_p = \Maps(B\Ga^\vee, B\mu_p)\).
  Since \(\mu_p\) is abelian, \(\Maps(B\Ga^\vee,B\mu_p) = B\mu_p \times \Homs(\Ga^\vee,\mu_p)\).
  By Cartier duality,
  \[ \Homs(\Ga^\vee, \mu_p) = \Homs(\Z/p\Z, \Ga),\]
  which fits into a pullback square 
  % https://q.uiver.app/#q=WzAsNCxbMCwwLCJcXEhvbXMoXFxaL3BcXFosXFxHYSkiXSxbMCwxLCJcXEdhIl0sWzEsMSwiXFxHYSJdLFsxLDAsIjAiXSxbMSwyLCJwIl0sWzMsMl0sWzAsMV0sWzAsM10sWzAsMiwiIiwxLHsic3R5bGUiOnsibmFtZSI6ImNvcm5lciJ9fV1d
\[\begin{tikzcd}[ampersand replacement=\&]
	{\Homs(\Z/p\Z,\Ga)} \& 0 \\
	\Ga \& \Ga
	\arrow[from=1-1, to=1-2]
	\arrow[from=1-1, to=2-1]
	\arrow["\lrcorner"{anchor=center, pos=0.125}, draw=none, from=1-1, to=2-2]
	\arrow[from=1-2, to=2-2]
	\arrow["p", from=2-1, to=2-2]
\end{tikzcd}\]
  Thus \(\Homs(\Z/p\Z, \mathbb G_a) = \Spec(A[c] \otimes_{A[s]} A)\)
  where \(A[c]\) is an \(A[s]\)-algebra via \(s \mapsto pc\).
  If \(A\) is an \(\Fp\)-algebra, then 
  \[ A[c] \otimes_{A[s]} A \simeq A[c] \otimes_A (A \otimes_{A[s]} A) \simeq A[c,d]\]
  where \(d\) is a generator for \(\mathrm{Tor}^1_{A[s]}(A,A) \cong (s)/(s^2)\), sitting in homological degree 1.

  Now \(1 + \epsilon V\) is the Cartier dual of the automorphism \(1 + \epsilon F\) of \(\Ga\).
  Thus 
  \((1 + \epsilon V)(c) = c + \epsilon c^p\),
  while \((1 +\epsilon V)(s) = s + \epsilon s^p\).
  Thus \((1 + \epsilon V)(d) = d\) since \(s^p\) is zero in \((s)/(s^2)\).

  We conclude \(V(c) = c^p\) and \(V(d) = 0\).
  Thus, by Theorem \ref{theorem: formula for differential}, 
  \[ d_p(d) = [V,\Bock](d) = c^p.\]
\end{example}

\section{Future work}\label{section: closing}

We close this paper with some remarks on what is left undone.

\subsection{Partition Lie structure}\label{closing sub: partition Lie}

In §\ref{subsection: atiyah and shifted tangent bundle}-\ref{subsection: restricted Lie algebras}, it was shown that the map \(V: L\Omega^1[1] \to L\Omega^p[p]\) is a \(p\)th power for the Atiyah class in the sense that \((1 \otimes V)at_E \sim at_E^p\) for almost connective quasi-coherent sheaves \(E\). This is only one of the two identities in Corollary \ref{corollary: restricted Lie algebra definition} which define a classical restricted Lie algebra; the other involves compatibility with the norm map \(N: \Sym^p(L\Omega^1[1]) \to \Gamma^p(L\Omega^1[1])\).
It is not clear to the author whether that norm identity even holds in the general context of derived geometry, or whether a more complicated identity holds instead.

Indeed, a more natural way to define \(V\) would be to proceed as follows. The loop space \(\mathcal{L}X \to X\) of a stack \(X\) is a (derived, higher) groupoid, and the Atiyah bracket is the Lie structure on its Lie algebra \(T[-1]X\). The Lie algebra of a groupoid should have the structure of a \emph{partition Lie algebroid}, a concept that has only very recently been defined \cite{Fu24, BMN25}.
Partition Lie algebras (algebroids where the base is the spectrum of a field) were defined in \cite{BM25} and used to classify deformation problems over a field. Over a general base, Gaitsgory and Rozenblyum explained that one must use algebroids instead of algebras \cite[Chapter 7]{GR17II}. The algebroid \(T[-1]X\) should be tangent to deformations of the diagonal \(X \to X \times X\), and thus admit a partition Lie algebroid structure. The definition of \(V\) and its identities would then be subsumed into the larger structure of a monad acting on \(T[-1]X\). These methods are out of reach of the current paper, and I hope to return to them in future work.

\subsection{Adams operations}
This paper only treats the pages in the Hochschild-Kostant-Rosenberg spectral sequence up to page \(p\). This is in some sense as far as we can go. If one knows that the differentials \(d_r\) in a spectral sequence \(E_r\) are zero for \(r < R\) and one has a formula for \(d_R\), then there is most likely not a general formula for \(E_{R+1} = \ker d_R / \im d_R\), and the train grinds to a halt.

However, we can say more. Many of the differentials of the Hochschild-Kostant-Rosenberg spectral sequence are zero, thanks to the Adams operations
\[ \psi_m: \fhkr \HHs(X/A) \to \fhkr \HHs(X/A)\]
induced by multiplication by \(m\) on \(\sone\).
The idea of using the Adams operations to study the HKR filtration goes back at least to Loday, who used them to split the HKR filtration in characteristic zero \cite[§4.5-6]{Lod92}.

\begin{theorem}\label{theorem: adams ops and differentials}
  Let \(A\) be an \(\Fp\)-algebra and \(X/A\) a scheme.
  Then the differentials \(d_r\) in the Hochschild-Kostant-Rosenberg spectral sequence are zero unless \(r \equiv 1 \mod p-1\).
\end{theorem}
\begin{proof}
  The $m$th Adams operation \(\psi_m\) acts by \(m^i\) on \(\gr^i \HHs(X/A)\).
  \cite[Theorem 3.5(e)]{Lod89}.
  Picking \(m\) to be a primitive root modulo \(p\) shows that \(\psi_m\) has different weights on \(\gr^i \HHs(X)\) and \(\gr^j \HHs(X)\) unless \(i \equiv j \mod p-1\).
  As the differentials \(d_r\) intertwine the Adams operations, we see \(d_r = 0\) unless \(p-1 \mid 1-r\).
\end{proof}

Our knowledge of the differentials beyond page \(p\) is rather limited. One example with nontrivial differentials beyond page \(p\) is \(X = B\mu_{p^n}\). Following the arguments of \cite[Proposition 4.5]{ABM21}, it follows that the HKR spectral sequence for \(B\mu_{p^n}\) has a nontrivial differential exactly on page \(p^n\).
This mirrors a similar phenomenon in formal group theory: two one-dimensional formal group laws over an \(\Fp\)-algebra can always be made isomorphic up to degree a power of \(p\) \cite[Lemma 1.6.6]{Haz78}.
As we have seen that deformations of one-dimensional formal groups are closely related to HKR, we may pose

\begin{conjecture}
  If \(A\) is an \(\Fp\)-algebra and \(X/A\) is a scheme,
  then the differentials \(d_r\) in the Hochschild-Kostant-Rosenberg spectral sequence for \(X/A\) are zero unless \(r\) is a power of \(p\).
\end{conjecture}

\subsection{More examples of non-degeneration}

For each prime \(p\), Antieau, Bhatt, and Mathew constructed a \(2p\)-dimensional smooth projective variety \(X\) over \(\bar{\mathbb{F}}_p\) such that the HKR spectral sequence has a nonzero differential on page \(p\). On the other hand, if \(\ch k = p > 0\) and \(X/k\) is smooth, then the HKR spectral sequence degenerates if \(\dim X < p\) \cite{Yek02}\footnote{Note that Theorem \ref{maintheorem: differential} part i. implies degeneration when \(\dim X < p\): when \(r < p\), the Theorem implies \(d_r= 0\); when \(r \geq p\), the differentials \(d_r\) are zero for dimension reasons.} and if \(\dim X = p\) and \(X/k\) is proper \cite{AV20}.

\begin{question}
  If \(\ch k = p > 0\), do there exist smooth proper varieties \(X/k\) of dimension \(p+1 \leq \dim X \leq 2p-1\) such that the HKR spectral sequence does not degenerate?
\end{question}

We do not know of such examples, but offer some comments on the problem.
Assume \(X/k\) is smooth and proper. If \(\dim X = p+1\) and \(p>2\), then by Theorem \ref{theorem: adams ops and differentials}, the only possibly nonzero differentials can be on page \(p\).
Following Antieau and Vezzosi \cite[Proposition 3.9]{AV20}, duality in the Hochschild complex implies that if the differential on page \(p\) is nonzero, then \(d_p: H^1(X,\Omega^1_{X/k}) \to H^{p+1}(X, \Omega^p_{X/k})\) is nonzero. Much of the difficulty in constructing examples seems to come from ensuring torsion in \(H^2(\Omega^1)\) and \(H^{p+1}(\Omega^p)\) of a lift.

\subsection{Beyond the multiplicative group}

The filtered circle is the classifying stack of the Cartier dual of the rescaled multiplicative group \(\Ghat_\lambda\).
In \cite[§6.3]{MRT22} and \cite{Mou24}, it is pointed out that for any one-dimensional formal Lie group \(E\), one can form a ``filtered \(E\)-circle'' \(\sone_E = BG_E\), where \(G_E \to \mathbb A^1/\mathbb G_m\) is the Cartier dual of the filtered deformation of \(E\) to the normal cone. Then one can define \(\HH^E(\Spec R)\) to be the filtered ring of functions on the mapping stack of \(E\)-loops
\[ \mathcal{L}_E \Spec R = \Mapsu(BG_E, \Spec R).\]
The special fiber of \(G_E\) over \(B\mathbb G_m\) is always \(\Gahat^\vee\), and thus the associated graded of \(\HH^G\) is also the ring of functions on the shifted tangent bundle \(T[-1]\Spec R\).
Thus, one also has an \(E\)-HKR spectral sequence converging to \(\HH^E(\Spec R)\).

If \(k\) is an algebraically closed field of characteristic \(p\), then one-dimensional formal groups over \(k\) are classified by their height \cite[§19.4]{Haz78}.
If \(h < \infty\) is the height of a one-dimensional formal Lie group \(E\), then \(E\) has a group law of the form
\[ v,w \mapsto v+w + \frac{v^{p^h} + w^{p^h}-(v+w)^{p^h}}{p}  \mod (v,w)^{p^h+1}\]
\cite[(18.3.5)]{Haz78}, and thus the associated deformed group law \(E_\lambda\) will be trivial modulo \(\lambda^{p^h-1}\). Following the proof of Theorem \ref{theorem: initial-split}, we find that in characteristic \(p\), the HKR filtration on \(\HH^E(\Spec R)\) is split to order \(p^h-2\), and in the associated spectral sequence, the differentials are zero before page \(p^h\). We expect that by following the ideas of §\ref{subsection: filtered circle as deformation}, there is a formula $d_{p^h}$ of the form $[\Bock, V^h]$ where $V^h: L\Omega^1[1] \to L\Omega^{p^h}[p^h]$ is the $h$th iterate of $V$.

\subsection{de Rham-to-cyclic spectral sequences}

Theorem \ref{maintheorem: differential} also has implications for the Rham-to-cyclic spectral sequence. 
Let $X$ be a classical smooth variety over a perfect field $k$ of characteristic $p$.
By construction, the filtered loop space $\mathcal{L}_{\mathrm{fil}}X = \Mapsu(\filS, X)$ carries an action of the group scheme $\filS$,
and so $\HH_{\mathrm{fil}}(X/k)$ carries an action of $\filS$.
Taking the homotopy fixed points or Tate fixed points of this action gives filtered complexes 
\[\mathrm{HC}_{\mathrm{fil}}^-(X/k) := \HH_{\mathrm{fil}}(X/k)^{h\filS}\]
\[\mathrm{HP}_{\mathrm{fil}}(X/k) := \HH_{\mathrm{fil}}(X/k)^{t\filS}.\]
Since $X/k$ is classical and smooth, the spectral sequence associated to these filtered complexes can be identified with the usual de Rham-to-cyclic spectral sequences \cite[Example 6.3.8]{Rak20}\cite[Theorem 5.4.1(d)]{MRT22}.
Hence, the differentials in the spectral sequences associated to $\mathrm{HC}^-$ and $\mathrm{HP}$ can in principle be understood as the fixed points of the differentials on the Hochschild-Kostant-Rosenberg spectral sequence. 
We do not do so here, but below we show that $V$ descends to an operation on homotopy fixed points.

The category of $(\filS)^{gr}$-representations can be identified with strict modules over 
\[ \Lambda = R\Gamma(B\Gacheck,\OO)^*\]
\cite[Proposition 4.2.3]{MRT22}.
Under this identification, 
the de Rham differential on corresponds to the action of a generator of $H^{-1}(\Lambda) = H^1(B\Gacheck,\OO)^*$.
\begin{proposition}
  The de Rham differential on $\gr \HH(X/k)$ commutes with $V$. 
  More precisely, the map $V: \gr \HH(X/k) \to \gr \HH(X/k)$ has the structure of a $\Lambda$-linear map.
\end{proposition}
\begin{proof}
  Let $\alpha$ be the automorphism of $B\Gacheck \otimes \Fp[\epsilon]/\epsilon^2$ induced by the Cartier dual of $(1 + \epsilon F): \Gahat \to \Gahat$.
  The map $(1+\epsilon V) :\gr \HH(X/k)[\epsilon]/\epsilon^2 \to \gr \HH(X/k)[\epsilon]/\epsilon^2$ is induced by $\alpha^*: \Mapsu(B\Gacheck, X) \to \Mapsu(B\Gacheck,X)$.
  Let $\alpha_\Lambda: \Lambda[\epsilon]/\epsilon^2 \to \Lambda[\epsilon]/\epsilon^2$ be the natural map.
  Since $\alpha$ is a group homomorphism, $\alpha^*$ is naturally a homomorphism
  \[  
  \gr\HH(X/k)[\epsilon]/\epsilon^2 \to (\alpha_\Lambda)_* \gr \HH(X/k)[\epsilon]/\epsilon^2
  \]
  of $\Lambda$-modules.
  Thus, it suffices to show an equivalence $\alpha_\Lambda \simeq \id$ of homomorphisms of associative rings.
  
  By \cite[preceding Lemma 6.5]{MM25}, for a formal Lie group $F$ there is a natural identification 
  \[ R\Gamma(BF^\vee, \OO) \simeq \Gamma^\star(\Hom(F^\vee,\Ga)[-1])\]
  of augmented commutative coalgebras,
  where $\Gamma^\star(-)$ means the divided power symmetric algebra.
  Let $i: \Gacheck \to \Ga$ be the generator of $\Hom(\Gacheck, \Ga)$.
  Since the Verschiebung $V$ is natural, $i \circ V_{\Gacheck} = V_{\Ga} \circ i$.
  But $V_{\Ga} = 0$. This proves the claim.
\end{proof}

According to \cite[Remark 5.3.4]{MRT22}, since $X/k$ is classical and smooth,
\[ \gr^i \HHfil(X/k)^{h\filS} = \begin{cases} \dR_{X/k}^{\geq i}[2i] & i \geq 0 \\
  \dR_{X/k}[2i] & i < 0 
\end{cases}
\]
Thus $V$ induces a map 
\[ V: \dR_{X/k}^{\geq i} \to \dR_{X/k}^{\geq (i+p-1)}[2p-2]\]
and in particular a map $\dR_{X/k} \to \dR_{X/k}[2p-2]$.

\appendix

\section{The Atiyah class for prestacks}\label{appendix: atiyah}

In this appendix, we review the cotangent complex of a morphism of prestacks and the Atiyah class of an almost-connective quasi-coherent sheaf.
The results as formulated in this section are all due to Lurie \cite{Lurie-SAG} and are replicated here for the readers' convenience.

\subsection{Cotangent complex}\label{subsection: cotangent complex for stacks}

We first recall Lurie's treatment of the cotangent complex of a morphism of prestacks \cite[§17.2]{Lurie-SAG}.
The cotangent complex will represent the functor of derivations. A derivation of an animated ring \(R\) into a connective \(R\)-module \(M\) will be a section of the projection \(R \oplus M \to R\), where \(R \oplus M\) is the square-zero extension. %However, this square-zero extension of animated rings only makes sense if \(M\) is a connective module; in the stacky setting one must work around this.

We now define the functor of relative derivations at a point of a prestack, following \cite[Definition 17.2.4.2]{Lurie-SAG}. 
%Let \(\Sp\) be the \(\infty\)-category of spectra.
Let \(f: X \to Y\) be a morphism of prestacks. If \(x: \Spec R \to X\), let \(\Der_{x/Y}: \Mod_R^{\leq 0} \to \Spc\) be the functor sending \(M \in \Mod_R^{\leq 0}\) to
\begin{equation}\label{eq: local derivation functor}
  \Der_{x/Y}(M) = \fib(X(R \oplus M) \to \{x\} \times_{Y(R)} Y(R \oplus M)),
\end{equation}
where we consider \(x\) as a point of \(X(R)\) mapping to \(Y(R)\).
\begin{remark}
  Formula \eqref{eq: local derivation functor} should be interpreted as ``derivations \(\OO_X \to M\) which are \(\OO_Y\)-linear.'' More geometrically, we can think of \(\Der_{x/Y}\) as the space of vertical tangent vectors at \(x\).
\end{remark}
In \cite[§17.2.3]{Lurie-SAG}, Lurie shows that the functors \(\Der_x\) fit together into a functor
\[ \Der_{X/Y} = \lim_{x: \Spec R \to X} \Der_{x/Y}: \QCoh(X)^{\leq 0} \to \Spc.\]

We would like to define the cotangent complex as a quasi-coherent sheaf representing the functor of derivations. However, the functor of derivations is only defined on connective sheaves.
Thus, to define the cotangent complex, we need to first consider extending functors from the full subcategory of connective sheaves.

Let \(C\) be a stable \(\infty\)-category. A functor \(F: C \to \Sp\) is \emph{reduced} if it preserves final objects, and \emph{excisive} if it carries pushout squares to pullback squares.
Let \(\Excast(C,\Sp)\) be the full subcategory of functors \(C \to \Sp\) which are reduced and excisive.

\begin{lemma}\cite[Lemma 17.2.1.2]{Lurie-SAG}
  \label{lemma: restrict excisive}
  Let \(C\) be a stable \(\infty\)-category equipped with a bounded-above \(t\)-structure.
  Then the restriction functor
  \[ \Excast(C, \Spc) \to \Excast(C^{\leq 0},\Spc)\]
  is a trivial Kan fibration.
\end{lemma}

Thus, reduced excisive functors on \(\QCoh(X)\) are recovered from their restriction to \(\QCoh(X)^{\leq 0}\). It follows \cite[Example 17.2.1.4]{Lurie-SAG} that sending \(M \in \QCoh(X)^{-}\) to \(\Hom(M,-)\) gives a fully faithful embedding
\[ (\QCoh(X)^{-})^{op} \to \Excast(\QCoh(X)^{\leq 0},\Spc).\]
That is, an almost connective quasi-coherent sheaf can be recovered from its values on connective objects.

\begin{definition}[\cite{Lurie-SAG}, 17.2.4.2]
  The morphism \(f: X \to Y\) of prestacks \emph{admits a cotangent complex} if and only if the functor \(\Der_{X/Y}: \QCoh(X)^{\leq 0} \to \Spc\) is almost represented by some \(L\Omega^1_{X/Y} \in \QCoh(X)^{-}\).
\end{definition}

\begin{remark}
  For \(f: X \to Y\) to admit a cotangent complex, the functor \(\Der_{X/Y}\) must be reduced and excisive. This is not automatic. There are various criteria to determine whether a morphism admits a cotangent complex, see e.g.\ \cite[§17.2.3-4]{Lurie-SAG}.
\end{remark}

\begin{lemma}[\cite{Lurie-SAG},Proposition 17.2.5.2]
  % https://q.uiver.app/#q=WzAsMyxbMCwxLCJYIl0sWzIsMSwiWiJdLFsxLDAsIlkiXSxbMCwyLCJmIl0sWzIsMSwiZyJdLFswLDEsImgiLDJdXQ==
  Let \[
    \begin{tikzcd}
      & Y \\
      X && Z
      \arrow["g", from=1-2, to=2-3]
      \arrow["f", from=2-1, to=1-2]
      \arrow["h"', from=2-1, to=2-3]
    \end{tikzcd}
  \]
  be a commutative diagram of prestacks. If \(g\) and \(h\) admit cotangent complexes, so does \(f\), and there is a canonical fiber sequence \(f^*L\Omega^1_{Y/Z} \to L\Omega^1_{X/Z} \to L\Omega^1_{X/Y}\).
\end{lemma}
\begin{remark}\label{remark: functoriality of cotangent complex}
  Note that \(f^*L\Omega^1_{Y/Z} \to L\Omega^1_{X/Z}\) is defined by taking the fiber of the horizontal morphisms in the commutative square
  % https://q.uiver.app/#q=WzAsNCxbMCwwLCJYKFIgXFxvcGx1cyBNKSJdLFsxLDAsIlxce3hcXH1cXHRpbWVzX3taKFIpfVooUlxcb3BsdXMgTSkiXSxbMSwxLCJcXHtmKHgpXFx9XFx0aW1lc197WihSKX1aKFIgXFxvcGx1cyBNKSJdLFswLDEsIlkoUiBcXG9wbHVzIE0pIl0sWzAsMywiZiJdLFszLDJdLFsxLDJdLFswLDFdXQ==
  \[
    \begin{tikzcd}
      {X(R \oplus M)} & {\{x\}\times_{Z(R)}Z(R\oplus M)} \\
      {Y(R \oplus M)} & {\{f(x)\}\times_{Z(R)}Z(R \oplus M)}
      \arrow[from=1-1, to=1-2]
      \arrow["f", from=1-1, to=2-1]
      \arrow[from=1-2, to=2-2]
      \arrow[from=2-1, to=2-2]
    \end{tikzcd},
  \]
  which yields a transformation \(\Der_{x/Z} \to \Der_{f(x)/Z}\). Taking the limit of these morphisms over \(x: \Spec R \to X\) gives a morphism \(\Hom(L\Omega^1_{X/Z} ,-) \to \Hom(f^*L\Omega^1_{Y/Z},-)\).
\end{remark}

\begin{lemma}[\cite{Lurie-SAG}, 17.2.4.6]\label{lemma: cotangent complex and pullback}
  Let
  \begin{equation}\label{eq: square for cotangent complex}
    % https://q.uiver.app/#q=WzAsNCxbMCwwLCJYJyJdLFswLDEsIlknIl0sWzEsMCwiWCJdLFsxLDEsIlkiXSxbMSwzXSxbMCwyLCJnIl0sWzAsMSwiZiciLDJdLFsyLDMsImYiLDJdXQ==
    \begin{tikzcd}
      {X'} & X \\
      {Y'} & Y
      \arrow["g", from=1-1, to=1-2]
      \arrow["{f'}"', from=1-1, to=2-1]
      \arrow["f"', from=1-2, to=2-2]
      \arrow[from=2-1, to=2-2]
    \end{tikzcd}
  \end{equation}
  be a commutative diagram in prestacks.
  \begin{itemize}
    \item Suppose \(f\) and \(f'\) admit cotangent complexes. Then there is a canonical morphism \(g^*L\Omega^1_{X/Y} \to L\Omega^1_{X'/Y'}\).
    \item Suppose \eqref{eq: square for cotangent complex} is a pullback square and \(f\) admits a cotangent complex. Then \(f'\) admits a cotangent complex and the canonical morphism \(g^*L\Omega^1_{X/Y} \to L\Omega^1_{X'/Y'}\) is an equivalence.
  \end{itemize}
\end{lemma}

% It is not true in general that if $G \to X$ is a group prestack admitting a cotangent complex, then $BG \to X$ admits a cotangent complex.
%  Nick Rozenblyum supplied the counterexample $G = \Omega_0 \mathbb{A}^1 = 0 \times_{\mathbb{A}^1} 0$}.

Suppose that \(Z\) is a prestack and \(\mathcal{M} \in \QCoh(Z)^{\leq 0}\).
If \(z: \Spec R \to Z\), then we have the square-zero extension \(R \oplus z^*\mathcal{M}\) of \(R\), and thus the prestack \(\Spec(R \oplus z^*\mathcal{M})\).
\begin{definition}
  If \(\mathcal{M} \in \QCoh(Z)^{\leq 0}\),
  then the trivial square-zero extension \(Z^{\mathcal{M}}\) of \(Z\) by \(\mathcal{M}\) is the relative spectrum
  \begin{equation*}
    Z^{\mathcal{M}} = \underline{\Spec}(\OO_Z \oplus \mathcal{M}) = \colim{x: \Spec R \to Z} \Spec(R \oplus x^*\mathcal{M}).
  \end{equation*}
\end{definition}
The zero map \(\OO_Z \oplus \mathcal{M} \to \OO_Z\) defines a morphism \(Z \to Z^{\mathcal{M}}\).
Now suppose that \(X/Y\) admits a cotangent complex.
By definition, if \(x: \Spec R \to X\), then the space of sections of \(\Spec R \to \Spec(R \oplus x^*\mathcal{M})\) together with a trivialization over \(Y\) is equivalent to \(\Der_{x/Y}(x^*\mathcal{M})\).
Expressing $Z$ as a colimit of affine schemes shows that if $z: Z \to X$, then 
\[\Maps(z^*L\Omega^1_{X/Y},\mathcal{M})\] is equivalent to the space of lifts 
\[
% https://q.uiver.app/#q=WzAsNCxbMCwwLCJaIl0sWzAsMSwiWl5NIl0sWzEsMCwiWCJdLFsxLDEsIlkiXSxbMSwzXSxbMCwxXSxbMiwzXSxbMCwyLCJ6Il0sWzEsMiwiIiwxLHsic3R5bGUiOnsiYm9keSI6eyJuYW1lIjoiZGFzaGVkIn19fV1d
\begin{tikzcd}
	Z & X \\
	{Z^M} & Y
	\arrow["z", from=1-1, to=1-2]
	\arrow[from=1-1, to=2-1]
	\arrow[from=1-2, to=2-2]
	\arrow[dashed, from=2-1, to=1-2]
	\arrow[from=2-1, to=2-2]
\end{tikzcd}
\]
where the map $Z^M \to Y$ is the composite $Z^M \to Z \to X \to Y$.

\subsection{First-order deformations of quasi-coherent sheaves}%\label{subsection: definition of Atiyah}

Let $X$ be a prestack. In this section, we discuss the deformation theory of quasi-coherent sheaves on $X$ along a (possibly non-split) square-zero extension of $X$.
%Let \(f: X \to Y\) be a morphism of prestacks. In this section, we define the Atiyah class of an almost-connective quasi-coherent sheaf \(E \in \QCoh(X)^{-}\). The approach is first to define the \emph{Atiyah transformation} \[ at_E: \Der_{X/Y} \to \Hom(E, -[1]\otimes E)\] which obstructs the prolongation of \(E\) along square-zero extensions. A version of the Yoneda lemma then defines the Atiyah class when \(f\) admits a cotangent complex. We follow \cite[§19.2.2]{Lurie-SAG}. 

To begin with, suppose that
\begin{equation}\label{eq: pullback of algebras}
% https://q.uiver.app/#q=WzAsNCxbMCwwLCJBIl0sWzAsMSwiQV8wIl0sWzEsMCwiQV8xIl0sWzEsMSwiQV97MDF9Il0sWzAsMV0sWzEsM10sWzIsM10sWzAsMl0sWzAsMywiIiwxLHsic3R5bGUiOnsibmFtZSI6ImNvcm5lciJ9fV1d
\begin{tikzcd}[ampersand replacement=\&,cramped]
	A \& {A_1} \\
	{A_0} \& {A_{01}}
	\arrow[from=1-1, to=1-2]
	\arrow[from=1-1, to=2-1]
	\arrow["\lrcorner"{anchor=center, pos=0.125}, draw=none, from=1-1, to=2-2]
	\arrow[from=1-2, to=2-2]
	\arrow[from=2-1, to=2-2]
\end{tikzcd}
\end{equation}
is a pullback diagram of algebras in $\QCoh(X)^{\leq 0}$. Extension of scalars along the morphisms in \eqref{eq: pullback of algebras} defines a functor 
\[ A\modc \to A_0\modc \underset{A_{01}\modc}{\times} A_1\modc.\]
\begin{proposition}[\cite{Lur09}, Proposition 16.2.1.1, Proposition 16.2.2.1]\label{lemma: modules over pullback}
  The functor 
  \[ A\modc \to A_0\modc \underset{A_{01}\modc}{\times} A_1\modc\]
  is fully faithful. 
  If $\tau^{\geq 0}(A_{0}) \to \tau^{\geq 0}(A_{01})$ is surjective in $\QCoh(X)^\heartsuit$, then the functor restricts to an equivalence 
  \[ A\modc^- \overset{\sim}{\to} A_0\modc^- \underset{A_{01}\modc^-}{\times} A_1\modc^-.\]
\end{proposition}

Suppose that $\tilde X$ is a square-zero extension of the prestack $X$ by $\mathcal{M} \in \QCoh(X)^{\leq 0}$. 
There is a pushout square 
\[ 
% https://q.uiver.app/#q=WzAsNCxbMSwwLCJYIl0sWzAsMSwiWCJdLFswLDAsIlhee1xcbWF0aGNhbHtNfVsxXX0iXSxbMSwxLCJcXHRpbGRlIFgiXSxbMiwxLCJ1IiwyXSxbMiwwLCJ2Il0sWzEsMywiXFx0aWxkZSB2Il0sWzAsMywiXFx0aWxkZSB1Il1d
\begin{tikzcd}[ampersand replacement=\&,cramped]
	{X^{\mathcal{M}[1]}} \& X \\
	X \& {\tilde X}
	\arrow["v", from=1-1, to=1-2]
	\arrow["u"', from=1-1, to=2-1]
	\arrow["{\tilde u}", from=1-2, to=2-2]
	\arrow["{\tilde v}", from=2-1, to=2-2]
\end{tikzcd}
\]
of schemes under $X$,
where $u$ and $v$ are sections of the canonical closed immersion $j: X \to X^{\mathcal{M}[1]}$.
and thus a fully faithful functor
\begin{equation}%\label{eq: sheaves on square zero extensions}
  \QCoh(\tilde X) \to \QCoh(X) \underset{u*^, \QCoh(X^{\mathcal{M}[1]}), v^*}{\times} \QCoh(X)
\end{equation}
which is an equivalence when $\QCoh$ is replaced by $\QCoh^-$.
The following Proposition is a generalization of [\cite{Lurie-SAG}, Proposition 19.2.2.2].
\begin{proposition} \label{prop: first-order-deformations-in-qcoh}
  Let \(E \in \QCoh(X)\). 
  \begin{enumerate}
    \item  There is a fully faithful functor 
    \[ \{E\} \underset{\QCoh(X),\tilde u^*}{\times} \QCoh(\tilde X) 
    \to \Isom(u^*E, v^*E) \times_{j^*, \Isom(E,E)} \{id_E\}
    \]
    Thus $E \times_{\QCoh(X)} \QCoh(\tilde X)$ is an $\infty$-groupoid. If $E \in \QCoh(X)^-$, then the functor is an equivalence.
    \item The $\infty$-groupoid $\Isom(u^*E, v^*E) \times_{j^*,\Isom(E,E)}\{\id_E\}$ is nonempty if and only if $u^*E \cong v^*E$. Choosing such an isomorphism gives an equivalence to 
    \[\Aut(u^*E,u^*E) \times_{\Aut(E)} \{\id_E\} \cong \Hom_{\QCoh(X)}(E, E \otimes \mathcal{M}[1])\]
    \item Suppose that $\alpha: \tilde X \to X$ is a section of the structure map $X\to \tilde X$, which in particular induces an isomorphism $u^* \simeq v^*$. Then the map 
    \[ 
      \{E\} \times_{\QCoh(X),u^*} \QCoh(\tilde X) \to \Hom_{\QCoh(X)}(E, E \otimes \mathcal{M}[1])
    \]
    induced by i.\ and ii.\
    sends $\tilde E \in \QCoh(\tilde X)$ to the morphism
    \[ E \to \cofib(\alpha_*\tilde E \to E)\]
    where $\alpha_*\tilde E \to E$ is the map induced by applying $\alpha_*$ to the unit $1 \to \tilde u_* \tilde u^*$.
    %\QCoh(X) \underset{u^*,\QCoh(X^{\mathcal{M}[1]}),v^*}{\times} \{E\}.
  \end{enumerate}
\end{proposition}
\begin{proof}
  \begin{enumerate}
  \item Consider the fully faithful functor 
  \[ \QCoh(\tilde X) \to \QCoh(X) \times_{v^*, \QCoh(X)^{\mathcal{M}[1]}, u^*} \QCoh(X)\]
  induced by pullback along $\tilde u$ and $\tilde v$.
  Taking the fiber product along $E$ in the first factor gives a fully faithful functor 
  \[ \{E\}\underset{\tilde v^*, \QCoh(X)}\times \QCoh(\tilde X) \to \{E\} \times_{u^*, \QCoh(X^{\mathcal{M}[1]}),v^*} \QCoh(X).
  \]
  The latter category is the category of objects $F \in \QCoh(X)$ equipped with an isomorphism $v^*F \to u^*E$. 
  Applying $j^*$ to this map shows that $F \simeq E$, and thus the map 
  \[ 
    \Isom(v^*E,u^*E) \times_{j^*,\Isom(E,E)} id_E \to \{E\}  \underset{u^*,\QCoh(X),v^*}{\times} \QCoh(X)
  \]
  sending $g: v^*E \simeq u^*E$ to $E$ equipped with $g: v^*E \to u^*E$ 
  is an equivalence.
  \item Suppose that $\Isom(v^*E,u^*E)$ is nonempty.
  Fixing $g: v^*E \to u^*E$ gives an equivalence 
  \[ \Isom(v^*E,u^*E) \to \Isom(u^*E,u^*E).\]
  Since $j$ is a square-zero immersion, any map $u^*E \to u^*E$ whose pullback along $j$ is $\id_E$ is an isomorphism. Thus the space we seek is 
  \begin{align*} 
    \Hom(u^*E,u^*E) \underset{j^*,\Hom(E,E)}{\times} \{\id_E\}
      &\cong \Hom(E, u_*u^*E) \times_{\Hom(E,E)}\{\id_E\}  \\
      &\cong \Hom(E, E \oplus E \otimes \mathcal{M}[1]) \times_{\Hom(E,E)} \{\id_E\}\\
      &= \Hom(E, E \otimes \mathcal{M}[1]).
  \end{align*}
    \item The pullback square 
    \[
% https://q.uiver.app/#q=WzAsNCxbMCwwLCJcXE9PX3tcXHRpbGRlIFh9Il0sWzAsMSwiXFx0aWxkZSB1XypcXE9PX1giXSxbMSwwLCJcXHRpbGRlIHZfKlxcT09fWCJdLFsxLDEsIihcXHRpbGRlIHZ1KV8qIChcXE9PX1ggXFxvcGx1cyBcXG1hdGhjYWx7TX1bMV0pIl0sWzAsMV0sWzAsMl0sWzEsM10sWzIsM11d
\begin{tikzcd}[ampersand replacement=\&,cramped]
	{\OO_{\tilde X}} \& {\tilde v_*\OO_X} \\
	{\tilde u_*\OO_X} \& {(\tilde vu)_* (\OO_X \oplus \mathcal{M}[1])}
	\arrow[from=1-1, to=1-2]
	\arrow[from=1-1, to=2-1]
	\arrow[from=1-2, to=2-2]
	\arrow[from=2-1, to=2-2]
\end{tikzcd}
    \]
    implies that for any $\tilde E \in \QCoh(\tilde X)$, there is a fiber sequence 
    \[ \tilde E \to \tilde u_*\tilde u^* \tilde E \oplus \tilde v_*\tilde v^*\tilde E \to (\tilde v u)_*(\tilde v u)^* \tilde E.\]
    Suppose we are given an isomorphism $\tilde u^*\tilde E \to E$.
    Applying $\alpha$ and using the equivalences $\alpha \tilde u \sim \id_X$
    gives a fiber sequence
    \[ \alpha_*\tilde E \to E \oplus \tilde v^* \tilde E \to u_*u^* \tilde v^* \tilde E.
    \]
    Hence, the cofiber of $\alpha_*\tilde E \to E$ is identified with the cofiber of $\tilde v^*\tilde E \to u_*u^*\tilde v^* \tilde E$.
    If $g: \tilde v^*\tilde E \to \tilde u^*\tilde E = E$ is the isomorphism induced by $\alpha \tilde u\sim \alpha \tilde v$, the map $E \to E \otimes \mathcal{M}[1]$ we seek is exactly the map $E \to \cofib(E \to u_*u^*E)$ induced by $g$. Hence it agrees with $E \to \cofib(\alpha_*\tilde E \to E)$.
  \end{enumerate}
\end{proof}

\subsection{The infinitesimal action of inertia}\label{subsection: d of inertia}

Fix a morphism of prestacks $f: X \to Y$. 
The relative loop space of $f$ is 
\[ \mathcal{L}X = X \underset{X \times_Y X}{\times} X\]
where the fiber product is along the diagonal maps $\Delta: X \to X \times_Y X$.

The relative loop space of $f$ is the beginning of the \v{C}ech nerve of $\Delta: X \to X \times_Y X$ and thus $q_1,q_2: \mathcal{L}X \rightrightarrows X$ has the structure of a groupoid.
We have a canonical map
\[c: X/\mathcal{L}X\to X \times X.\]
Picking a cosection $X \times_Y X \to X$ of $\Delta$ realizes $\mathcal{L}X$ as a group over $X$.\footnote{
Recall that a groupoid over $X$ is the same as an effective epimorphism $X \to Z$ \cite[Theorem 6.1.0.6, Proposition 6.2.2.7]{Lur09}, while a group over $X$ is an effective epimorphism $X \to Z$ with a cosection $Z \to X$.
}
There are two natural cosections of the diagonal, namely the two projections $q_1,q_2: X \times_Y X \to X$.
Let us make $\mathcal{L}X$ into a group via $q_1$.
Then we obtain a map $q_1 \circ c: X/\mathcal{L}X \to X$ which makes $\mathcal{L}X$ into a group over $X$.

The other projection $p_2: X \times_Y X \to X$ gives a morphism $q_2 \circ c: X/\mathcal{L}X \to X$ distinct from the structure map $q_1 \circ c$. This gives each $E \in \QCoh(X)$ a nontrivial structure of a representation of $\mathcal{L}X$, as noted in \cite[Chapter 8, 6.1.2]{GR17II} and \cite[Proposition 1.2.5]{KoP21}.
\begin{definition}
  For $E \in \QCoh(X)$, the \emph{canonical $\mathcal{L}X$-equivariant structure} on $E$ is $(q_2 \circ c)^*F \in \QCoh(X/\mathcal{L}X)$.
  Let $\can := (q_2 \circ c)^*: \QCoh(X) \to \QCoh(X/\mathcal{L}X)$.
\end{definition}

Now suppose that $G/Y$ is a group prestack admitting a cotangent complex. We now define the \emph{infinitesimal action} associated to a representation $E \in \QCoh(BG)$ of $G$.
If $q: Y \to BG$ is the quotient map,
the infintesimal action on $E$ is just the derivative of the action map of $G$ on $q^*E$ at the identity of $G$.
More formally, consider the beginning of the Cech nerve of $Y \to BG$:
\[
% https://q.uiver.app/#q=WzAsNCxbMCwxLCJZIl0sWzEsMSwiQkciXSxbMCwwLCJHIl0sWzEsMCwiWSJdLFswLDEsInEiLDJdLFsyLDAsInAiLDIseyJvZmZzZXQiOjF9XSxbMiwzLCJwIl0sWzMsMSwicSJdXQ==
\begin{tikzcd}
	G & Y \\
	Y & BG
	\arrow["p", from=1-1, to=1-2]
	\arrow["p"', shift right, from=1-1, to=2-1]
	\arrow["q", from=1-2, to=2-2]
	\arrow["q"', from=2-1, to=2-2]
\end{tikzcd}
\]
The homotopy making the square commute induces an isomorphism $act_G: p^*q^*E \overset{\sim}{\to}p^*q^*E$.
Suppose we are given a tangent vector to $G$ at the identity, that is, a diagram of the form 
\begin{equation}\label{eq: tangent vector at identity}
% https://q.uiver.app/#q=WzAsNCxbMSwwLCJHIl0sWzEsMSwiWSJdLFswLDAsIlxcU3BlYyhSIFxcb3BsdXMgTSkiXSxbMCwxLCJcXFNwZWMgUiJdLFszLDJdLFsxLDAsImUiXSxbMiwwLCJ4Il0sWzMsMSwieSJdXQ==
\begin{tikzcd}
	{\Spec(R \oplus M)} & G \\
	{\Spec R} & Y
	\arrow["x", from=1-1, to=1-2]
	\arrow[from=2-1, to=1-1]
	\arrow["y", from=2-1, to=2-2]
	\arrow["e", from=2-2, to=1-2]
\end{tikzcd}
\end{equation}
Taking the pullback of $act_G$ along $x$ gives a morphism $x^*act_G$ whose reduction to $\Spec R$ is the identity of $y^*q^*E$. The space of such morphisms is exactly 
\[\Hom_R( y^*q^*E, y^*q^*E \otimes M)\]
As such tangent vectors $x$ as in \eqref{eq: tangent vector at identity} are classified by $\Hom_R(e^*\ell_{G/Y}, M)$, we obtain a natural transformation 
\[ \Hom_R(e^*\ell_{G/Y},-) \to \Hom_R(y^*q^*E, y^*q^*E \otimes -)\]
of functors on $\Mod_R^{\leq 0}$.
This natural transformation is natural in $y$, and thus assembles to a natural transformation 
\[ \Hom_{\QCoh(Y)}(\ell_{G/Y}, -) \to \Hom_{\QCoh(Y)}(q^*E,q^*E \otimes -)\]
of functors on $\QCoh(Y)^{\leq 0}$.
By Lemma \ref{lemma: restrict excisive} and the Yoneda Lemma, this natural transformation is induced by a map 
\[\dact_G: q^*E \to q^*E \otimes e^*\ell_{G/Y}\]
unique up to contractible choices.

If $f:X \to Y$ admits a cotangent complex, then according to the diagram 
\[
% https://q.uiver.app/#q=WzAsNCxbMCwwLCJcXG1hdGhjYWx7TH1YIl0sWzAsMSwiWCJdLFsxLDAsIlgiXSxbMSwxLCJYIFxcdGltZXNfWSBYIl0sWzEsMywiXFxEZWx0YSJdLFsyLDMsIlxcRGVsdGEiXSxbMCwxLCJwXzEiXSxbMCwyLCJwXzIiXV0=
\begin{tikzcd}
	{\mathcal{L}X} & X \\
	X & {X \times_Y X}
	\arrow["{p_2}", from=1-1, to=1-2]
	\arrow["{p_1}", from=1-1, to=2-1]
	\arrow["\Delta", from=1-2, to=2-2]
	\arrow["\Delta", from=2-1, to=2-2]
\end{tikzcd}
\]
so does $\mathcal{L}X/X$, and 
\[ L\Omega^1_{\mathcal{L}X/X} \simeq p_2^* L\Omega^1_{\Delta} \simeq p_2^* L\Omega^1_{X/Y}[1],\]
so that the co-Lie complex of $\mathcal{L}X/X$ is identified with 
\[ \ell_{\mathcal{L}X/X} := e^*L\Omega^1_{\mathcal{L}X/X} \simeq L\Omega^1_{X/Y}[1].\] 
\begin{definition}
  For $f: X \to Y$ a morphism of prestacks admitting a cotangent complex and $E \in \QCoh(X)$, the \emph{Atiyah class} of $E/Y$ is the infinitesimal action of the canonical $\mathcal{L}X$-equivariant structure on $E$:
  \[ at_{E/Y} : E \to E \otimes \ell_{\mathcal{L}X/X} \simeq E \otimes L\Omega^1_{X/Y}[1].\]
\end{definition}
\begin{remark}
  In characteristic zero, Kondryev and Prikhodko also defined the Atiyah class of $E \in \QCoh(X)$ as the infinitesimal action of inertia on the canonical $\mathcal{L}X$-equivariant structure \cite[Definition 1.3.2]{KoP21}, except they used the correspondence between formal groups and Lie algebras in characteristic zero. 
\end{remark}
\begin{remark}
  We consider the Atiyah class only as a linear map, but since $\mathcal{L}X$ is a group, it should enhance to the structure of a coaction of the co-partition Lie algebra $\ell_{\mathcal{L}X/X}$.
\end{remark}
Let us spell out how to calculate the action map for the canonical $\mathcal{L}X$-equivariant structure on $E \in \QCoh(X)$.
The map $q_1 \circ c: X/\mathcal{L}X \to X$ gives a homotopy between $p_1$ and $p_2$ according to the diagram 
\begin{equation}\label{eq: infinitesimal-inertia-1}
  % https://q.uiver.app/#q=WzAsOCxbMCwwLCJcXG1hdGhjYWx7TH1YIl0sWzAsMSwiWCJdLFsxLDAsIlgiXSxbMSwxLCJYIFxcdGltZXNfWSBYIl0sWzMsMSwiWCJdLFszLDIsIlgiXSxbMiwyLCJYIl0sWzIsMSwiWCJdLFswLDEsInBfMSIsMl0sWzAsMiwicF8yIl0sWzMsNSwicV8xIiwxLHsibGFiZWxfcG9zaXRpb24iOjIwfV0sWzcsNl0sWzYsNV0sWzcsNF0sWzQsNV0sWzEsNl0sWzIsNF0sWzIsM10sWzEsM10sWzAsN11d
\begin{tikzcd}
	{\mathcal{L}X} & X && \\
	X & {X \times_Y X} & X & X \\
	&& X & X
	\arrow["{p_2}", from=1-1, to=1-2]
	\arrow["{p_1}"', from=1-1, to=2-1]
	\arrow[from=1-1, to=2-3]
	\arrow[from=1-2, to=2-2]
	\arrow[from=1-2, to=2-4]
	\arrow[from=2-1, to=2-2]
	\arrow[from=2-1, to=3-3]
	\arrow["{q_1}"{description, pos=0.2}, from=2-2, to=3-4]
	\arrow[from=2-3, to=2-4]
	\arrow[from=2-3, to=3-3]
	\arrow[from=2-4, to=3-4]
	\arrow[from=3-3, to=3-4]
\end{tikzcd}.
\end{equation}
On the other hand, the action map for $\can(E)$ is the isomorphism $p_1^*q^*E \to p_2^*q^*E$ induced by the diagram 
\begin{equation}\label{eq: infinitesimal-inertia-2}
  % https://q.uiver.app/#q=WzAsOCxbMCwwLCJcXG1hdGhjYWx7TH1YIl0sWzAsMSwiWCJdLFsxLDAsIlgiXSxbMSwxLCJYIFxcdGltZXNfWSBYIl0sWzMsMSwiWCJdLFszLDIsIlgiXSxbMiwyLCJYIl0sWzIsMSwiWCJdLFswLDEsInBfMSIsMl0sWzAsMiwicF8yIl0sWzMsNSwicV8yIiwxLHsibGFiZWxfcG9zaXRpb24iOjIwfV0sWzcsNl0sWzYsNV0sWzcsNF0sWzQsNV0sWzEsNl0sWzIsNF0sWzIsM10sWzEsM10sWzAsN11d
\begin{tikzcd}
	{\mathcal{L}X} & X && \\
	X & {X \times_Y X} & X & X \\
	&& X & X
	\arrow["{p_2}", from=1-1, to=1-2]
	\arrow["{p_1}"', from=1-1, to=2-1]
	\arrow[from=1-1, to=2-3]
	\arrow[from=1-2, to=2-2]
	\arrow[from=1-2, to=2-4]
	\arrow[from=2-1, to=2-2]
	\arrow[from=2-1, to=3-3]
	\arrow["{q_2}"{description, pos=0.2}, from=2-2, to=3-4]
	\arrow[from=2-3, to=2-4]
	\arrow[from=2-3, to=3-3]
	\arrow[from=2-4, to=3-4]
	\arrow[from=3-3, to=3-4]
\end{tikzcd}
.
\end{equation}
Hence, the action map of $\can(E)$ is the composite 
\[ p_1^*E \to p_2^*E \leftarrow p_1^*E\]
where the two maps are induced by $p_1$ and $p_2$ according to the diagrams \eqref{eq: infinitesimal-inertia-1} and \eqref{eq: infinitesimal-inertia-2}.

\begin{example}\label{ex: classical atiyah class}
  Suppose that $f: X \to Y$ is a smooth morphism of smooth schemes.
  Let $I$ be the ideal of the diagonal $\Delta: X \to X \times_Y X$, so that $L\Omega^1_{\Delta} \simeq I/I^2[1]$.
  Let $X^{(2)} = \Specu(\OO_{X\times_Y X}/I^2)$ be the first-order neighborhood of the diagonal
  and $i: X^{(2)} \to X \times_Y X$ the inclusion. 
  We have a commutative diagram 
  \[% https://q.uiver.app/#q=WzAsOCxbMSwxLCJYXnsoMil9Il0sWzAsMCwiWF57SS9JXjJbMV19Il0sWzAsMSwiWCJdLFsxLDAsIlgiXSxbMiwyLCJYIl0sWzIsMSwiXFxtYXRoY2Fse0x9WCJdLFszLDEsIlgiXSxbMywyLCJYIFxcdGltZXNfWVgiXSxbMSwyXSxbMSwzXSxbMiwwXSxbMywwXSxbMiw0XSxbMyw2XSxbMSw1XSxbNSw0XSxbNSw2XSxbNiw3XSxbNCw3XSxbMCw3XV0=
\begin{tikzcd}
	{X^{I/I^2[1]}} & X && \\
	X & {X^{(2)}} & {\mathcal{L}X} & X \\
	&& X & {X \times_YX}
	\arrow[from=1-1, to=1-2]
	\arrow[from=1-1, to=2-1]
	\arrow[from=1-1, to=2-3]
	\arrow[from=1-2, to=2-2]
	\arrow[from=1-2, to=2-4]
	\arrow[from=2-1, to=2-2]
	\arrow[from=2-1, to=3-3]
	\arrow[from=2-2, to=3-4]
	\arrow[from=2-3, to=2-4]
	\arrow[from=2-3, to=3-3]
	\arrow[from=2-4, to=3-4]
	\arrow[from=3-3, to=3-4]
\end{tikzcd}\]
where the map $X^{I/I^2[1]} \to \mathcal{L}X$ induces an isomorphism $\ell_{\mathcal{L}X/X} \overset{\sim}{\to} I/I^2[1]$.
The square-zero extension $X \to X^{(2)}$ is split by
\[q_1' = q_1 \circ i: X^{(2)} \to X \times_Y X \to X.\]
Let $q_2' = q_2 \circ i$. The infinitesimal action of $\mathcal{L}X$ on $E$ is the image of $q_2'^*E$ under the map
\[
  \QCoh(X^{(2)}) \times_{\QCoh(X)} \{E\} \to \Hom(E, E \otimes I/I^2[1])
\]
of Proposition \ref{prop: first-order-deformations-in-qcoh}.
By Proposition \ref{prop: first-order-deformations-in-qcoh}.iii, this map is exactly the boundary morphism of the exact triangle 
\[ I/I^2 \to (q_1')_* (q_2')^* E \to E \to^{+1},\]
which is exactly the classical definition of the Atiyah class (see e.g.\ \cite[§10.1.5]{HL97}).
\end{example}

\begin{remark}
  In principle, the argument in Example \ref{ex: classical atiyah class} can be extended to an arbitrary morphism of prestacks admitting a cotangent complex using Gaitsgory and Rozenblyum's construction of the derived first-order neighborhood \cite[Chapter 9, §5.1]{GR17II}. Gaitsgory and Rozenblyum restrict to laft prestacks in order to use ind-coherent sheaves, but this is not essential for constructing the Atiyah class, although we do not pursue this idea further. {Moore25}
\end{remark}

\subsection{The Atiyah class and connections}
\label{subsection: atiyah and comparison to Lurie}

In \cite[§19.2.2]{Lurie-SAG}, Lurie gives a different treatment of the Atiyah class based on the following idea:
Given $E \in \QCoh(X)$ and a derivation 
\[ \delta \in \Der_{X/Y}(\mathcal{M}) \]
associated to $s_\delta: X^\mathcal {M} \to X$,
the pullback $s_\delta^*E$ is a square-zero deformation of $E$, thus by Proposition \ref{prop: first-order-deformations-in-qcoh} classified by an object in $\Maps(E, E \otimes \mathcal{M}[1])$.
Thus we have a natural transformation 
\[ \Der_{X/Y} \to \Maps(E, E \otimes -[1]).\]
Lurie defines the Atiyah class of $E$ to be the morphism $E \to E \otimes L\Omega^1_{X/Y}[1]$ representing this transformation.
\begin{remark}
  This transformation should be interpreted as the obstruction for $E$ to admit a connection.
\end{remark}
\begin{remark}
  Lurie defines this transformation only for $E \in \QCoh(X)^-$, but this restriction is not necessary, since the functors in Proposition \ref{prop: first-order-deformations-in-qcoh} are defined for all $E$.
\end{remark}
\begin{proposition}
  Suppose that $f: X \to Y$ is a morphism of prestacks admitting a cotangent complex. 
  Then the natural transformation 
  \[ \Der_{X/Y} \to \Maps(E, -[1] \otimes E)\]
  induced by the Atiyah class is equivalent to the transformation sending 
  \[ 
    \left(\delta \in \Der_{X/Y}(\mathcal{M})\right) \mapsto  [s_\delta^*E] \in \Maps(E, \mathcal{M}[1]\otimes E).
  \]
\end{proposition}
The argument here is similar to S.\ Moore's discussion of the Atiyah class in \cite[§3.1-3]{Moore25}, although Moore does not use the language of infinitesimal inertia.
\begin{proof}
  Suppose that $x: \Spec R \to X$ and $\delta: x^* L\Omega^1_{X/Y} \to M$ is a point of $\Der_{x/Y}(M)$.
  Let $s_\delta: \Spec(R \oplus M) \to X$ be the corresponding map.
  Then $s_\delta$ fits into a commutative diagram
  \begin{equation}\label{eq: comparison to Lurie diagram}
  % https://q.uiver.app/#q=WzAsNyxbMCwxLCJcXFNwZWMgUiJdLFsxLDAsIlxcU3BlYyBSIl0sWzAsMCwiXFxTcGVjKFIgXFxvcGx1cyBNWzFdKSJdLFszLDIsIlggXFx1bmRlcnNldHtZfXtcXHRpbWVzfVgiXSxbMiwyLCJYIl0sWzMsMSwiWCJdLFsxLDEsIlxcU3BlYyhSIFxcb3BsdXMgTSkiXSxbMiwwXSxbNCwzXSxbNSwzXSxbMSw2XSxbMCw2XSxbMCw0XSxbNiwzLCIoc18wLHNfXFxkZWx0YSkiLDFdLFsyLDFdLFsxLDVdXQ==
\begin{tikzcd}
	{\Spec(R \oplus M[1])} & {\Spec R} && \\
	{\Spec R} & {\Spec(R \oplus M)} && X \\
	&& X & {X \underset{Y}{\times}X}
	\arrow[from=1-1, to=1-2]
	\arrow[from=1-1, to=2-1]
	\arrow[from=1-2, to=2-2]
	\arrow[from=1-2, to=2-4]
	\arrow[from=2-1, to=2-2]
	\arrow[from=2-1, to=3-3]
	\arrow["{(s_0,s_\delta)}"{description}, from=2-2, to=3-4]
	\arrow[from=2-4, to=3-4]
	\arrow[from=3-3, to=3-4]
\end{tikzcd}.
\end{equation}
Now $q_2^*E$ pulls back along $(s_0,s_\delta)$ to $s_\delta^*E \in \Mod_{R \oplus M}$.
Thus 
$[s_\delta^*E] \in \Hom(x^*E, x^*E \otimes M[1])$ is obtained by comparing the two pullbacks of $q_2^*E$ to $\Spec(R\oplus M[1])$.

The diagram \eqref{eq: comparison to Lurie diagram}
induces a map $\alpha: \Spec(R \oplus M[1]) \to \mathcal{L}X$. The map $\alpha$ is classified by 
\[\delta[1]: x^*\ell_{\mathcal{L}X} = x^*L\Omega^1_{X/Y}[1] \to \mathcal{M}[1].\]
The Atiyah class $at_E: E \to E \otimes \ell_{\mathcal{L}X/X}[1]$ 
is obtained by comparing the two isomorphisms of the pullbacks of $E$ to $\mathcal{L}X$.
After pulling back to $\Spec(R \oplus M[1])$,
we see that 
\[\delta[1] \circ at_E: x^*E \to x^*E \otimes M[1]\]
is the difference of the two maps $x^*E \to x^*E \otimes M[1]$ obtained by pulling back along $s_\delta$ and $s_0$. But $s_0$ induces the zero map $\ell_{\mathcal{L}X/X} \to M[1]$.
\end{proof}

\section{Loops on stacks}\label{appendix: loops on stacks}

Let \(X\) be a stack. In \cite{BN12}, Ben-Zvi and Nadler explain that since the loop space functor \(X \mapsto \mathcal{L}X\) does not satisfy étale descent, the loop space is ``nonlocal,'' unlike the shifted tangent bundle \(T[-1]X\). For example, if \(G\) is a finite discrete group, \(\mathcal{L}BG = G/G\) is the stack of conjugacy classes and has many components, but \(T[-1]BG = BG\) has one component.

To address the problem of non-locality, various authors have introduced ``smaller'' loop spaces.
Ben-Zvi and Nadler introduced two smaller loop spaces, the formal and unipotent loops.
Furthermore, in \cite{ABM21}, Antieau, Bhatt, and Mathew deal with the sheafification of Hochschild homology in syntomic topology.
In this appendix, we compare these constructions. In Theorem \ref{theorem: abm is formal}, it is shown that Antieau-Bhatt-Mathew's sheafified Hochschild homology agrees with functions on the formal loop space.

\subsection{Formal completions of morphisms of prestacks}

Gaitsgory and Rozenblyum introduced the formal completion of a morphism of prestacks. \cite[Chapter 2, 3.1.3 (iii)]{GR17II}

If \(R\) is an animated ring, then let \(\red{R}\) denote the quotient of the discrete ring \(\pi_0(R)\) by its nilradical, considered as an animated ring \cite[Chapter 4, 6.1.4]{GR17I}.

\begin{definition}
  If \(X\) is a prestack, we define the \emph{infinitesimal prestack} of \(X\) to be
  \[ X_{\mathrm{inf}}(R) = X(\red{R}).\]
\end{definition}
In characteristic zero, the infinitesimal prestack is Gaitsgory-Rozenblyum's de Rham prestack, but away from characteristic zero, the infinitesimal prestack does not compute de Rham cohomology.

\begin{definition}[\cite{GR17II}, Chapter 2, 3.1.3 (iii)]\label{definition: formal completion}
  The \emph{formal completion} of a morphism \(f: X \to Y\) of prestacks is the fiber product
  \[
    Y_X^{\wedge} = Y \times_{Y_{\mathrm{inf}}} X_{\mathrm{inf}}.
  \]
\end{definition}

\begin{example}\label{ex: classical completion}
  Let \(A\) be a discrete ring, \(Y = \Spec A\), \(I \subseteq A\) an ideal, and \(X = \Spec A/I\).
  For any animated ring \(R\), since \(A\) is discrete,
  \begin{align*}
    Y_X^\wedge(\Spec R) &= \Maps(A,R) \times_{\Maps(A, \red{R})} \Maps(A/I, \red{R}).
  \end{align*}
  That is, \(Y_X^\wedge(\Spec R)\) is the space of ring homomorphisms \(A \to R\) such that the image of every element in \(I\) is nilpotent.
  If the ideal \(I\) is finitely generated, then the image of \(I\) is elementwise nilpotent if and only if there exists some \(n \geq 0\) such that \(I^n\) maps to zero. Thus, if \(I\) is finitely generated, the completion \(Y_X^\wedge\) in the sense of Definition \ref{definition: formal completion} is equivalent to \(\varinjlim_{n} \Spec(A/I^n)\) in prestacks, that is, the usual completion of \(\Spec A\) along \(I\) as a formal scheme.
\end{example}

As Example \ref{ex: classical completion} shows, some finiteness conditions are needed for the formal completion of Definition \ref{definition: formal completion} to behave as expected.

\begin{definition}[\cite{GR17I}, Chapter 2, 1.7.1]
  Let \(A\) be a discrete commutative ring.
  A simplicial \(A\)-algebra \(R\) is of \emph{finite type} if \(\pi_0(R)\) is an \(A\)-algebra of finite type and \(\bigoplus_{i=1}^\infty \pi_i(R)\) is a finitely generated \(\pi_0(R)\)-module.
\end{definition}
Let \({}^{\leq n}\CAlg_{A}\) be the category of \(n\)-coconnective simplicial \(A\)-algebras and \({}^{\leq n}\CAlg_{A,ft}\) be the full subcategory of finite-type \(A\)-algebras.

\begin{definition}[\cite{GR17I}, Chapter 2, 1.7.2]
  A prestack \(X\) is \emph{locally almost of finite type} over \(A\) if:
  \begin{itemize}
    \item \(X\) is convergent: \(X(R) = \lim_n X(\tau_{\leq n}R)\) for all \(R\);
    \item For all \(n \geq 0\), the restriction of \(X\) to \({}^{\leq n}\CAlg_A\) is left Kan extended from \({}^{\leq n}\CAlg_{A,ft}\).
  \end{itemize}
\end{definition}

\begin{lemma}\label{lemma: completion and cotangent complex}
  Let \(A\) be a Noetherian commutative ring and let \(X,Y,Z\) be prestacks locally almost of finite type over \(A\).
  Suppose
  % https://q.uiver.app/#q=WzAsMyxbMSwwLCJYIl0sWzAsMSwiWSJdLFsyLDEsIloiXSxbMCwxXSxbMSwyXSxbMCwyXV0=
  \[
    \begin{tikzcd}
      & X \\
      Y && Z
      \arrow[from=1-2, to=2-1]
      \arrow[from=1-2, to=2-3]
      \arrow[from=2-1, to=2-3]
    \end{tikzcd}
  \]
  is a commutative diagram of prestacks such that each morphism admits a cotangent complex and such that \(L\Omega^1_{X/Y} \to L\Omega^1_{X/Z}\) is an equivalence. Then
  \[ Y_X^\wedge \to Z_X^\wedge \]
  is an equivalence.
\end{lemma}
\begin{proof}
  Since \(X,Y,Z\) are locally almost of finite type, it suffices to show that \(Y_X^\wedge(R) \to Z_X^\wedge(R)\) is an equivalence if \(R\) is \(n\)-coconnective and of finite type.
  Since \(A\) is Noetherian and \(R\) is of finite type, \(\pi_0(R)\) is Noetherian, so its nilradical is finitely generated and thus nilpotent. Thus \(R\) is an iterated square-zero extension of \(\red{\pi_0(R)}\). Since \(L\Omega^1_{X/Y} \to L\Omega^1_{X/Z}\) is an equivalence, whenever \(\tilde S\) is a square-zero extension of \(S\), the map
  \[ X(S) \times_{Y(S)} Y(\tilde S) \to X(S) \times_{Z(S)} Z(\tilde S) \]
  is an equivalence. Thus by induction,
  \[ X(\red{\pi_0(R)}) \times_{Y(\red{\pi_0(R)})} Y(R) \to X(\red{\pi_0(R)}) \times_{Z(\red{\pi_0(R)})} Z(R)\]
  is an equivalence.
\end{proof}

\subsection{Formal loops and unipotent loops}

\begin{definition}[\cite{BN12}, Definition 1.22]
  The \emph{formal loop space} \(\hat{\mathcal{L}}X\) of a stack \(X\) is the formal completion of \(\mathcal{L}X\) at the constant loops \(X \to \mathcal{L}X\).
\end{definition}

Let \(\aff(-) = \Spect(R\Gamma(-,\OO))\) denote the affinization functor (see §\ref{subsection: affine stacks}).

\begin{definition}[\cite{BN12}, Definition 1.24]
  The \emph{unipotent loop space} of a stack \(X\) is
  \[ \mathcal{L}^uX = \Mapsu(\aff(\sone),X) = \Mapsu(B\Ghat_m^\vee,X).\]
\end{definition}

By definition, the unipotent loop space \(\mathcal{L}^u X\) is the underlying stack of Moulinos-Robalo-Toën's filtered mapping stack \(\Mapsu(\filS,X)\).

Let \(\hat{\mathcal{L}^u}X\) be the formal completion of \(\mathcal{L}^uX\) along the constant maps.
The map \(\mathcal{L}^u X \to \mathcal{L}X\) induced by affinization of \(\sone\) induces a map
\begin{equation}\label{eq: compare formal unipotent and formal loops}
  \hat{\mathcal{L}}^uX \to \hat{\mathcal{L}}X
\end{equation}

\begin{theorem}[c.f. \cite{BN12}, §6.3]\label{formal and unipotent}
  Let \(A\) be a Noetherian commutative ring and \(X/A\) be a prestack locally almost of finite type admitting a cotangent complex.
  Then the map \(\hat{\mathcal{L}}^uX \to \hat{\mathcal{L}}X\) is an equivalence.
\end{theorem}
\begin{proof}
  Consider the triangle
  \[ X \to {\mathcal{L}}^uX \to \mathcal{L}X.\]
  Since \(X\) is locally almost of finite type and admits a cotangent complex over \(A\), the same holds for \(\mathcal{L}X\).
  Since \(\aff(\sone)\) is the affinization of a finite CW complex, it is represented by a finite cosimplicial algebra, and thus \(\Mapsu(\aff(\sone),X)\) is also locally almost of finite type and admits a cotangent complex.
  By Lemma \ref{lemma: completion and cotangent complex}, it is sufficient to check that \(L\Omega^1_{X/\mathcal{L}^uX} \to L\Omega^1_{X/\mathcal{L}X}\) is an equivalence.

  It is equivalent to check that if \(f:\Spec R \to X\), then \(f^*L\Omega^1_{\mathcal{L}X /A} \to f^*L\Omega^1_{\mathcal{L}^uX/A}\) is an equivalence.
  But now if \(g: \Spec R \to \Mapsu(Y,Z)\), then
  \[
    g^* L\Omega^1_{\Maps(Y,Z)/A} = R\Gamma(Y, g^*L\Omega^1_{Z/A}).
  \]
  We consider the above formula where \(g\) is the constant loop at \(f\). Since \(\sone \to \aff(\sone)\) is the affinization morphism, the map
  \[
    R\Gamma(\aff(\sone), f^*L\Omega^1_{X/A}) \to R\Gamma(\sone, f^*L\Omega^1_{X/A})
  \]
  is an equivalence.
\end{proof}

Theorem \ref{formal and unipotent} shows that the map \(\hat{\mathcal{L}}X \to \mathcal{L}X\) factors through the unipotent loop space as \(\hat{\mathcal{L}}X \to \mathcal{L}^uX \to \mathcal{L}X\).

\subsection{Antieau-Bhatt-Mathew's Hochschild homology of syntomic stacks}

In \cite{ABM21}, Antieau, Bhatt, and Mathew define a variant of Hochschild homology \(\HH^{ABM}(\cX/k)\) of a syntomic stack \(\cX/k\) by descent in the syntomic topology, and stated that it computes a completion of Hochschild homology \cite[Warning 2.9]{ABM21}. For example, \(\HH^{ABM}_0(B\mathbb{G}_m/k) = k\series{t-1}\).

\begin{theorem}\label{theorem: abm is formal}
  Let \(A\) be a Noetherian commutative ring.
  Let \(X/A\) be a syntomic stack. Then
  \(\HH^{ABM}(X/A) \simeq R\Gamma(\hat{\mathcal{L}}\cX, \mathcal O_{\hat{\mathcal{L}}\cX})\).
\end{theorem}
\begin{proof}
  By definition, \(\HH^{ABM}(-/A)\) is the sheafification of \(X \mapsto R\Gamma(\mathcal{L}X,\OO_{\mathcal{L}X})\) on affine schemes \(X/A\). If \(X\) is affine, then the loop space and formal loop space on \(X\) agree, so it suffices to show that \(X \mapsto R\Gamma(\hat{\mathcal{L}}X, \OO_{\hat{\mathcal{L}}X})\) is a syntomic sheaf. 
  In fact, we show \(X \mapsto R\Gamma(\hat{\mathcal{L}}X, \OO_{\hat{\mathcal{L}}X})\) satisfies fpqc descent.
  Since \(\hat{\mathcal{L}}X\) is the completion of the loop space at the constant loops, \(R\Gamma(\hat{\mathcal{L}}X,\OO_{\hat{\mathcal{L}}X})\) is complete with respect to the HKR filtration. The associated graded of the HKR filtration is the symmetric algebra on \(L\Omega^1[1]\), and \(\Sym^i(L\Omega^1[1])\) satisfies fpqc descent by \cite[Theorem 3.1]{BMS19}.
  Thus \(X \mapsto R\Gamma(\hat{\mathcal{L}}X, \OO_{\hat{\mathcal{L}}X})\) satisfies fpqc descent. 
\end{proof}

%\section{References}
\printbibliography

\end{document}